\documentclass[a4paper,twoside,11pt,reqno]{amsart}
\usepackage[margin=1 in]{geometry}                
\usepackage[colorlinks,linkcolor=blue,citecolor=blue,urlcolor=blue]{hyperref}
\usepackage[english]{babel}
\usepackage{amsthm,amsfonts,amssymb,mathrsfs,amsmath}
\usepackage{slashbox}
\usepackage[shortlabels]{enumitem}
\usepackage[dvipsnames]{xcolor}
\usepackage[normalem]{ulem}

\usepackage{mathtools}
\mathtoolsset{showonlyrefs,showmanualtags}
\usepackage{accents}
\newlength{\dhatheight}
\newcommand{\doublehat}[1]{%
  \settoheight{\dhatheight}{\ensuremath{\hat{#1}}}%
  \addtolength{\dhatheight}{-0.25ex}%
  \widehat{\vphantom{\rule{1pt}{\dhatheight}}%
    \smash{\widetilde{#1}}}}

\renewcommand{\doublehat}[1]{\widetilde{#1}}

\usepackage{todonotes}  

\newtheorem{theorem}{Theorem}[section]
\newtheorem{corollary}[theorem]{Corollary}
\newtheorem{lemma}[theorem]{Lemma}
\newtheorem{proposition}[theorem]{Proposition}

\theoremstyle{definition}
\newtheorem{definition}[theorem]{Definition}

\theoremstyle{remark}
\newtheorem{remark}[theorem]{Remark}
\newtheorem{example}[theorem]{Example}

\numberwithin{equation}{section}

\renewcommand{\epsilon}{\varepsilon}  

\newcommand{\F}{\mathcal{F}}
\newcommand{\fT}{\mathfrak{T}}   

\newcommand{\Q}{\mathbb{Q}}
\newcommand{\R}{\mathbb{R}}
\newcommand{\Z}{\mathbb{Z}}

\newcommand{\C}{\mathbb{C}}
\newcommand{\FF}{{\mathbb{F}}} 
\newcommand{\Sph}{\mathbb{S}} 

\DeclareMathOperator{\Ran}{Ran}  
\DeclareMathOperator{\Ker}{Ker}  
\DeclareMathOperator{\Span}{span}
\DeclareMathOperator{\rank}{rank}
\DeclareMathOperator{\Hom}{Hom}  
\DeclareMathOperator{\Tor}{Tor}  

\newcommand{\esp}{\mathbf{E}} 
\newcommand{\cH}{{\mathcal H}} 

\newcommand{\Gr}{\mathrm{Gr}} 
\newcommand{\Taut}{\mathrm{Taut}} 
\newcommand{\B}{\mathcal{B}} 

\newcommand{\cT}{\mathcal{T}} 
\newcommand{\Sym}{\mathrm{Sym}} 
\DeclareMathOperator{\spec}{spec}

\newcommand{\tM}{\widetilde{M}}
\newcommand{\tx}{\widetilde{x}}

\newcommand{\term}[1]{\textbf{#1}}

\title[Morse inequalities for ordered eigenvalues]{Morse inequalities
  for ordered eigenvalues of generic self-adjoint families}
\date{\today}

\thanks{}

\author{Gregory Berkolaiko}
\address{Gregory Berkolaiko\\
	Department of Mathematics\\
	Texas A\&M University\\
	College Station\\
	Texas \ 77843\\
	USA}
\email{berko@math.tamu.edu}
\urladdr{\url{http://www.math.tamu.edu/~berko}}

\author{Igor Zelenko}
\address{Igor Zelenko\\
	Department of Mathematics\\
	Texas A\&M University\\
	College Station\\
	Texas \ 77843\\
	USA}
\email{zelenko@math.tamu.edu}
\urladdr{\url{http://www.math.tamu.edu/~zelenko}}

\DeclareMathOperator{\codim}{codim}
\DeclareMathOperator{\Tr}{Tr}

\begin{document}


\begin{abstract}
  In many applied problems one seeks to identify and count the
  critical points of a particular eigenvalue of a smooth parametric
  family of self-adjoint matrices, with the parameter space often being
  known and simple, such as a torus.  Among particular settings where
  such a question arises are the Floquet--Bloch decomposition of
  periodic Schr\"odinger operators, topology of potential energy
  surfaces in quantum chemistry, spectral optimization problems such
  as minimal spectral partitions of manifolds, as well as nodal
  statistics of graph eigenfunctions.  In contrast to the
  classical Morse theory dealing with smooth functions, the
  eigenvalues of families of self-adjoint matrices are not smooth at
  the points corresponding to repeated eigenvalues (called, depending
  on the application and on the dimension of the parameter space,
  the diabolical/Dirac/Weyl points or the conical intersections).

  This work develops a procedure for associating a Morse polynomial to
  a point of eigenvalue multiplicity; it utilizes the assumptions of
  smoothness and self-adjointness of the family to provide concrete
  answers.  In particular, we define the notions of non-degenerate
  topologically critical point and generalized Morse family, establish
  that generalized Morse families are generic in an appropriate sense,
  establish a differential first-order conditions for criticality, as
  well as compute the local contribution of a topologically critical
  point to the Morse polynomial.  Remarkably, the non-smooth
  contribution to the Morse polynomial turns out to
  depend only on the size of the eigenvalue multiplicity and the
  relative position of the eigenvalue of interest and not on the
  particulars of the operator family; it is expressed in terms of the
  homologies of Grassmannians.
\end{abstract}

\maketitle

\section{Introduction}

Let $\mathrm{Sym}_n(\R)$ and $\mathrm{Sym}_n(\C)$ denote the spaces of
$n\times n$ real symmetric (correspondingly, complex Hermitian)
matrices.  When referring to both spaces at once, we will use the term
``self-adjoint matrices'' and use the notation $\mathrm{Sym}_n$.  The
eigenvalues $\{\widehat\lambda_i(A)\}_{i=1}^n$ of a matrix
$A\in \mathrm{Sym}_n$ are real and will be numbered in the increasing
order,
\begin{equation}
\label{spectrum}
\widehat\lambda_1(A) \leq \widehat\lambda_2(A) \leq \cdots \leq \widehat\lambda_n(A).
\end{equation}
Further, let $M$ be a smooth (i.e.\ $C^\infty$) compact $d$-dimensional
manifold. A smooth $d$-parametric family of self-adjoint matrices (on $M$)
is a smooth map $\F: M \to \mathrm{Sym}_n$.

The aim of this paper is to develop the Morse
theory for the $k$-th ordered eigenvalue
\begin{equation} 
\label{branch} 	
\lambda_k:=\widehat \lambda_k\circ \F
\end{equation}
viewed as a function on $M$.  This question is motivated by numerous
problems in mathematical physics.  The boundaries between isolating
and conducting regimes in a periodic (crystalline) structure are
determined by the extrema of eigenvalues of an operator\footnote{The
  particulars of the operator depend on what is being conducted:
  electrons, light, sound, etc.} family defined on a $d$-dimensional
torus $M$ (for an introduction to the mathematics of this subject, see
\cite{Kuc_bams16}).  Other critical points of the eigenvalues give
rise to special physically observable features of the density of
states, the van Hove singularities \cite{vHo_prl53}.  Classifying all
critical points of an eigenvalue (also on a torus) by their degree is
used to study oscillation of eigenfunctions via the nodal--magnetic
theorem \cite{Ber_apde13,Col_apde13,AloGor_jst23,AloGor_prep24}.  More
broadly, the area of eigenvalue optimization encompasses questions
from understanding the charge distribution in an atomic nucleus
\cite{EstLewSer_plms21}, configuration of atoms in a polyatomic
molecule \cite{ConicalIntersections_book,Mat_cr21}, to shape
optimization \cite{Henrot_book06,Henrot_book17} and optimal partition
of domains and networks \cite{HelHof_jems13,BanBerRazSmi_cmp12}.  The
dimension of the manifold $M$ in these applications can be very high
or even infinite.

Morse theory is a natural tool for connecting statistics of the
critical points with the topology of the underlying manifold.
However, the classical Morse theory is formulated for functions that
are sufficiently smooth, whereas the function $\lambda_k$ is
generically non-smooth at the points where $\lambda_k(x)$ is a
repeated eigenvalue of the matrix $\F(x)$.  And it is these
points of non-smoothness that play an outsized role in the
applications \cite{CasGraphen_rmp09,ConicalIntersections_book}.

By Bronstein's theorem \cite{Bro_smz79}, each $\lambda_k$ is
Lipschitz.  Furthermore, by classical perturbation theory \cite{Kato}, the
function $\lambda_k$ is smooth along a submanifold $N \subset M$ if
the multiplicity of $\lambda_k(x)$ is constant on $N$; the latter
property induces a stratification of $M$.  There exist
generalizations of Morse theory to Lipschitz functions \cite{APS97},
continuous functions \cite[\S45]{FF89}, as well as to stratified
spaces \cite{GM88}.  These generalizations will provide the
general foundation for our work, but the principal thrust of this
paper is to leverage the properties of $\Sym_n$ and to get explicit
--- and beautiful --- answers for the Morse data in terms of the local
behavior of $\F$ at a discrete set of points we will identify as
``critical''.  One of the surprising findings is that the Morse data
attributable to the non-smooth directions at a critical point does not
depend on the particulars of the family $\F$.

To set the stage for our results we now review informally the main
ideas of Morse theory, which links the topological invariants of the
manifold $M$ to the number and the indices of the critical points of a
function $\phi$ on $M$.

In more detail, if $\phi$ is smooth, a point $x\in M$ is called a
\term{critical point} if the differential of $\phi$ vanishes at $x$.
The \term{Hessian} (second differential) of $\phi$ at $x$ is a
quadratic form on the tangent space $T_x M$.  In local coordinates it
is represented by the matrix of second derivatives, the \term{Hessian
  matrix}.  The \term{Morse index} $\mu(x)$ is defined as the negative
index of this quadratic form or, equivalently, the number of negative
eigenvalues of the Hessian matrix.  It is assumed that the second
differential at every critical point of $\phi$ is non-singular, i.e.\
the Hessian matrix has no zero eigenvalues; such critical points are
called \term{non-degenerate}.  Non-degenerate critical points are
isolated and therefore there are only finitely many of them on $M$. A
smooth function $\phi$ is called a \term{Morse function} if all its
critical points are non-degenerate.

The main result of the classical Morse theory quantifies the change in
the topology of the level curves of $\phi$ around a critical point.
To be precise, for a point $x\in M$ and its neighborhood $U$
(which we will always assume to be homeomorphic to a ball)
define the \term{local sublevel sets},
\begin{equation}
  \label{eq:local_sublevel_set_def}
  U_x^{-\epsilon}(\phi)
  := \{y\in U \colon \phi(y) \leq \phi(x)-\epsilon\}
  \qquad \text{and} \qquad
  U_x^{+\epsilon}(\phi)
  := \{y\in U \colon \phi(y) \leq \phi(x)+\epsilon\}.  
\end{equation}
If $x$ is a non-degenerate critical point of index $\mu=\mu(x)$, then,
for a sufficiently small neighborhood $U$ of $x$ and sufficiently
small $\epsilon>0$, the quotient space
$U_x^{+\epsilon}(\phi)/U_x^{-\epsilon}(\phi)$ is homotopy
equivalent to the $\mu$-dimensional sphere $\Sph^\mu$.  The global
consequences of this are the \term{Morse inequalities}: given a Morse
function $\phi$, denote by $c_q$, $q=0,\ldots,d$,
the number of its critical points of index $q$.  Then there exist $d$
integers $r_q \geq 0$ such that
\begin{equation}
  \label{eq:Morse_ineq_full}
  \begin{split}
  c_0 &= \beta_0 + r_1,\\
  c_1 &= \beta_1 + r_1 + r_2,\\
  c_2 &= \beta_2 + r_2 + r_3,\\
  &\ldots\\
  c_{d-1} &= \beta_{d-1}+r_{d-1}+r_d,\\
  c_d &= \beta_d + r_d,    
  \end{split}
\end{equation}
where $\beta_q$ is the $q$-th
Betti number of the manifold $M$, defined as the rank of the homology
group\footnote{Throughout the paper, we use integer coefficient
  homology, unless specified otherwise.} $H_q(M)=H_q(M;\Z)$.  To put it
another way, the Betti numbers $\beta_q$ give the lower bound for the 
number of critical points of index $q$; extra critical points can only
be created in pairs of adjacent index.

Equations~\eqref{eq:Morse_ineq_full} can be expressed concisely in
terms of generating functions: one defines the \term{Morse polynomial}
$P_\phi(t)$ of a Morse function $\phi$ as the sum of $t^{\mu(x)}$ over
all critical points $x \in M$ of $\phi$.  On the topological side, the
\term{Poincar\'e polynomial} $P_{M}(t)$ of the manifold $M$ is the sum
of $\beta_qt^q$.  Then the \term{Morse inequalities} are equivalent to
the identity
\begin{equation}
\label{eq:Morse_inequalities}
\left( P_\phi(t) - P_{M}(t) \right) / (1+t) = R(t),
\end{equation}
where $R(t)$ is a polynomial with nonnegative coefficients.  

Now assume that $\phi$ is just continuous; the local sublevel sets
$U^{\pm\epsilon}_x(\phi)$ are still well-defined.  Mimicking the
classical Morse theory of smooth functions we adopt the following
definitions (cf.\ \cite[\S 45, Def.\ 1, 2 and 3]{FF89}, the critical
points are called bifurcation points there):

\begin{definition}
  \label{gencritdef}	
  A point $x\in M$ is a \term{topologically regular point} of a continuous function
  $\phi$ if there exists a small enough neighborhood $U$ of $x$ in $M$ and
  $\epsilon>0$ such that $U_x^{-\epsilon}(\phi)$
  is a strong deformation retract
  of $U_x^{+\epsilon}(\phi)$. We say that a point is
  \term{topologically critical} if it is not topologically regular.
\end{definition}

\begin{remark}
  \label{rem:crit_not_top_crit}
  If $\phi$ is smooth, a topologically critical point $x$ is also
  critical in the usual (differential) sense.  The converse is, in
  general, not true: for example, if $M=\mathbb R$ and $\phi(x)=x^3$,
  then $x=0$ is critical but not topologically critical.  On the other
  hand, by the aforementioned main result of the classical Morse
  theory, if $x$ is a \emph{non-degenerate} critical point then it is
  also topologically critical.
\end{remark}

\begin{definition}
  \label{Morsepoly} 
  Given a continuous function $\varphi$ with a finite set of topologically critical
  points, the \term{Morse
  polynomial} $P_\phi$ is the sum, over the  topologically critical points $x$, of the
  Poincar\'e polynomials of the relative homology groups
  $H_*\big(U_x^{+\epsilon}(\phi),\,
  U_x^{-\epsilon}(\phi)\big)$,
  where $U$ is a small neighborhood of $x$ and $\epsilon>0$ is
  sufficiently small.
\end{definition}

\begin{remark}
  If $\phi$ is a \emph{smooth} Morse function,
  Definition~\ref{Morsepoly} reduces to the classical one as the
  relative homology groups
  $H_*\bigl(U_x^{+\epsilon}(\phi),\,
  U_x^{-\epsilon}(\phi)\bigr)$ coincide with the reduced
  homology groups of the $\mu(x)$-dimensional sphere
  $\mathbb S^{\mu(x)}$, where $\mu(x)$ is the Morse index of $x$, and
  so the contribution of $x$ to the Morse polynomial $P_\phi(t)$ is
  equal to $t^{\mu(x)}$.
\end{remark}

With Definitions \ref{gencritdef} and \ref{Morsepoly}, the Morse
inequalities \eqref{eq:Morse_inequalities} hold true for continuous
functions $\phi$ with finite number of topologically critical points
(see, e.g., \cite[\S 45, Theorem.\ 1]{FF89}). The proof is essentially
the same as the proof of the classical Morse inequality given in
\cite[\S 5]{Milnor_MorseTheory} and is based on the exact sequence of
pairs: the latter implies the subadditivity of relative Betti numbers
and, more generally, of the partial alternating sums of relative Betti
numbers, which implies the required Morse inequalities.

\emph{It is thus our goal to give a
  prescription for computing the Morse polynomial $P_{\lambda_k}$
  in terms of $\F$ and its derivatives, under some natural
assumptions on $\F$.}  To
that end we will need to:
\begin{enumerate}
\item Provide an explicit characterization of non-smooth topologically
  critical and topologically regular points of $\lambda_k$;
\item Give a natural definition of a \emph{non-degenerate} non-smooth
  topologically critical point;
\item For a non-degenerate topologically critical point $x$ of $\lambda_k$, find
  the relative homology
  \begin{equation*}
    H_q(U_x^{\epsilon}(\lambda_k),\,
    U_x^{-\epsilon}(\lambda_k)) 
  \end{equation*}
  for a sufficiently small neighborhood $U$ of $x$ and sufficiently
  small $\epsilon>0$.  As a by-product, this will determine the
  correct contribution from $x$ to the Morse polynomial
  $P_{\lambda_k}(t)$ of $\lambda_k$.
\end{enumerate}
We remark that these questions are local in nature and we do not
need to enforce compactness of $M$ while answering them.

In this work, we completely implement the above objectives in the case
of generic smooth families; additionally, our sufficient condition for a
regular point is obtained for arbitrary families.  The first objective
is accomplished in the form of a ``first derivative test'', with the
derivative being applied to the smooth object: the family $\F$ (see
equation~\eqref{eq:HF_matrix} and Theorems \ref{thm:regular} and
\ref{thm:critical} for details).

The Morse contribution of a critical point (third objective) will
consist of two parts: the classical index of the Hessian of
$\lambda_k$ in the directions of smoothness of $\lambda_k$ and a
contribution from the non-smooth directions which, remarkably, turns
out to \emph{depend only on the size of the eigenvalue multiplicity
  and the relative position of the eigenvalue of interest and not on
  the particulars of the operator family}.
Theorem~\ref{thm:homology_geom} expresses this contribution in terms
of homologies of suitable Grassmannians; explicit formulas for the
Poincar\'e polynomial are provided in Theorem~\ref{thm:critical}.  In
Section~\ref{sec:applications} we mention some simple practical
corollaries of our results as well as pose further problems.

\subsection{A differential characterization of a topologically critical point}
\label{sec:intro_cp}

Our primary focus is on the points $x\in M$ where the eigenvalue
$\lambda_k$ has multiplicity and is not differentiable.  However,
simple examples (for instance, Example~\ref{ex:inclined_cones} below)
show that not every point of eigenvalue multiplicity is topologically
critical.

\newcommand\U{\mathcal{U}}

Denote by $\esp_k$ the
eigenspace of $\lambda_k$ at a point $x\in M$ of multiplicity
$\nu = \dim\esp_k$. The \term{compression} of a matrix $X \in
  \Sym_n$ to the space $\esp_k$ is the linear operator
  $X_{\esp_k}: \esp_k \to \esp_k$ acting as $v \mapsto P_{\esp_k} X
  v$, where $P_{\esp_k}$ is the orthogonal projector onto $\esp_k$.
  The matrix representation of $X_{\esp_k}$ can be computed as
  \begin{equation}
    \label{eq:restriction_def}
    X _{\esp_k} := \U^* X \U,
  \end{equation}
  where $\U: \mathbb{F}^\nu\to\mathbb{F}^n$ is a linear isometry such
  that $\Ran(\U) = \esp_k$ (explicitly, the columns of $\U$ are an
  orthonormal basis of $\esp_k$).  Introduce the linear operator
$\mathcal{H}_{x}: T_{x}M \to \Sym_\nu$ acting as
\begin{equation}
  \label{eq:HF_matrix}
  \cH_{x} \colon
  v \mapsto \big( d\mathcal{F}(x)v \big)_{\esp_k},
\end{equation}
While the operator $\cH_{x}$ depends on the choice of the isometry
$\U$ in \eqref{eq:restriction_def} (or, equivalently, the choice of basis in
$\esp_k$), we will only use its properties that are invariant
under unitary conjugation.

We recall that a matrix $A\in\Sym_\nu$ is \term{positive semidefinite}
(notation: $A\in\Sym_\nu^+$) if all of its eigenvalues are
non-negative, \term{positive definite} (notation: $A\in\Sym_\nu^{++}$)
if all eigenvalues are strictly positive.  
 We denote by $S^\perp$ the orthogonal
complement of a space $S$ in $\Sym_\nu$ with respect to the
\term{Frobenius inner product}
$\langle X, Y\rangle:=\Tr (XY)$.  For future reference we note that if
$X\in\Sym_\nu^{++}$ and $Y\in \Sym_\nu^{+}$, $Y\neq 0$, then $\langle
X, Y\rangle>0$ (see, e.g., \cite[Example~2.24]{BoydVandenberghe_optimization}).

Our first main result gives a sufficient condition for a point of
eigenvalue multiplicity to be topologically regular.

\begin{theorem}	
  \label{thm:regular}
  Let $\F: M \to \Sym_n$ be a smooth family whose eigenvalue
  $\lambda_k$ has multiplicity $\nu\geq1$ at the point $x\in M$.  If
  $\Ran\cH_{x}$ contains a positive definite matrix or, equivalently\footnote{Note that this equivalence is not immediate and is established  in the beginning of the proof of Theorem \ref{thm:APS_Clarke_regular} below.},
  \begin{equation}
    \label{eq:regular_condition}
    \left(\Ran\cH_{x}\right)^\perp \cap \Sym_\nu^+ = 0,
  \end{equation}
  then $x$ is topologically regular for $\lambda_k$.
\end{theorem}

This theorem is proved in Section~\ref{sec:precritical_but_regular} by
studying the Clarke subdifferential at the point $x$.  We formulate
the conditions in terms of both $\Ran\cH_{x}$ and
$\left(\Ran\cH_{x}\right)^\perp$ because the former emerges naturally
from the proof while the latter is simpler in practical computations:
generically it is one- or zero-dimensional as we will see in
Section~\ref{sec:transversality}. 

\begin{remark}
  \label{rem:simple_eig}
  Condition (\ref{eq:regular_condition}) should be viewed as being
  analogous to the ``non-vanishing gradient'' in the smooth Morse
  theory.  By what is sometimes called Hellmann--Feynman theorem (see
  Appendix~\ref{sec:hellmann_feynman} and references therein), the
  eigenvalues of $\cH_{x}v\in \Sym_\nu$ give the slopes of the
  branches splitting off from the multiple eigenvalue
  $\lambda_k(\mathcal{F}(x))$ when we leave $x$ in the direction $v$.
  The regularity condition of Theorem~\ref{thm:regular} is equivalent
  to having a direction in which \emph{all} eigenvalues are increasing.

  To further illustrate this point, consider the special case $\nu=1$
  when the eigenvalue $\lambda_k$ is smooth.  Let $\psi$ be the
  eigenvector corresponding to $\lambda_k$ at the point $x$.  The
  operator $\cH_{x}: T_{x}M \to \R$ in this case maps $v$ to
  $\left< \psi, \big(d\mathcal{F}(x)v\big) \psi
  \right>_{\mathbb{F}^n}$
  which is equal to the directional derivative of $\lambda_k(x)$ in
  the direction $v$.  The condition of Theorem~\ref{thm:regular} is
  precisely that this derivative is non-zero in some
  direction, i.e.\ the gradient does not vanish.
\end{remark}

Due to the topological nature of Definition~\ref{gencritdef}, one
cannot expect that a zero gradient-type condition alone would be
sufficient for topological criticality (cf.\
Remark~\ref{rem:crit_not_top_crit}).  To formulate a sufficient
condition we need some notion of ``non-degeneracy'', which will have a
smooth (S) and non-smooth (N) parts.

\begin{definition}
  \label{def:nondeg_crit_N}
  Let $\F: M \to \Sym_n$ be a smooth family whose eigenvalue
  $\lambda_k$ has multiplicity $\nu\geq1$ at the point $x\in M$.  We
  say that $\F$ satisfies the \term{non-degenerate criticality
    condition (N)} at the point $x$ if
  \begin{equation}
    \label{eq:ndccn}
    \left(\Ran\cH_{x}\right)^\perp = \Span\{B\},
    \quad
    B \in \Sym_n^{++}.
  \end{equation}
\end{definition}

\begin{remark}
  \label{rem:explaining_condition1}
  Condition~\eqref{eq:ndccn} ensures non-degenerate criticality in the
  directions in which $\lambda_k$ is non-smooth (hence ``N''); a
  single condition plays two roles:
  \begin{itemize}
  \item it ensures that \eqref{eq:regular_condition} is
    violated (intuitively, ``the gradient is zero''), and
  \item it ensures that $\Ran\cH_{x}$ has codimension 1, which will be
    interpreted in Section~\ref{sec:transversality} as a type of
    transversality condition (intuitively, ``non-degeneracy in the
    directions in which $\lambda_k$ is non-smooth'').
  \end{itemize}
  
\end{remark}

Once condition (N) is satisfied at a point $x$, we need to pay special
attention to a submanifold $S$ where the multiplicity of $\lambda_k$
remains the same.

\begin{proposition}
  \label{prop:const_mult_stratum}
  Let $\F: M \to \Sym_n$ be a smooth family whose eigenvalue
  $\lambda_k$ has multiplicity $\nu\geq1$ at the point $x\in M$.  If
  $\F$ satisfies the non-degenerate criticality condition (N) at the
  point $x$, then there exists a submanifold $S\subset M$ such that
  for any $y$ in a small neighborhood of $x$ in $M$, the multiplicity
  of $\lambda_k(y)$ is equal to $\nu$ if and only if $y\in S$.

  This submanifold, which we will call the \term{(local) constant
    multiplicity stratum} attached to $x$, has the following
  properties:
  \begin{enumerate}
  \item\label{item:codimensionS} $S$ has codimension
    $s(\nu) := \dim \Sym_\nu(\mathbb{F}) - 1$ in $M$,
  \item\label{item:restrictionCritical} the restriction
    $\lambda_k \big|_{S}$ is a smooth function which has a critical
    point at $x$, i.e.\ $d\left(\lambda_k\big|_S\right)(x)=0$.
  \end{enumerate}
\end{proposition}

The proof of the above Proposition is in Section~\ref{proof:prop}.

\begin{definition}
  \label{def:nondeg_crit_S}
  Assume $\F$ satisfies the non-degenerate criticality condition (N)
  at the point $x$ for the eigenvalue $\lambda_k$ (in
  particular, x is a critical point of $\lambda_k \big|_{S}$).  We will say
  $\F$ satisfies the \term{non-degenerate criticality condition (S)}
  if $x$ is a \emph{non-degenerate} critical point of  $\lambda_k \big|_{S}$.
\end{definition}

Naturally, ``S'' stands for smooth criticality.  It turns out that,
together, conditions (N) and (S) are sufficient for topological
criticality.  To quantify the topological change in the sublevel sets
we need additional terminology.  The \term{relative index} of the
$k$-th eigenvalue at point $x$ is
\begin{equation}
  \label{eq:relative_index_def}
  i(x) = \#\left\{\lambda \in\spec\bigl(\F(x)\bigr) \colon
    \lambda \leq \lambda_k(x)\right\} - k + 1.
\end{equation}
In other words, $i(x)$ is the sequential number of $\lambda_k$ among
the eigenvalues equal to it, but counting from the top.  It is an
integer between $1$ and $\nu(x)$, the multiplicity of the
eigenvalue $\lambda_k(x)$ of the matrix $\F(x)$.
We will need the quantity $s(i) := \dim \Sym_i(\mathbb{F}) - 1$, which
  already appeared in a different role in 
  Proposition~\ref{prop:const_mult_stratum}. It is given explicitly by
\begin{equation}
  \label{eq:codim_nu}
  s(i) :=
  \dim \Sym_i(\mathbb{F}) - 1 = 
  \begin{cases}
    \frac12i(i+1)-1, &\mathbb F=\mathbb R,\\
    i^2-1, & \mathbb F=\mathbb C.
  \end{cases}
\end{equation}
Finally, we denote by $\binom{n}{k}_q$ the $q$-binomial coefficient,
\begin{equation}
  \label{qbinom}
  \binom{n}{k}_q := \cfrac{\prod_{i=1}^n (1-q^i)}
  {{\prod_{i=1}^k} (1-q^i)\prod_{i=1}^{n-k} (1-q^i)},
\end{equation}
which is well known (\cite[Corollary 2.6]{KC02}) to be a polynomial in
$q$.

\begin{definition}
  \label{def:generalized_Morse}
  A smooth family $\F:M \to \Sym_n(\FF)$ is called \term{generalized
    Morse} if, at every point $x \in M$, $\F$ either satisfies the
  regularity condition~\eqref{eq:regular_condition} or satisfies the
  the non-degenerate criticality conditions (N) and (S).
\end{definition}

\begin{theorem}	
  \label{thm:critical}
  Consider the eigenvalue $\lambda_k$ of a smooth family
  $\F: M \to \Sym_n$.  
  
  \begin{enumerate}
  \item \label{item:criticality} If $\F$ satisfies non-degenerate criticality
    conditions (N) and (S) at $x$, then $x$ is a topologically critical
    point of $\lambda_k$.  The set of points $x$ where conditions (N)
    and (S) are satisfied is discrete.
    
  \item \label{item:Morse_polynomial}  If $M$ is a compact manifold and the family
    $\F$ is generalized Morse, Morse inequalities
    \eqref{eq:Morse_inequalities} hold for the function $\lambda_k:M\to\R$
    with the Morse polynomial $P_\phi(t) := P_{\lambda_k}(t)$ given by
    \begin{equation}
      \label{eq:Morse_poly_total_short}
      P_{\lambda_k}(t)   
      := \sum_{x\in \mathrm{CP}(\F)} P_{\lambda_k} (t; x),
    \end{equation}
    where the summation is over all topologically critical points $x$
    of $\F$ and, denoting by $\nu(x)$ the multiplicity of the
    eigenvalue $\lambda_k$ of $\F(x)$, by $i(x)$ its relative index,
    and by $\mu(x)$ the Morse index of the restriction
    $\lambda_k\big|_S$,
    \begin{equation}
      \label{eq:Morse_contrib_nonsmooth}
      P_{\lambda_k}(t; x):= 
      t^\mu\,\fT_\nu^i = t^{\mu+s(i)}
      \begin{cases}
        {\binom{\lfloor (\nu-1)/2\rfloor }{(i-1)/2}}_{t^4},
        & \text{$\FF=\R$ and $i$ is odd}, \\[5pt]
        0, &  \text{$\FF=\R$, $i$ is even,  and $\nu$ is odd}, \\[5pt]
        t^{\nu-i}{\binom{\nu/2-1}{i/2-1}}_{t^4},
        & \text{$\FF=\R$, $i$ is even, and $\nu$ is even},\\[5pt]
        {\binom{\nu-1}{i-1}}_{t^2},
        & \text{$\FF=\C$}.
      \end{cases}
    \end{equation}
  \end{enumerate}
\end{theorem}
The topological criticality claimed in part (\ref{item:criticality})
follows immediately whenever the Poincar\'e polynomial of the relative
$\Z$-homology --- which is given in
\eqref{eq:Morse_contrib_nonsmooth} --- is non-zero.  The case in the
second line of \eqref{eq:Morse_contrib_nonsmooth} is more complicated,
because that Poincar\'e polynomial is zero.  To handle this case, in the
final stages of the proof in Section~\ref{sec:main_theorem_part1} we
will additionally calculate the Poincar\'e polynomial of the relative
$\mathbb Z_2$-homology, see~\eqref{Poincare_torsion_Z2_Grassmannian_1}.

We also note that in equation~\eqref{eq:Morse_contrib_nonsmooth} we
introduced the notation $\fT_\nu^i$ for the family-independent Morse
contribution to $P_{\lambda_k} (t; x)$.  This contribution arises from
the ``non-smooth'' directions transverse 
to the ``smooth'' constant multiplicity stratum $S$.  The index
$\mu = \mu(x)$ along the stratum $S$ depends on the particulars of the
family $\F$.

Now we state a
result showing that for a ``typical'' $\F$, either Theorem~\ref{thm:regular} or
Theorem~\ref{thm:critical} holds at every point $x\in M$.

\begin{theorem}
  \label{thm:gen_Morse_generic_intro}
  The set of generalized Morse families is open and dense in the
  Whitney topology in $C^{r} (M,\Sym_n)$ for $2\leq r\leq \infty$.
\end{theorem}

This result will be established in Section~\ref{sec:transversality} as
a part of Theorem~\ref{thm:gen_Morse_genericty}.  We will use
transversality arguments similar to those in the proof of genericity
of classical Morse functions (see, for example, \cite[Chapter 4, Theorem
1.2]{Hirsch94}) via the strong (or jet) Thom transversality theorem
for stratified spaces.

\subsection{Geometrical description of the relative homology groups}
\label{sec:intro_morse_poly}

In this subsection we provide some idea of what goes into the proof of
Theorem~\ref{thm:critical}, describing some geometric objects whose
integer
homology is quantified in \eqref{eq:Morse_contrib_nonsmooth}.

%


First we introduce some notation. 
We denote by $\Gr_\FF(k,n)$ the Grassmannian of (non-oriented)
$k$-dimensional subspaces in $\FF^n$.
Theorem~\ref{thm:homology_geom} below uses certain homologies of
$\Gr_\R(k,n)$ with local coefficients, namely
$H_*\bigl(\Gr_\R(k, n);\widetilde\Z\bigr)$.  The construction of
this homology can be found, for instance, in \cite[Sec.~3H]{Hatcher} or
\cite[Chapter~5]{DK01}; it will also be briefly summarized in
Section~\ref{sec:main_theorem_proofs}.

Recall that for a given topological space $Y$, the \term{cone} of $Y$
is $\mathcal{C}Y := Y\times [0,1] / \bigl(Y \times \{0\}\bigr)$, and
the \term{suspension} of $Y$ is
$\mathcal S Y := \mathcal{C}Y / \bigl(Y \times \{1\}\bigr)$.  For
example, $\mathcal S \mathbb{S}^\mu = \mathbb{S}^{\mu+1}$.

\begin{theorem}
  \label{thm:homology_geom}
  In the context of Theorem~\ref{thm:critical}, we have the following
  equivalent descriptions of the relative homology
  $H_r\bigl(U_x^{+\epsilon}(\lambda_k),\,
  U_x^{-\epsilon}(\lambda_k)\bigr)$,
  \begin{enumerate}
  \item \label{item:Hvia_Rspace}
    \begin{equation}
      \label{eq:Hvia_Rspace}
      H_r\bigl(U_x^{+\epsilon}(\lambda_k),\,
      U_x^{-\epsilon}(\lambda_k)\bigr)
      =
      \begin{cases}
        \mathbb{Z}, & \text{if}\quad i(x)=1,\ r=\mu(x), \\
        0, & \text{if} \quad i(x)=1,\ r\neq \mu(x), \\
        \widetilde{H}_{r-\mu(x)}\left(\mathcal{S}
          \mathcal{R}_{\nu(x)}^{i(x)}\right),
        & \text{if} \quad 1 < i(x) \leq \nu(x),
      \end{cases}
    \end{equation}
    where
    \begin{equation}
      \label{eq:hatR_def}
      \mathcal{R}_\nu^i :=
      \{R\in \Sym^+_\nu \colon \Tr R=1, \rank R < i\},
    \end{equation}
    and $\widetilde{H}_q$ denotes the $q$-th reduced homology group.
    
  \item \label{item:Hvia_Grass}
    \begin{equation}
      \label{eq:relhom_main}
      H_r\bigl(U_x^{+\epsilon}(\lambda_k),\,
      U_x^{-\epsilon}(\lambda_k)\bigr)
      =
      \begin{cases}
        H_{r-\mu-s(i)}\big(\Gr_\R(i-1, \nu-1)\big),
        & \text{$\FF=\R$ and $i$ is odd},\\
        H_{r-\mu-s(i)}\bigl(\Gr_\R(i-1, \nu-1);\,\widetilde\Z\bigr),
        & \text{$\FF=\R$ and $i$ is even}, \\
        H_{r-\mu-s(i)}\big(\Gr_\C(i-1, \nu-1)\big),
        & \text{$\FF=\C$}.
      \end{cases}
    \end{equation}
  \end{enumerate}
\end{theorem}

To prove Theorem~\ref{thm:homology_geom}, in
Section~\ref{sec:sublevel_change1} we will first separate out the
contribution to the relative homology
$H_r\bigl(U_x^{+\epsilon}(\lambda_k),\,
U_x^{-\epsilon}(\lambda_k)\bigr)$ of the local constant multiplicity
stratum $S$ and reduce the computation to the case when $S$ is a
single point.  In the latter case, it will be shown that
$H_r\bigl(U_x^{+\epsilon}(\lambda_k),\,
U_x^{-\epsilon}(\lambda_k)\bigr)$ reduces to the homology of the space
$\mathcal{S}\mathcal{R}_\nu^i$.  In the next step, we will see that
$\mathcal{S}\mathcal{R}_\nu^i$ is homotopy equivalent to the Thom
space of a real bundle of rank $s(i)$ over the Grassmannian
$\Gr_\FF(i-1, \nu-1)$.  The difference between the odd and the even
$i$ when $\FF=\R$ is that this real bundle is orientable in the former
case and non-orientable in the latter.  So, part
(\ref{item:Hvia_Grass}) of the Theorem follows from the Thom
isomorphism theorem in the oriented bundle case and more general tools
such as the usual/twisted version of Poincar\'e--Lefschetz duality in
the non-orientable bundle case \cite{DK01,Hatcher,FF16}.

The study of $\mathbb Z_2$ and integer homology groups of the complex and real
Grassmannians was at the heart of the development of algebraic
topology and, in particular, the characteristic classes.  Starting
from the classical works of Ehresmann
\cite{Ehresmann1934,Ehresmann1937}, the answers appearing in
\eqref{eq:Morse_contrib_nonsmooth} were explicitly
calculated using the Schubert cell decomposition and combinatorics of
the corresponding Young diagrams \cite[Theorem IV, p. 108]{Iwamoto48},
\cite{Andrews76}.  The calculation of twisted homologies of real
Grassmannians is less well-known but can be deduced from the classical
work \cite{Chern1951} and incorporated into a unified algorithm
\cite{CasKod13}, or computed by means of the general theory of de Rham
cohomologies of homogeneous spaces, see \cite[chapter XI,
pp.\ 494-496]{GWH1976}.

Examples of local contributions to the Morse polynomial for
topologically critical points of multiplicities up to $8$ are
presented in Table~\ref{tab:Morse_poly} in the real case.  The
possible contribution from the smooth directions is ignored because
those are specific to the family $\F$.  In other words, we set
$\mu(x)=0$ in equation~\eqref{eq:Morse_contrib_nonsmooth}.  In the
cases when the second line of \eqref{eq:Morse_contrib_nonsmooth}
applies, the contribution of $0$ \emph{does not} mean that the point
is regular; $0$ appears because the polynomial ignores the torsion
part of the corresponding homologies.  We also observe that the
  contribution of the top eigenvalue ($i=1$) is always $t^0$; the
  contribution of the bottom eigenvalue ($i=\nu$) is always
  $t^{s(\nu)}$.  By analogy with smooth Morse theory one can guess
  that the top eigenvalue always experiences a minimum, while the
  bottom eigenvalue always experiences a maximum ($s(\nu)$ being the
  dimension of the space of non-smooth directions).  This guess is
  rigorously established in Corollary \ref{extremum_criteria} and its
  proof in section \ref{sec:main_theorem_proofs}.


\begin{table}
  \centering
  \begin{tabular}{|l|c|c|c|c|c|c|c|c|}
    \hline
    \backslashbox{$\nu$}{$i$}& $1$ & $2$ & $3$ & $4$ &
                                                       $5$ & $6$ & $7$
    &$8$\\ 
    \hline
    $2$
                             & 1 & $t^2$ &&&&&&\\
    \hline
    $3$
                             & 1 & 0 & $t^5$ &&&&&\\
    \hline
    $4$
                             & 1 & $t^4$ & $t^5$ & $t^9$ &&&&\\
    \hline
    $5$
                             & 1 & 0 & $t^5+t^9$ & 0 & $t^{14}$ &&& \\
    \hline
    $6$
                             & 1 & $t^6$ & $t^5+t^9$ & $t^{11}+t^{15}$
                                                     & $t^{14}$ & $t^{20}$ &&
    \\
    \hline
    $7$
                             & 1 & 0 & $t^5+t^9+t^{13}$ & 0 & $t^{14}+t^{18}+t^{22}$
                                                           & 0 &  $t^{27}$ &\\
    \hline 
    $8$ &
          1&$t^8$&$t^5+t^9+t^{13}$&$t^{13}+t^{17}+t^{21}$&$t^{14}+t^{18}+t^{22}$
                                                           &$t^{22}+t^{26}+t^{30}$&$t^{27}$&$t^{35}$\\
    \hline                                           
  \end{tabular}
  \vskip .1in
  \caption{Non-smooth contributions $\fT_\nu^i(t)$ to the Morse
    polynomial from a topologically critical point of $\lambda_k(x)$
    in the real case ($\FF=\R$, first three cases of
    equation~\eqref{eq:Morse_contrib_nonsmooth}).}
  \label{tab:Morse_poly}
\end{table}

\subsection*{Acknowledgement}

We are grateful to numerous colleagues who aided us with helpful
advice and friendly encouragement.  Among them are Andrei Agrachev,
Lior Alon, Ram Band, Mark Goresky, Yuji Kodama, Khazhgali Kozhasov,
Peter Kuchment, Sergei Kuksin, Sergei Lanzat, Antonio Lerario, Jacob
Shapiro, Stephen Shipman, Frank Sottile, Bena Tshishiku, and Carlos
Valero.  We also thank anonymous reviewers for many insightful
comments that improved our paper. GB was partially supported by NSF
grants DMS-1815075 and DMS-2247473.  IZ was partially supported by NSF
grant DMS-2105528 and Simons Foundation Collaboration Grant for
Mathematicians 524213.

\section{Examples, applications and an open question}

\subsection{Examples}
\label{sec:examples}

In this section we collect examples illustrating our criteria for
regularity and criticality.

\begin{example}
  \label{ex:inclined_cones}
  Consider the two families
  \begin{equation}
    \label{eq:two_cones}
    \F_1(x) =
    \begin{pmatrix}
      x_1 & x_2 \\ x_2 & -x_1
    \end{pmatrix},
    \quad\mbox{and}\quad
    \F_2(x) = 
    \begin{pmatrix}
      x_1 & x_2 \\ x_2 & 2x_1
    \end{pmatrix},
    \qquad
    x = (x_1,x_2) \in \R^2.
  \end{equation}
  Both families $\F_1$ and $\F_2$ have an isolated point of
  multiplicity 2 at $(x_1,x_2)=(0,0)$.  Focusing on the lower
  eigenvalue $\lambda_1$, its level curves in the case of
  $\F_1$ undergo a significant change at the value $0$ ---
  they change from circles to empty, see
  Fig.~\ref{fig:two_cones}(top).  Therefore, the point $(0,0)$ is topologically critical and, visually, $\lambda_1$ of $\F_1$ has
  a maximum at $(0,0)$.  In contrast, the level curves and the
  sublevel sets of $\F_2$ remain homotopically equivalent, see
  Fig.~\ref{fig:two_cones}(bottom).  The point $(0,0)$ is not topologically critical
  for $\lambda_1$ of $\F_2$.

  \begin{figure}
    \centering
    \includegraphics[scale=0.30]{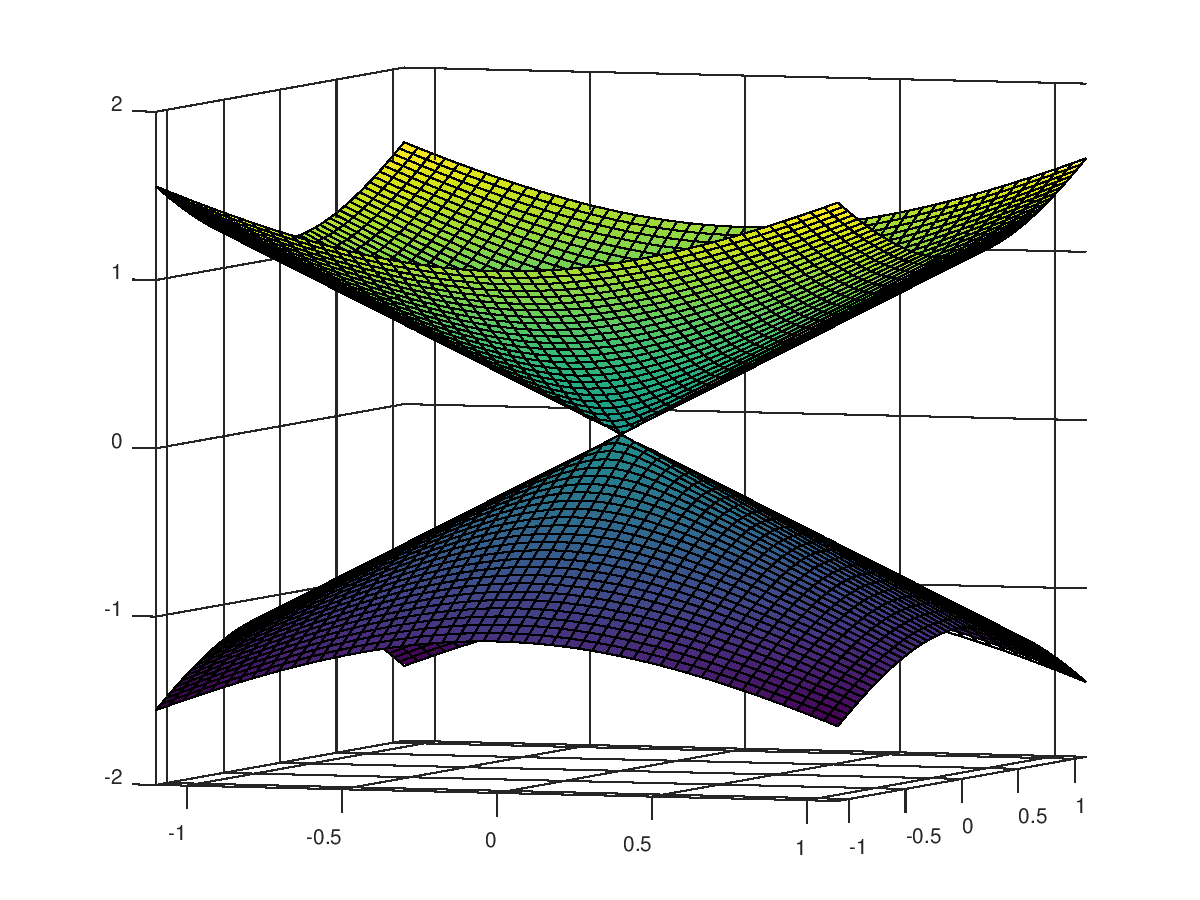}
    \includegraphics[scale=0.30]{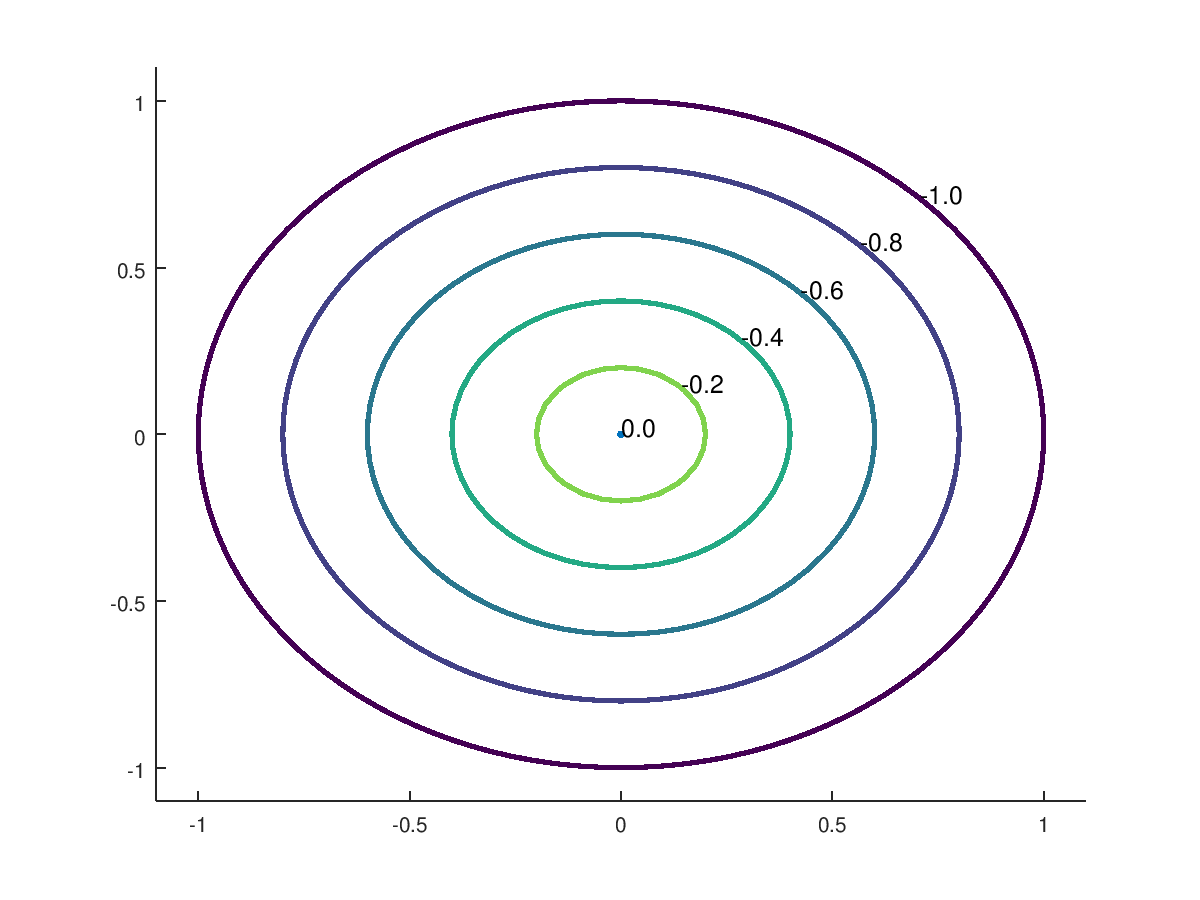}\\
    \includegraphics[scale=0.30]{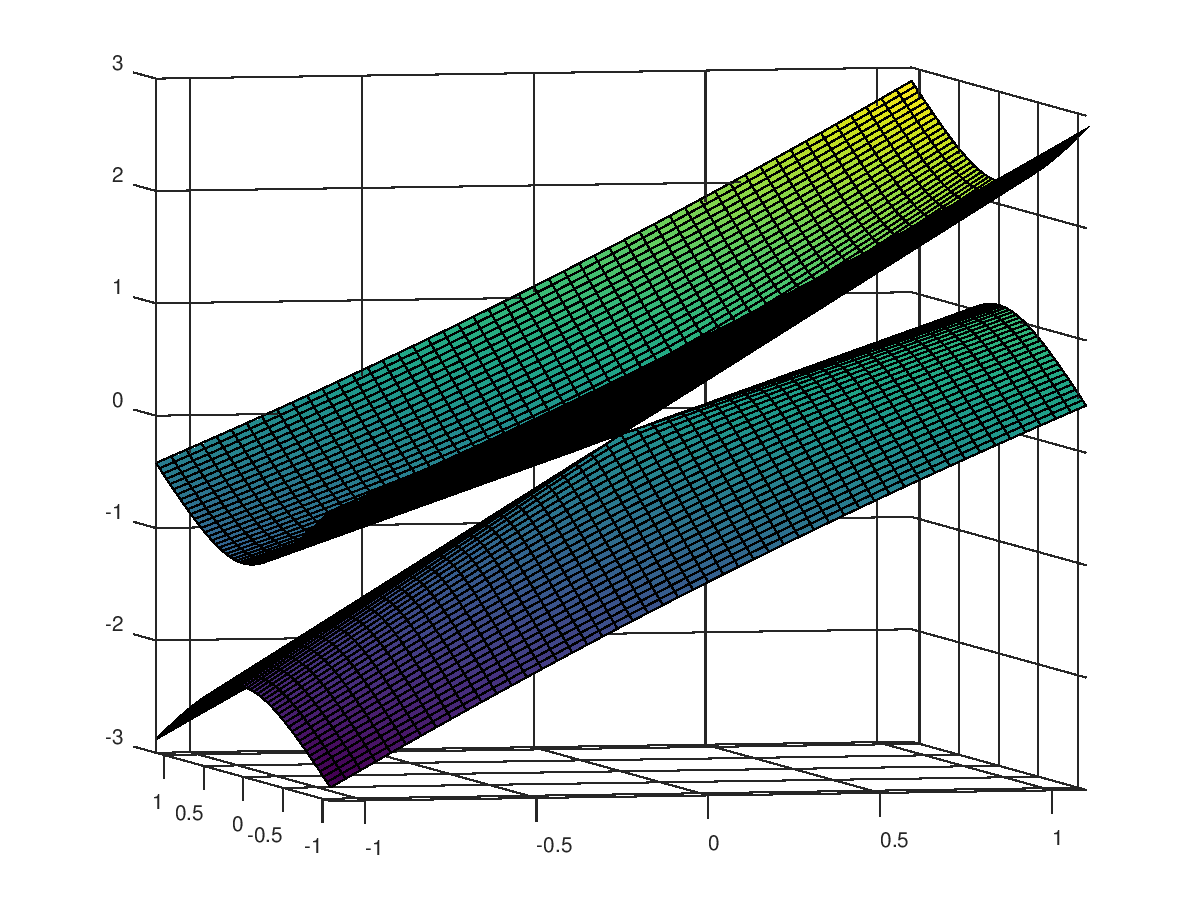}
    \includegraphics[scale=0.30]{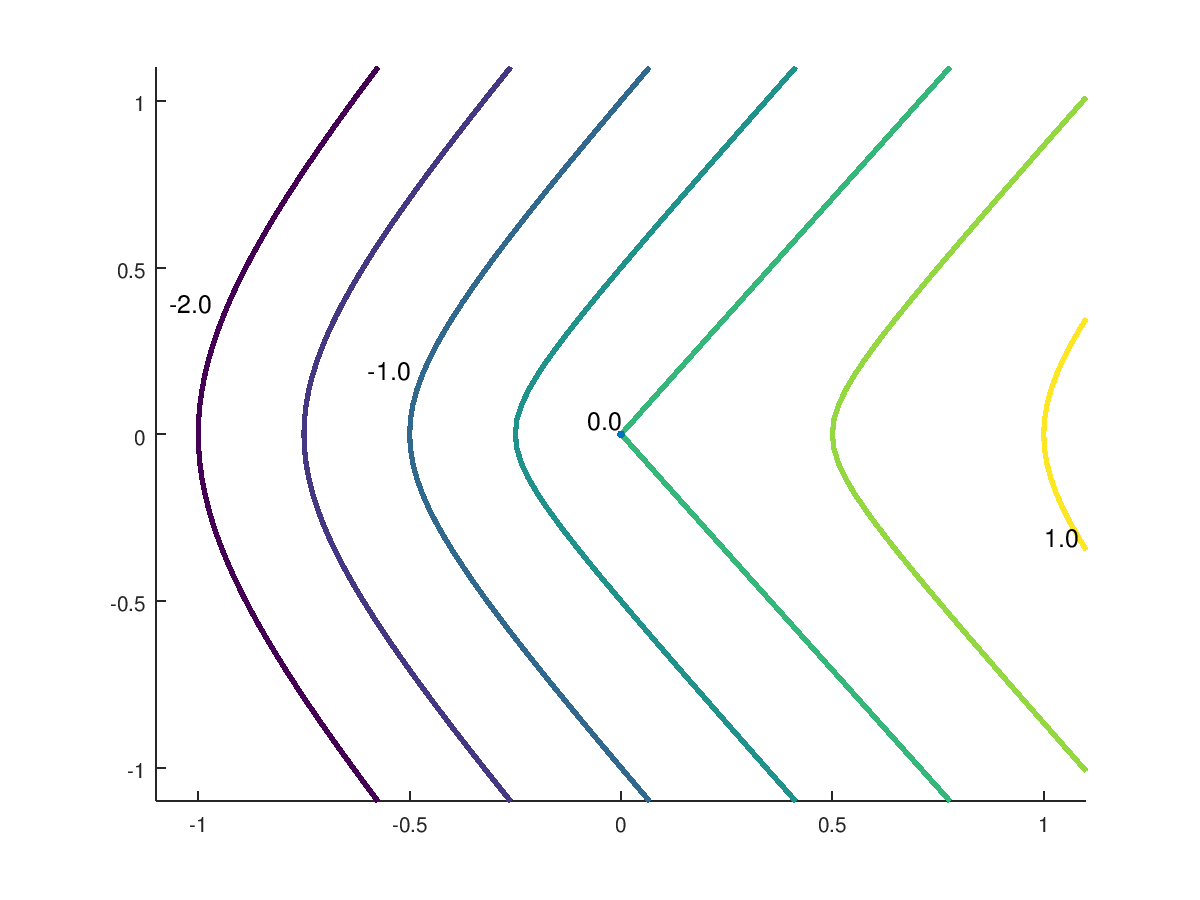}
    \caption{Eigenvalue surfaces (left) and contours of the first
      eigenvalue (right) for the families $\F_1$ (top) and
      $\F_2$ (bottom) from
      equation~\eqref{eq:two_cones}.}
    \label{fig:two_cones}
  \end{figure}
 
  We now check the condition of Theorem \ref{thm:regular} for the
  families $\F_1$ and $\F_2$. At the point $x=(0,0)$ the eigenspace $\esp_k$ is the
  whole of $\R^2$ and no restriction is needed.
  For the family $\F_1$, the mapping
  $\cH_x$ from \eqref{eq:HF_matrix} is
  \begin{equation*}
    \cH_x : v=
    \begin{pmatrix}
      v_1 \\ v_2
    \end{pmatrix}
    \mapsto v_1
    \begin{pmatrix}
      1 & 0 \\ 0 & -1
    \end{pmatrix}
    + v_2
    \begin{pmatrix}
      0 & 1 \\
      1 & 0
    \end{pmatrix},
  \end{equation*}
  and therefore
  \begin{equation*}
    \left(\Ran\cH_{x}\right)^\perp = 
    \Span\left\{
      \begin{pmatrix} 1&0\\0&1\end{pmatrix}
    \right\},
  \end{equation*}
  satisfying non-degenerate criticality condition (N),
  \eqref{eq:ndccn}.  Since the point $x=(0,0)$ of eigenvalue
  multiplicity 2 is isolated, the criticality condition (S) is
  vacuously true.  Theorem~\ref{thm:critical} applies at $x=(0,0)$ and
  the Morse data for the two sheets is given by the $\nu=2$ row of
  Table~\ref{tab:Morse_poly}.
  
  Proceeding to the family $\F_2$, a similar calculation yields
  \begin{equation*}
    \left(\Ran\cH_{x}\right)^\perp = 
    \Span\left\{
      \begin{pmatrix} 2&0\\0&-1\end{pmatrix}
    \right\}
  \end{equation*}
  contains no positive semidefinite matrices.  Hence $x$ is a
  topologically regular point for $\F_2$ by Theorem~\ref{thm:regular}.
\end{example}

\begin{example}
	\label{semidef_example}
  The case of $\left(\Ran\cH_{x}\right)^\perp$ being spanned by a
  semidefinite matrix which satisfies
  neither~\eqref{eq:regular_condition} nor condition~\eqref{eq:ndccn},
  is borderline.  As an example, consider the family
  \begin{equation}
    \label{eq:borderline_family}
    \F(x) =
    \begin{pmatrix}
      x_1 & x_2 \\
      x_2 & x_1x_2 + x_1^2
    \end{pmatrix}.
  \end{equation}
  For the point $x=(0,0)$ of multiplicity 2 we have
  \begin{equation*}
    \Ran\cH_{x} = \Span\left\{
      \begin{pmatrix}
        1 & 0 \\ 0 & 0
      \end{pmatrix},
      \begin{pmatrix}
        0 & 1 \\ 1 & 0
      \end{pmatrix}
    \right\},
    \qquad
    \left(\Ran\cH_{x}\right)^\perp =
    \Span\left\{
      \begin{pmatrix}
        0&0\\0&1
      \end{pmatrix}
    \right\}.
  \end{equation*}
  Regularity condition~\eqref{eq:regular_condition} is violated and so
  is criticality condition~\eqref{eq:ndccn}.  However, the constant
  multiplicity stratum $S$ is well-defined: it is the isolated point
  $\{x\}$.  As can be seen in Figure~\ref{fig:deg}, we have both
  behaviors (regular and critical) at once: the lower eigenvalue has a
  topologically regular point at $x$ while the upper has a
  topologically critical point there.
  \begin{figure}
    \centering
    \includegraphics[scale=0.19]{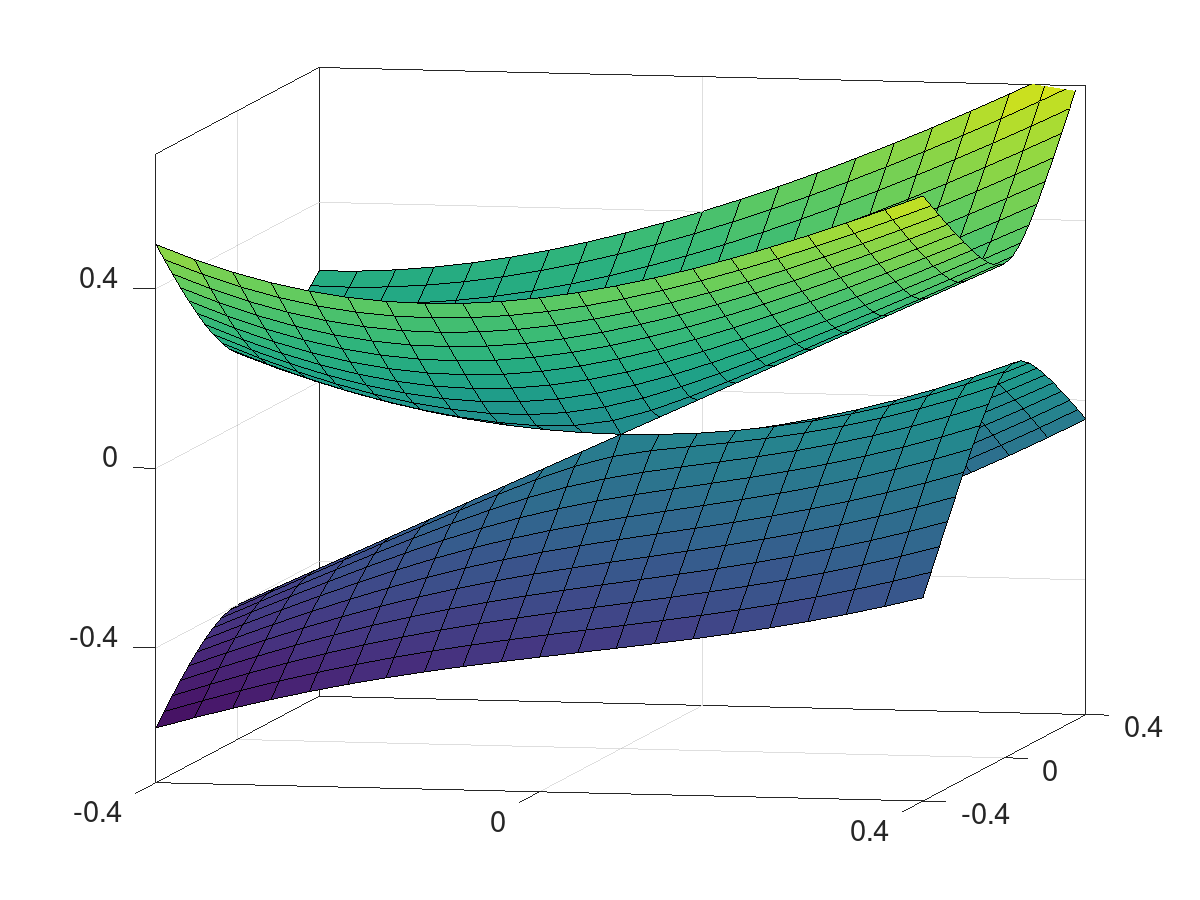}
    \includegraphics[scale=0.19]{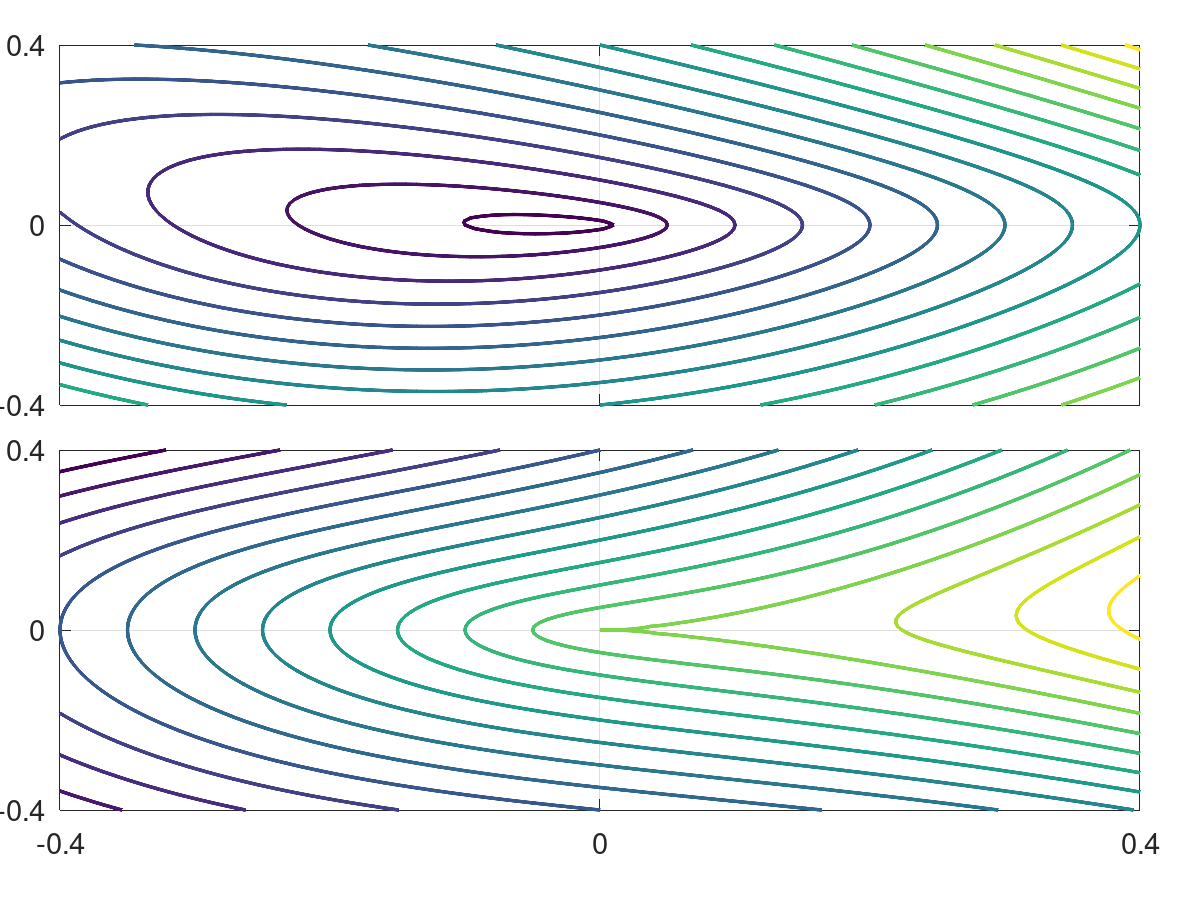}
    \caption{Eigenvalue surfaces (left) and contours for the family
      $\F$ from equation~\eqref{eq:borderline_family}.  The point
      $x=(0,0)$ is topologically regular for the bottom eigenvalue and topologically critical
      (non-smooth minimum) for the top one.}
    \label{fig:deg}
  \end{figure}
\end{example}

\begin{example}
  \label{ex:PosCone}
  We now explore in detail the regularity and criticality
  conditions of Theorems~\ref{thm:regular} and \ref{thm:critical} for
  families of $2\times2$ matrices.
  We parametrize $\Sym_2(\R)$ using $\R^3$ via the mapping
  \begin{equation}
    \label{eq:Sym2_param}
    (x,y,z) \mapsto
    \begin{pmatrix}
      x+y & z \\ z & x-y
    \end{pmatrix}.
  \end{equation}
  In this parameterization, the Frobenius inner product
    (normalized by $1/2$) becomes the Euclidean inner product, making
    orthogonality visual.  In the $(x,y,z)$ space, the sign definite
  matrices form the interior of the cone, $x^2 < y^2+z^2$.

  \begin{figure}
    \centering
    \includegraphics[scale=0.55]{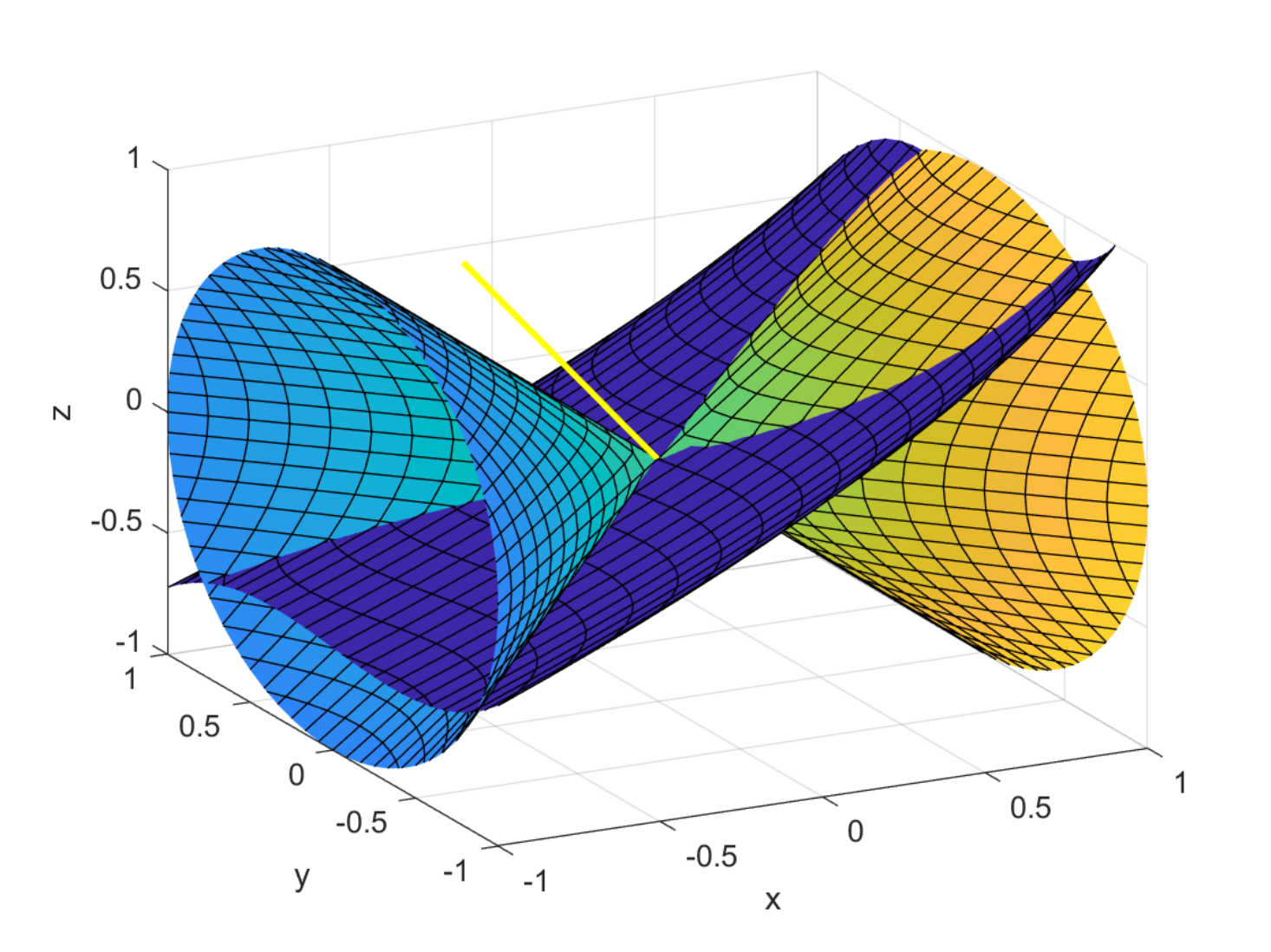}
    \includegraphics[scale=0.55]{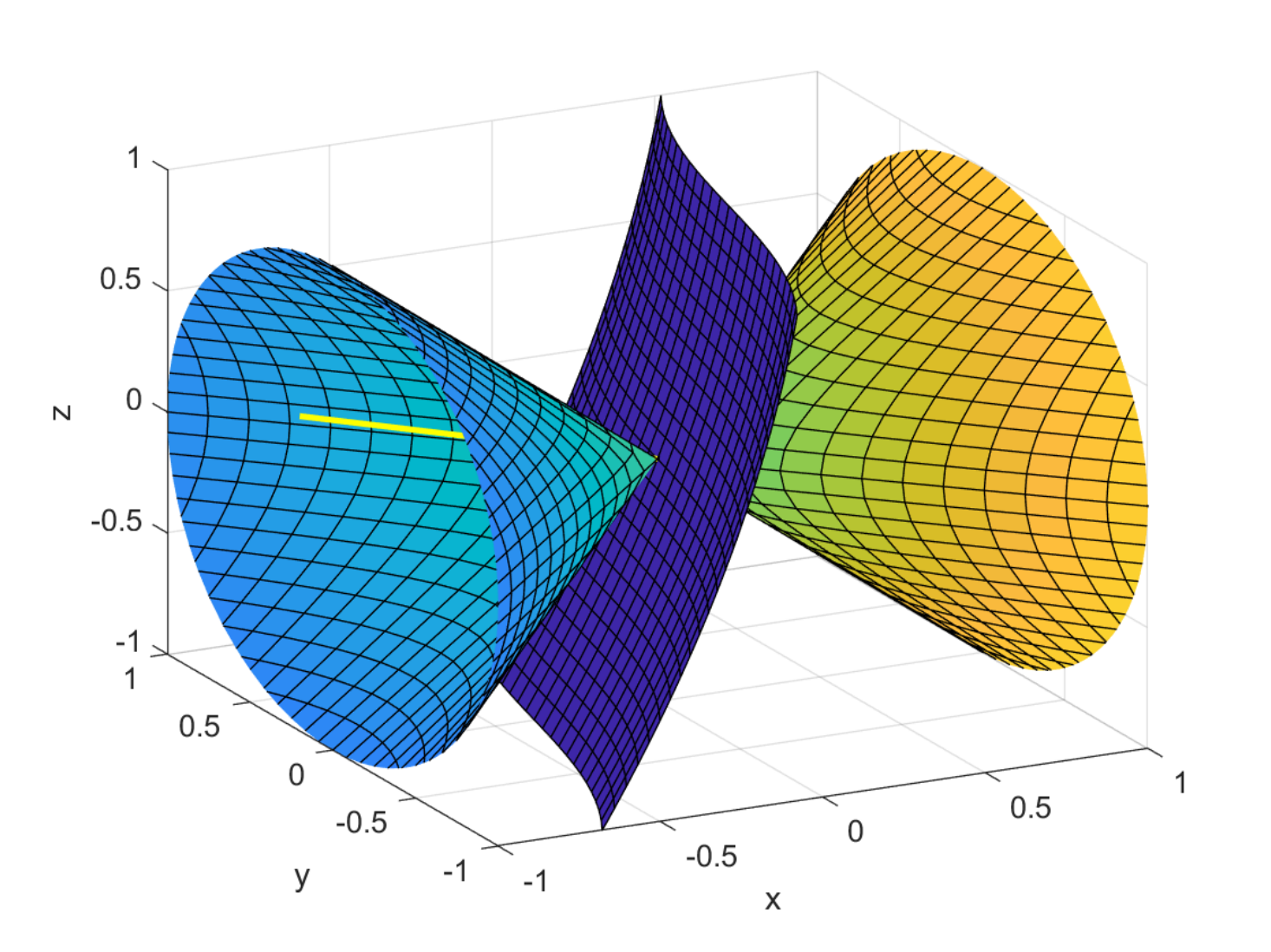}
    \caption{The cone whose interior consists of sign definite
      $2\times2$ matrices is visualized in the 3-dimensional space
      parametrizing $\Sym_2(\R)$ via \eqref{eq:Sym2_param}.  The
      monochrome surfaces represent the images of the families $\F$
      satisfying (on the left) the regularity condition and (on the
      right) criticality condition (N) at $\F(0)=0$. The normal to the
      surface, representing a matrix in
      $\left(\Ran\cH_{0}\right)^\perp$, lies outside $\Sym^{+}_2$ in
      the left graph and inside $\Sym^{++}_2$ (up to an overall sign)
      in the right graph.}
    \label{fig:PosCone}
  \end{figure}

  Assume that $\F$ depends on two parameters and satisfies $\F(0)=0$,
  with $x=0$ being the only point of multiplicity 2.
  Figure~\ref{fig:PosCone} shows two such families.  The regularity
  condition of Theorem~\ref{thm:regular} is equivalent to the tangent
  space at $0$ to the image of $\F$ intersecting the sign definite
  cone, Figure~\ref{fig:PosCone} (left).  Similarly, criticality
  condition (N), is equivalent to the tangent space to the image of
  $\F$ having dimension 2 and lying outside the cone, which puts the
  normal to $\F$ at $0$ inside the cone, see Figure~\ref{fig:PosCone}
  (right).
\end{example}

\subsection{Some applications}
\label{sec:applications}

Now we give some consequences of our main Theorems.  We start with the
observation that a maximum of an eigenvalue $\lambda_k$ cannot occur
at a point of multiplicity where $\lambda_k$ coincides with an
eigenvalue below it (the proof of Corollary~\ref{extremum_criteria} is
given at the end of section \ref{sec:main_theorem_proofs}).

\begin{corollary}
  \label{extremum_criteria}
  Let $x$ be a non-degenerate topologically critical point of
  the eigenvalue $\lambda_k$ of a generalized Morse
  family (generic by Theorem~\ref{thm:gen_Morse_generic_intro})
  $\F:M \rightarrow \mathrm{Sym}_n(\FF)$, where $\FF = \R$ or
  $\C$.  Then $x$ is a local maximum (resp.\ minimum) of
  $\lambda_k$ if and only if the following two conditions hold
  simultaneously:
  \begin{enumerate}
  \item \label{item:extremum_branch} the eigenvalue $\lambda_k$ is
    the bottom (resp.\ top) eigenvalue among those coinciding with
    $\lambda_k(x)$ at $x$; equivalently, the relative index
    $i(x, k) = \nu(x)$ (resp.\ $i(x) = 1$).
  \item \label{item:extremum_smooth} the restriction of $\lambda_k$ to
    the local constant multiplicity stratum attached to $x$ has a local
    maximum (resp.\ minimum) at $x$.
  \end{enumerate}
  Consequently, for a  generalized Morse family $\F$ over a compact manifold $M$,
  we have the strict inequalities
  \begin{align}
    \label{eq:global_max}
    &\max_{x\in M} \lambda_{k-1}(x) < \max_{x\in M} \lambda_{k}(x),
      \qquad k=2,\ldots n; \\
    \label{eq:global_min}
    &\min_{x\in M} \lambda_{k-1}(x) < \min_{x\in M} \lambda_{k}(x),
    \qquad k=2,\ldots n.
  \end{align}
\end{corollary}	

Similarly\footnote{The similarity is natural since our
  Theorem~\ref{thm:critical}
  reproduces the classical Morse inequalities if one sets $n=1$.} to
the classical Morse theory,
Theorem~\ref{thm:critical} can be
used to obtain lower bounds on the number of critical points of a
particular type, smooth or non-smooth.  Our particular example is
motivated by condensed matter physics, where the density of states
(either quantum or vibrational) of a periodic structure has
singularities caused by critical points \cite{Mon_jcp47,Smo_ppsl52} in
the ``dispersion relation'' --- the eigenvalue spectrum as a function
of the wave vector ranging over the reciprocal space.  Van Hove
\cite{vHo_prl53} classified the singularities (which are now known as
``Van Hove singularities'') and pointed out that they are unavoidably
present due to Morse theory applied to the reciprocal space, which is
a torus due to periodicity of the structure.

Of primary interest is to estimate the number of \emph{smooth}
critical points which produce stronger singularities.  Below we make
the results of \cite{vHo_prl53} rigorous, sharpening the estimates in
$d=3$ dimensions.  We also mention that higher dimensions, now open to
analysis using Theorem~\ref{thm:critical} are not a mere mathematical
curiosity: they are accessible to physics experiments through
techniques such as periodic forcing or synthetic dimensions
\cite{Pri_pt22}.

\begin{corollary}  
  \label{cor:torus23} 
  Assume that $M$ is the torus $\mathbb T^d$,
  $d=2$ or $3$. Let $\F: M \rightarrow \Sym_n$ be a generalized Morse
  family (generic by Theorem~\ref{thm:gen_Morse_generic_intro}).  Then
  the number $c_\mu(k)$ of smooth critical points of $\lambda_k$ of
  Morse index $\mu = 0,\ldots,d$ satisfied the following lower bounds.
  \begin{enumerate}
  \item In $d=2$ any ordered eigenvalue has at least
    two smooth saddle points, i.e.
    \begin{equation}
      \label{eq:d2m1}
      c_1(k) \geq 2, \qquad k=1,\ldots,n.
    \end{equation}
  \item In $d=3$,
    \begin{align}
      \label{eq:d3k1}
      &c_1(1) \geq 3, \\
      \label{eq:d3genk}
      &c_1(k) + c_2(k-1) \geq 4, \qquad k=2,\ldots,n, \\
      \label{eq:d3kn}
      &c_2(n) \geq 3.
    \end{align}
  \end{enumerate}
\end{corollary}

\begin{remark}
  Only the simpler estimates \eqref{eq:d3k1} and \eqref{eq:d3kn} for
  the bottom and top eigenvalue appear in \cite{vHo_prl53} for
  $d=3$; the guaranteed existence of smooth critical points in the
  intermediate eigenvalues \eqref{eq:d3genk} is a new result.  The
  intuition behind this result is as follows: when a point of
  eigenvalue multiplicity affects the count of smooth critical points
  of $\lambda_k$, it also affects the count of smooth critical point
  for neighboring ordered eigenvalues, such as $\lambda_{k-1}$, and it does so in
  a strictly controllable fashion since the Morse data depends little
  on the particulars of the family $\F$.  Carefully tracing these
  contributions across different ordered eigenvalues leads to sharper estimates.
\end{remark}

\begin{proof}[Proof of Corollary~\ref{cor:torus23}]
  It follows from
  Proposition~\ref{prop:const_mult_stratum}(\ref{item:codimensionS})
  that the maximal multiplicity of the eigenvalue is 2 (otherwise the
  codimension of $S$ is larger than the dimension $d=2$ or $3$ of the
  manifold).
  
  In the case $d=2$, the non-smooth critical points are
  isolated.  According to the first row of Table~\ref{tab:Morse_poly},
  such points do not contribute any $t^1$ terms.  Therefore, the
  coefficient of $t$ in $P_{\lambda_k}$ is $c_1(k)$ and, by Morse
  inequalities \eqref{eq:Morse_inequalities}, it is greater or equal
  than the first Betti number of $\mathbb{T}^2$, which is 2.

  In the case $d=3$ we need a more detailed analysis of the Morse
  inequalities \eqref{eq:Morse_inequalities} for $\lambda_k$.  
  We
  write them as
  \begin{equation*}
  	\sum_{p=0}^3 \big(c_p(k)+d_p(k)\big)t^p
  	= (1+t)^3 + (1+t)\big(\alpha_0(k) + \alpha_1(k)t + \alpha_2(k)t^2\big),
  \end{equation*}
  where $d_p(k)$ is the contribution to the polynomial $P_{\lambda_k}$
  coming from the points of multiplicity 2, $(1+t)^3$ is the Poincar\'e
  polynomial of $\mathbb{T}^3$, and where $\alpha_p(k)$ are the
  nonnegative coefficients of the remainder term $R(t)$ in
  \eqref{eq:Morse_inequalities}. Then
  similarly to \eqref{eq:Morse_ineq_full}, we have
  \begin{align}
    \label{eq:a0}
    &c_0(k) + d_0(k) = 1 + \alpha_0(k) \geq 1, \\
    \label{eq:a1}
    &c_1(k) + d_1(k) = 3 + \alpha_0(k) + \alpha_1(k) \geq 2 + c_0(k) + d_0(k), \\
    \label{eq:a1a}
    &c_2(k) + d_2(k) = 3 + \alpha_1(k) + \alpha_2(k) \geq 2 + c_3(k) + d_3(k), \\
    \label{eq:a2}
    &c_3(k) + d_3(k) = 1 + \alpha_2(k) \geq 1.
  \end{align}

  We also observe that if $\lambda_k$ has a non-smooth critical point
  $x$ counted in $d_0(k)$, then $\nu(x)=2$ with $\mu(x)=0$ and
  $i(x)=1$ (since this is the only way to obtain $t^0$ in
  \eqref{eq:Morse_contrib_nonsmooth} for $\nu=2$).  This implies that
  $\lambda_{k-1}(x) = \lambda_k(x)$ with the same constant
  multiplicity curve $S$ and the same point $x$ is a critical point of
  $\lambda_{k-1}$ with $\nu=2$, $\mu=0$ and $i=2$.  From
  Table~\ref{tab:Morse_poly} we have $P_{\lambda_{k-1}}(t;x) = t^2$,
  namely $x$ contributes to $d_2(k-1)$.  This argument can be done in
  reverse and also extended to points contributing to $d_1(k)$ (with
  $\nu=2$, $\mu=1$ and $i=1$), resulting in
  \begin{align}
    \label{eq:d_coeff_linked}
    &d_0(k) = d_2(k-1),
    \qquad
    d_1(k) = d_3(k-1),
    \qquad
      k=2,\ldots,n, \\
    \label{eq:d_coeff_boundary}
    &d_0(1)=d_1(1)=0,
    \qquad
    d_2(n)=d_3(n)=0.
  \end{align}
  The boundary values in \eqref{eq:d_coeff_boundary} are obtained by
  noting that we cannot have $\lambda_1(x) = \lambda_0(x)$ or
  $\lambda_n(x) = \lambda_{n+1}(x)$ since eigenvalues $\lambda_0$ and
  $\lambda_{n+1}$ do not exist.
  
  For $k=1$, \eqref{eq:d_coeff_boundary} substituted into
  \eqref{eq:a0} and \eqref{eq:a1} gives $c_0(1) \geq 1$ and
  $c_1(1) \geq 2+c_0(1) \geq 3$, establishing \eqref{eq:d3k1}.
  Estimate \eqref{eq:d3kn} is similarly established from
  \eqref{eq:d_coeff_boundary}, \eqref{eq:a2} and \eqref{eq:a1a}.

  Replacing $k$ with $k-1$ in estimate \eqref{eq:a1a} and using
  \eqref{eq:d_coeff_linked} gives
  \begin{equation}
    \label{eq:ineq_k-1}
    c_2(k-1) + d_0(k) \geq 2 + c_3(k-1) + d_1(k) \geq 2 + d_1(k).
  \end{equation}
  Adding this last inequality to line \eqref{eq:a1} results in
  \eqref{eq:d3genk} after cancellations and the trivial estimate
  $c_0(k)\geq 0$.
\end{proof}

\begin{remark}
  \label{rem:not_3torus}
  It is straightforward to extend \eqref{eq:d3k1}--\eqref{eq:d3kn} to
  an arbitrary compact $3$-dimensional manifold $M$ with Betti numbers
  $\beta_r$, obtaining
  \begin{align}
    \label{eq:Md3k1}
    &c_1(1) \geq \beta_1, \\
    \label{eq:Md3genk}
    &c_1(k) + c_2(k-1) \geq \beta_1+\beta_2-\beta_0-\beta_3,
      \qquad k=2,\ldots,n, \\
    \label{eq:Md3kn}
    &c_2(n) \geq \beta_2.
  \end{align}
  These inequalities extend to $d=3$ the results of
  Valero \cite{Val_jta09} who studied critical points of principal
  curvature functions (eigenvalues of the second fundamental form) of
  a smooth closed orientable surface.
\end{remark}

Independence of the transverse  Morse contributions from the
particulars of the family $\F$ also allows one to sort the terms in
the Morse polynomial.  This is illustrated by the next simple result.

Let $\mathrm{Conseq}_{k,n}$ be the set of all subsets of
$\{1,\ldots, n\}$ containing $k$ and consisting of consecutive
numbers, i.e. subsets of the form $\{j_1, j_1+1, \ldots, j_2\} \ni k$.
Given $J \in\mathrm{Conseq}_{k,n}$, let $i(k; J)$ be the sequential
number of $k$ in the set $J$ but counting from the top (cf.\
\eqref{eq:relative_index_def}).  As usual, $|J|$ will denote the
cardinality of $J$.

Let $\F:M\rightarrow \Sym_n(\mathbb F)$ be a generalized Morse family. For any 
set $J\in \mathrm {Conseq}_{k,n}$, let
\begin{equation}
  \label{eq:const_mult_stratumJ}
  S(k, J) = \big\{x\in M: \lambda_j(x)=\lambda_k(x)
  \textrm{ if and only if } j\in J\big\}.
\end{equation}
By our assumptions, $S(k, J)$ are smooth embedded
submanifolds of $M$ and the restrictions
$\lambda_k|_{S(k, J)}$ of the eigenvalue $\lambda_k$ to
$S(k, J)$ are smooth.
 
\begin{corollary}
  \label{cor:Conseq}
  Given a generalized Morse family $\F:M\rightarrow \Sym_n(\mathbb F)$
  the following inequality holds
  \begin{equation}
    \label{Morse_ineq_submanifold}
    \sum_{J\in  \mathrm {Conseq}_{k,n}}
    \fT_{|J|}^{i(k;J)}(t) P_{\lambda_k|_{S(k, J)}}(t)
    \succeq P_{\lambda_k}(t) \succeq P_M(t),
  \end{equation}
  where $P(t)\succeq Q(t)$ if and only if the all coefficients of the
  polynomials $P(t)-Q(t)$ are nonnegative, the polynomials
  $\fT_{|J|}^{i(k;J)}$ are defined in
  \eqref{eq:Morse_contrib_nonsmooth}, and $P_{\lambda_k|_{S(k,J)}}(t)$
  are the Morse polynomials of the smooth functions
  $\lambda_k|_{S(k,J)}$ on $S(k, J)$.  In particular
  $P_{\lambda_k|_{S(k,\{k\})}}$ is the total contribution of all smooth
  critical points of $\lambda_k$.
\end{corollary}

\begin{proof}
  We only need to prove the  left inequality in
  \eqref{Morse_ineq_submanifold}.  By \eqref{eq:Morse_contrib_nonsmooth},
  the contribution of a topologically critical point $x \in S(k,J)$ to
  $P_{\lambda_k}(t)$ is $t^{\mu(x)}\, \fT_{|J|}^{i(k;J)}(t)$.
  Therefore, the left-hand side of \eqref{Morse_ineq_submanifold} is
  different from $P_{\lambda_k}(t)$ in that the former also includes
  contributions from smooth critical points of $\lambda_k|_{S(k,J)}$
  that do not give rise to a topologically critical point of
  $\lambda_k$.  However, those contributions are polynomials with
  non-negative coefficients, producing the inequality.
\end{proof}

We demonstrate Corollary~\ref{cor:Conseq} in a simple example
involving an intermediate eigenvalue. Letting $n=3$, $k=2$, and using the
first two rows of Table \ref{tab:Morse_poly}, inequality
\eqref{Morse_ineq_submanifold} reads:
\begin{equation}
  \label{Morse_ineq_n=3}
  P_{\lambda_2|_{S(2, \{2\})}} (t)
  +t^2 P_{\lambda_2|_{S(2,\{2,3\})}}(t)
  +P_{\lambda_2|_{S(2,\{1,2\})}}(t)
  \succeq  P_{\lambda_2}(t)\succeq P_M(t).	
\end{equation}
Note that the term with $P_{\lambda_2|_{S(2,\{1,2,3\})}}(t)$
does not appear in \eqref{Morse_ineq_n=3} because
$\fT_3^2(t)=0$ according to the second row of Table
\ref{tab:Morse_poly}.
Further simplifications of inequalities \eqref{Morse_ineq_n=3} are
possible if it is known a priori that $\lambda$ is a perfect Morse
function when restricted to the connected components of the constant
multiplicity strata $S_{k, \{k-1,k\}}(\F)$ and $S_{k, \{k,k+1\}}(\F)$.

\subsection{An open question}
\label{sec:questions}

Finally, we mention an \emph{open question} which naturally follows
from our work: to classify Morse contributions from points where the
multiplicity $\nu$ is higher than what is suggested by the codimension
calculation in the von~Neumann--Wigner theorem \cite{vNeWig_pz29}.
Such points often arise in physical problems due to presence of a
discrete symmetry; for an example, see
\cite{FefWei_jams12,BerCom_jst18}.  At a point of ``excessive
multiplicity'', the transversality condition
\eqref{eq:transversality_def} is not satisfied because $d < s(\nu)$,
but one can still define an analogue of the ``non-degeneracy in the
non-smooth direction'' (cf.\ Remark~\ref{rem:explaining_condition1}).
It appears that the Morse indices are independent of the particulars
of the family $\F$ when the ``excess'' $s(\nu)-d$ is equal to $1$, but
whether this persists for higher values of $s(\nu)-d$ is still
unclear.


\section{Regularity condition:
  proof of Theorem~\ref{thm:regular}}
\label{sec:precritical_but_regular}

In this section we establish Theorem~\ref{thm:regular}, namely the
sufficient condition for a point to be regular (see
Definition~\ref{gencritdef}).

Recall the definition of the Clarke directional
derivative\footnote{This is usually a stepping stone to defining the
  Clarke subdifferential, but we will limit ourselves to Clarke
  directional derivative which is both simpler and sufficient for our
  needs.}  of a locally Lipschitz function $f:M\rightarrow \mathbb R$
(for details, see, for example, \cite{Cbook,MP99}).  Given
$v\in T_x M$ , let $\widehat V$ be a vector field in a neighborhood of
$x$ such that $\widehat V (x)=v$ and let $e^{t\widehat V}$ denote the
local flow generated by the vector field $\widehat V$. Then the
\term{Clarke generalized directional derivative} of $f$ at $x$ in the
direction $v$ is
\begin{equation}
  \label{derder}
  f^\circ(x, v)
  = \limsup_{\substack {y\to x \\ t\to 0+}}
  \cfrac{f(e^{t\widehat V} y) - f(y)}{t}.
\end{equation}

Independence of this definition of the choice of $\widehat V$ follows
from the flow-box theorem and the chain rule for the Clark
subdifferential, see \cite[Thm 1.2(i) and Prop 1.4(ii)]{MP99}.

\begin{definition}
  \label{def:Clarke_critical}
  The point $x$ is called a \term{critical point} of $f$ \term{in the
    Clarke sense}, if
  \begin{equation}
    \label{eq:Clarke_critical_def}
    0 \leq f^\circ(x, v)
    \quad\mbox{for all}\quad
    v\in T_x M.    
  \end{equation}
  Otherwise, the point $x$ is said to be \term{regular in the Clarke
    sense}.
\end{definition}

The assumptions of Theorem~\ref{thm:regular} will be shown to imply
that the point $x$ is regular in the Clarke sense, whereupon we
will use the following result.

\begin{theorem}\cite[Proposition 1.2]{APS97}
  \label{thm:APS_Clarke_regular}
  A point regular in the Clarke sense is  topologically regular in the sense of
  Definition~\ref{gencritdef}.
\end{theorem}


\begin{proof}[Proof of Theorem~\ref{thm:regular}]
  We first establish that condition~\eqref{eq:regular_condition},
  namely
  \begin{equation*}
    \left(\Ran\cH_{x}\right)^\perp \cap \Sym_\nu^+ = 0,
  \end{equation*}
  is equivalent to
  existence of a matrix $C \in \Ran\cH_{x}$ which is (strictly) positive
  definite.  Despite being intuitively clear, the proof of this fact
  is not immediate and we provide it for completeness; a similar
  result is known as Fundamental Theorem of Asset Pricing in
  mathematical finance \cite{Duffie_AssetPricing}.  Assume the contrary,
  \begin{equation}
    \label{eq:contrary_pos_def}
    \Ran\cH_{x} \cap \Sym_\nu^{++} = \emptyset.
  \end{equation}
  The set $\Sym_\nu^{++}$ is open and convex (the latter can be seen
  by Weyl's inequality for eigenvalues).  A suitable version of the
  Helly--Hahn--Banach separation theorem (for example, \cite[Thm
  7.7.4]{NariciBeckenstein_TVS}) implies existence of a functional
  vanishing on $\Ran\cH_{x}$ and positive on $\Sym_\nu^{++}$.  By Riesz
  Representation Theorem, this functional is
  $\langle D, \cdot \rangle$ for some $D\in \Sym_\nu$, for which we
  now have $D \in \left(\Ran\cH_{x}\right)^\perp$ and $\langle D, P \rangle > 0$ for all
  $P \in \Sym_\nu^{++}$. In particular, $D$ is non-zero and belongs to
  the dual cone of $\Sym_\nu^{++}$, namely to $\Sym_\nu^{+}$
  \cite{BoydVandenberghe_optimization}, contradicting condition~\eqref{eq:regular_condition}.

  Secondly, results of \cite[Theorem 4.2]{Cox_incol94} (see
  also\footnote{Note that there is a misprint in the direction of the
    inequality in \cite[Section 6]{HU-L99}.} \cite[Section 6]{HU-L99})
  show that
  \begin{align}
    \label{eq:cox_estimate}
    \lambda^\circ_k(x,v)
    &\leq \max\left\{
      \Big\langle u, \big(d\mathcal{F}(x)v \big)u\Big\rangle
      \colon u \in \esp_k, \|u\|=1 \right\} \\
    &= \lambda^{\max}\left(\big(d\mathcal{F}(x)v \big)_{\esp_k}
      \right)
      = \lambda^{\max}\big(\mathcal{H}_{x}(v)\big),
  \end{align}
  where $\esp_k$ is the eigenspace of the eigenvalue $\lambda_k(x)$ of
  $\mathcal{F}(x)$; the middle equality is by the variational
  characterization of the eigenvalues and the last by the definition \eqref{eq:HF_matrix}
  of $\mathcal{H}_x$.
  
  We already established that there exists $v$ such that
  $\mathcal{H}_{x}(v) \in \Sym^{++}_\nu$.  Then $\mathcal{H}_{x}(-v)$
  is negative definite and we have
  \begin{equation}
    \label{eq:clarke_result}
    \lambda^\circ_k(x,v)
    \leq \lambda^{\max}\big(\mathcal{H}_{x}(-v)\big)
    < 0.
  \end{equation}
  The point $x$ is regular in the Clarke sense and, therefore, regular
  in the sense of Definition~\ref{gencritdef}.
\end{proof}

\section{Transversality and its consequences}
\label{sec:transversality}


\begin{definition}
  \label{def:transversality}
  We say that a family $\F$ is \term{transverse (with respect to
  	eigenvalue $\lambda_k$)}   at a point $x$ if 
  \begin{equation}
    \label{eq:transversality_def}
    \mathcal{I}_\nu + \Ran \cH_{x} = \Sym_\nu,
  \end{equation}
  where $\nu$ is the multiplicity of $\lambda_k$ at the point $x$ and
  $\mathcal{I}_\nu := \Span(I_\nu) \subset \Sym_\nu$ is the space of
  multiples of the identity matrix.
\end{definition}




In this section we explore the consequences of the transversality
condition, equation~\eqref{eq:transversality_def}.  In particular, in
Lemma~\ref{lem:transversality_reduced_test} we interpret
condition~\eqref{eq:transversality_def} as transversality of the
family $\F$ and the subvariety of $\Sym_n(\mathbb F)$ of matrices with
multiplicity.  We then show that transversality at a non-degenerate topologically 
critical point allows us to work separately in the smooth and
non-smooth directions.  In particular, we establish that a
non-degenerate topologically critical point satisfies the sufficient conditions of
Goresky--MacPherson's stratified Morse theory.  The latter allows us
to separate the Morse data at a topologically critical point into a smooth part and
a transverse  part; the latter will be shown in
section~\ref{sec:sublevel_change1} to be independent of the
particulars of the family $\F$.

Let $Q^n_{k,\nu}$ be the subset of $\mathrm{Sym}_n(\mathbb F)$, where
$\mathbb F$ is $\mathbb R$ or $\mathbb C$, consisting of the matrices
whose eigenvalue $\lambda_k$ has multiplicity $\nu$.  It is well-known
\cite{Arn_fap72} that the set $Q^n_{k,\nu}$ is a semialgebraic
submanifold of $\mathrm{Sym}_n$ of codimension\footnote{The reason for
  this codimension to be equal to $\dim\Sym_\nu(\mathbb{F}) - 1$ is as
  follows: Symmetric matrices with non-repeated eigenvalues can be
  encoded by their eigenvalues and unit eigenvectors. When an
  eigenvalue is repeated $\nu$ times, there is a loss of $\nu-1$
  parameters from the eigenvalues plus an extra freedom of choice of
  an orthonormal basis in the corresponding eigenspace.  Thus the
  desired codimension is equal to $\nu-1$ plus the dimension of the
  space of orthonormal bases of $\mathbb F^\nu$, adding up to
  $\dim\Sym_\nu(\mathbb{F}) - 1$.}
\begin{equation}
  \label{eq:codim_nu_again}
  s(\nu) :=
  \dim\Sym_\nu(\mathbb{F}) - 1 = 
  \begin{cases}
    \frac12\nu(\nu+1)-1, &\mathbb F=\mathbb R,\\
    \nu^2-1, & \mathbb F=\mathbb C.
  \end{cases}
\end{equation}
In particular, if $\nu>1$ (the eigenvalue $\lambda_k$ is not simple),
then $\codim \, Q^n_{k,\nu}\geq 2$, if $\mathbb F=\mathbb R$ and
$\codim \, Q^n_{k,\nu}\geq 3$, if $\mathbb F=\mathbb C$.  We remark
that we use real dimension in all (co)dimension calculations.

\begin{lemma}
  \label{lem:transversality_reduced_test}
  Let $\F: M \to \Sym_n$ be a smooth family whose eigenvalue
  $\lambda_k$ has multiplicity $\nu$ at the point $x\in M$ (i.e.\
  $\F(x) \in Q^n_{k,\nu}$).  Then $\F$ is  transverse  at $x$ in the
  sense of \eqref{eq:transversality_def} if and only if
  \begin{equation}
    \label{eq:transversality_big}
    \Ran d\F(x) +T_{\F(x)} Q^n_{k, \nu}= T_{\F(x)} \Sym_n
    \quad (\ \cong \Sym_n\ ).
  \end{equation}
\end{lemma}

\begin{remark}
  It is easy to see that when $\nu=1$, both conditions
  \eqref{eq:transversality_def} and \eqref{eq:transversality_big} are
  satisfied independently of $\F$.  When $\nu=n$, conditions
  \eqref{eq:transversality_def} and \eqref{eq:transversality_big}
  become identical.  The transversality condition in the case
  $\nu=n=2$ is illustrated in Figure~\ref{fig:PosCone_trans}.
\end{remark}

\begin{figure}
    \centering
    \includegraphics[scale=0.55]{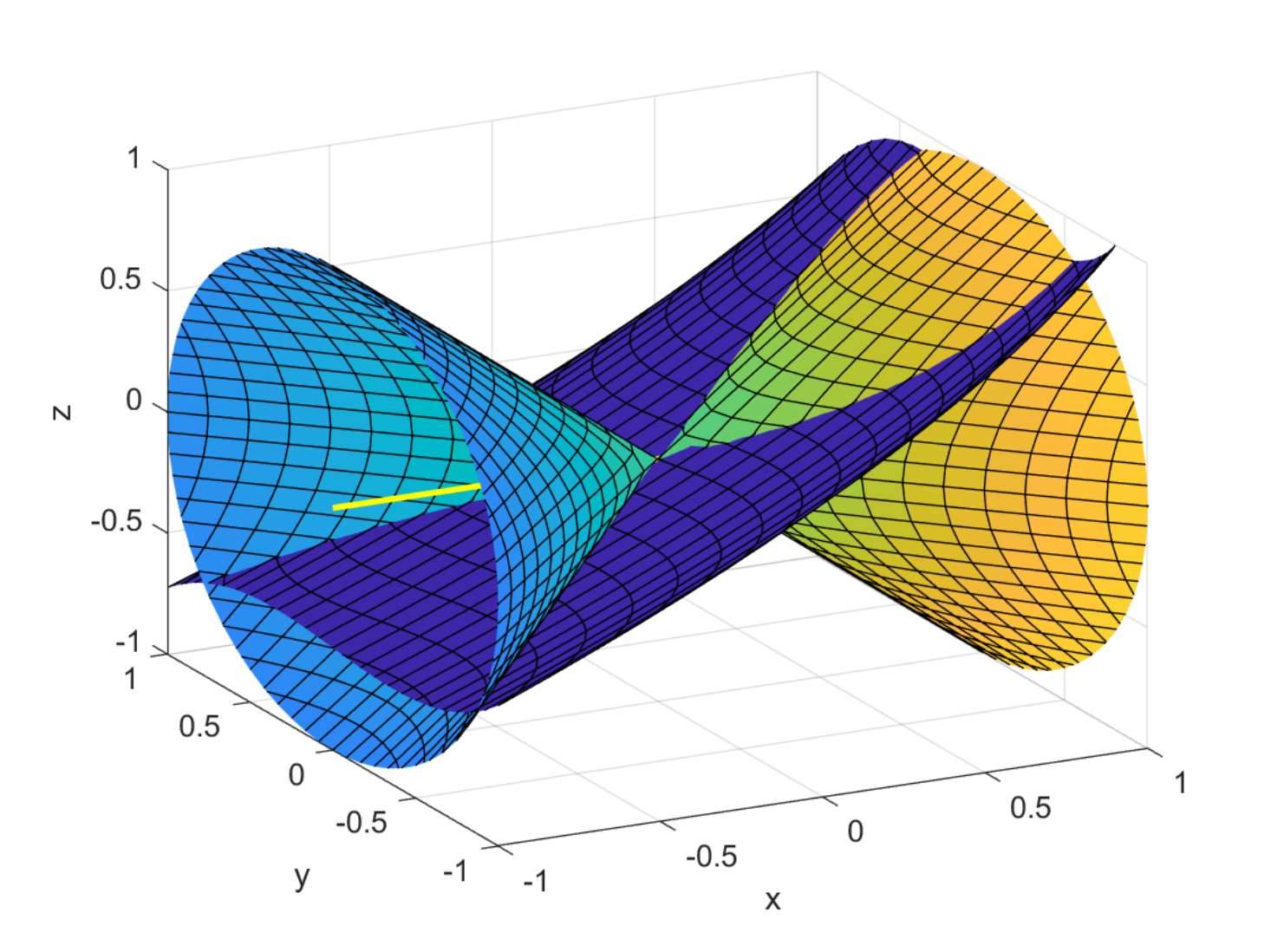}
    \includegraphics[scale=0.55]{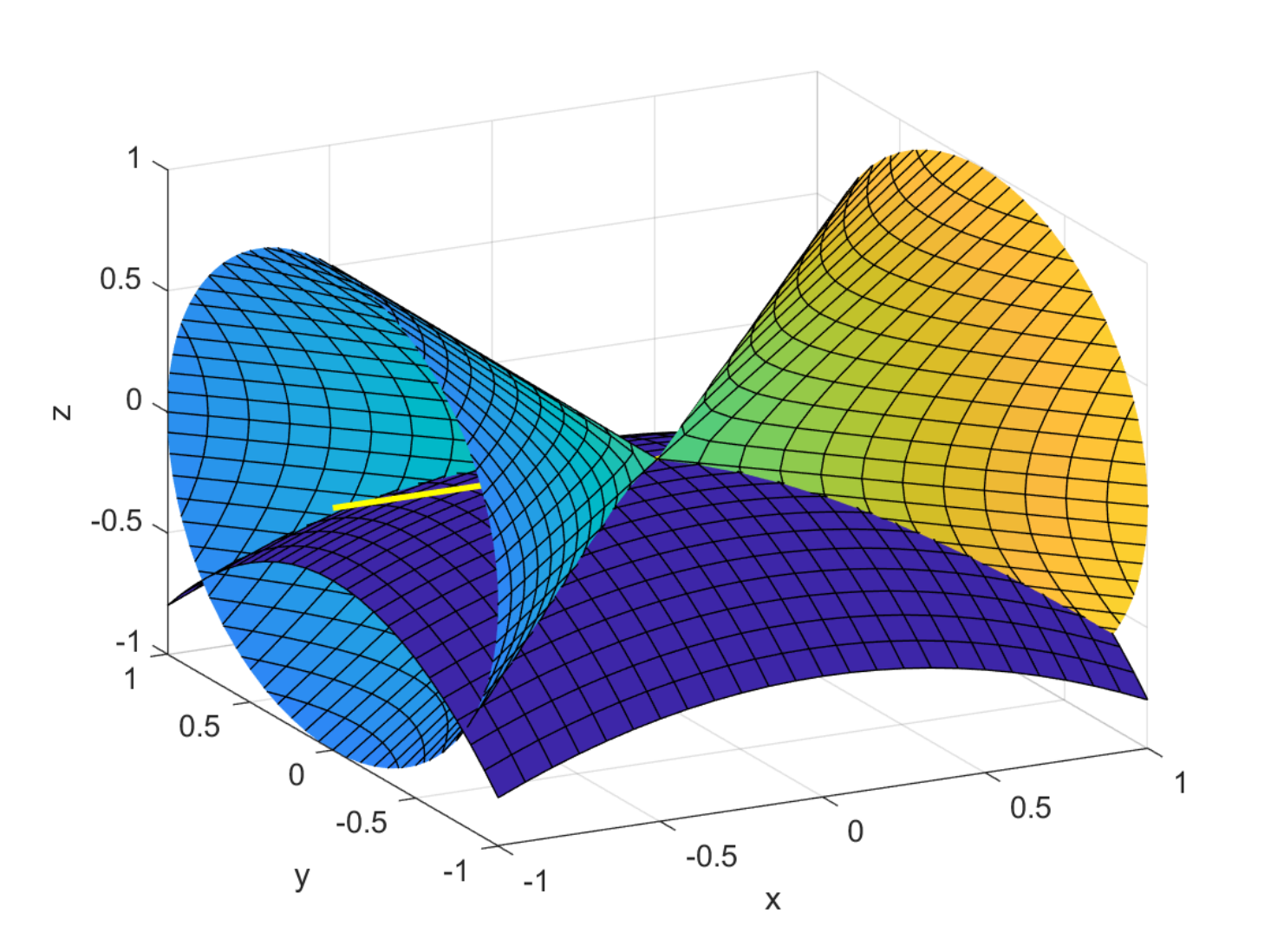}
    \caption{Examples of two families visualized in the 3-dimensional
      space parametrizing $\Sym_2(\R)$ via \eqref{eq:Sym2_param}.  The
      subspace $\mathcal{I}_2$ is drawn as a line.  The monochrome
      surfaces represent the images of the families $\F$, satisfying
      $\F(0)=0$.  The family on the left is transverse at $x=0$ and
      the family on the right is not.  The cones are drawn for
      comparison with Figure~\ref{fig:PosCone}.}
    \label{fig:PosCone_trans}
  \end{figure}

\begin{proof}[Proof of Lemma~\ref{lem:transversality_reduced_test}]
  Consider the linear mapping $h : \Sym_n \to \Sym_\nu$ acting as a
  compression to $\esp_k$.  Namely,
  $A \mapsto A_{\esp_k} = \U^* A \U$, where $\U$ is a linear isometry
  $\FF^\nu \to \esp_k$, see \eqref{eq:restriction_def}.  The mapping
  $h$ is onto: for any $B \in \Sym_\nu$, choosing
  $\widetilde B = \U B \U^* \in \Sym_n$ yields
  $h\big(\widetilde B\big) = \U^*\U B \U^*\U = I_\nu B I_\nu = B$.
  Furthermore, by Hellmann--Feynman theorem
  (Theorem~\ref{thm:hellmann-feynman}), $A \in T_{\F(x)} Q_{k,\nu}^n$
  if and only if $h(A) \in \mathcal{I}_\nu$ (informally, a direction
  is tangent to $Q_{k,\nu}^n$ if and only if the eigenvalues remain
  equal to first order).  Finally, by definition of $\cH_x$ we have
  $h\big(\Ran d\F(x)\big) = \Ran \cH_{x}$.

  Assuming condition~\eqref{eq:transversality_big} and applying to it
  the mapping $h$, we get
  \begin{equation*}
    \Sym_\nu = h\big(\Sym_n\big)
    =h\big(\Ran d\F(x)\big) + h\Big(T_{\F(x)}
    Q_{k,\nu}^n\Big)
    = \Ran \cH_{x} + \mathcal{I}_\nu,
  \end{equation*}
  establishing \eqref{eq:transversality_def}.  Conversely, assume
  $\mathcal{F}$ violates condition~\eqref{eq:transversality_big},
  meaning that
  \begin{equation*}
    \Ran d\F(x) + T_{\F(x)} Q^n_{k, \nu} = G + T_{\F(x)} Q^n_{k, \nu}
  \end{equation*}
  for some linear subspace $G$ of $\dim G < \codim T_{\F(x)} Q^n_{k,
    \nu} = s(\nu)$.  Applying $h$ to both sides we get
  \begin{equation*}
    \Ran \mathcal{H}_{x} + \mathcal{I}_\nu
    = h\left(\Ran d\F(x) + T_{\F(x)} Q^n_{k, \nu}\right)
    = h\left(G + T_{\F(x)} Q^n_{k, \nu}\right)
    = h(G) + \mathcal{I}_\nu.
  \end{equation*}
  Counting dimensions, we arrive to $\dim\left(\Ran \mathcal{H}_{x} +
    \mathcal{I}_\nu \right) < s(\nu) + 1 = \dim(\Sym_\nu)$, and therefore
  \eqref{eq:transversality_def} cannot hold.
\end{proof}

\begin{corollary}
  \label{cor:const_deg_manifold}
  ~
  \begin{enumerate}
  \item If $x$ satisfies non-degenerate criticality condition (N), see
    Definition~\ref{def:nondeg_crit_N}, then $\F$ is  transverse  at
    $x$.
  \item \label{item:S_well_def} If $\F$ is  transverse  at a point $x$,
    the constant multiplicity stratum $S$ of $x$ is a submanifold of
    $M$ of codimension $s(\nu)$ and the function $\lambda_k$ restricted
    to $S$ is smooth.
  \end{enumerate}
\end{corollary}

\begin{proof}
  Recall that non-degenerate criticality condition (N) states that
  $\left(\Ran\cH_{x}\right)^\perp$ is spanned by a positive definite
  matrix.  In particular, the codimension of $\Ran \cH_{x}$ is 1.
  Furthermore, the identity matrix $I_\nu$ is not in $\Ran \cH_{x}$
  because the identity cannot be orthogonal to a positive definite
  matrix.  Therefore condition \eqref{eq:transversality_def} holds.

  Now let $\F$ be transverse  at $x$ and
  let $\nu=\nu(x)$ be the multiplicity of the eigenvalue $\lambda_k$
  at $x$.  Then $S$ is the connected component of
  $\F^{-1}\left(Q^n_{k, \nu}\right)$ containing $x$.  Transversality
  implies $S$ is a submanifold of codimension
  $\codim Q^n_{k, \nu} = s(\nu)$.  The smoothness of $\lambda_k$
  restricted to $S$ is a standard result of perturbation theory for
  linear operators (see, for example, \cite[Section II.1.4 or Theorem
  II.5.4]{Kato}).  To see it, one uses the Cauchy integral formula for
  the \term{total eigenprojector} (or \term{Riesz projector}),
  i.e. the projector onto the span of eigenspaces of the eigenvalues
  lying in a small neighborhood around $\lambda_k(x)$.  It
    follows that the total eigenprojector is smooth in a sufficiently
    small neighborhood of $x\in M$ (the neighborhood on $M$ needs to
    be small enough so that no eigenvalues cross the contour of
    integration).  Once restricted to $y\in S$, the eigenprojector is
  simply $\lambda_k(y) I_\nu$, therefore $\lambda_k(y)$ is also
  smooth.
\end{proof}

\begin{lemma}
  \label{lem:precriticality_condition}
  If $x$ satisfies non-degenerate criticality condition (N), then
  $\Ran \mathcal{H}_{x}\big|_{T_{x}S} = 0$ and $x$ is a critical point
  of the locally smooth function $\lambda_k\big|_S$.
\end{lemma}

\begin{proof}
  By Hellmann--Feynman theorem, see
  Appendix~\ref{sec:hellmann_feynman}, the eigenvalues of
  $\cH_{x}v\in \Sym_\nu$ give the slopes of the eigenvalues
  splitting off from the multiple eigenvalue
  $\lambda_k\big(\mathcal{F}(x)\big)$ when we leave $x$ in the
  direction $v$.  Leaving in the direction $v \in T_{x} S$, where $S$
  is the constant multiplicity stratum attached to $x$, must produce
  equal slopes, i.e.\ $\cH_x(v)$ is a multiple of the identity matrix
  $I_\nu$ for every $v\in T_{x}S$.  But a non-zero multiple of the
  identity cannot be orthogonal to $B\in\Sym^{++}_\nu$, therefore
  $\Ran \mathcal{H}_{x}\big|_{T_{x}S} = 0$.  In other words, the
  slopes of the branches splitting off from the multiple eigenvalue
  $\lambda_k(x)$ are all zero.
\end{proof}

\begin{proof}[Proof of Proposition~\ref{prop:const_mult_stratum}]
  \label{proof:prop}
  Corollary~\ref{cor:const_deg_manifold} and
  Lemma~\ref{lem:precriticality_condition}, combined, give the
  conclusions of Proposition~\ref{prop:const_mult_stratum}.
\end{proof}

The next step is to enable ourselves to focus on the directions
transverse to $S$.

\begin{corollary}
  \label{cor:transversal_cut_of_CP}
  Let $\F: M \to \Sym_n$ be a smooth family whose eigenvalue
  $\lambda_k$ satisfies conditions (N) and (S)
  (Definitions~\ref{def:nondeg_crit_N} and \ref{def:nondeg_crit_S}) at
  the point $x\in M$.  Let $S$ be the constant multiplicity stratum at
  $x$ and let $N$ be a submanifold of $M$ of dimension
  $\dim N = \codim_M S = s(\nu)$ which intersects $S$ transversely at
  $x$.

  Then the eigenvalue $\lambda_k$ of the restriction $\F \big|_N$ also
  satisfies conditions (N) and (S).
\end{corollary}

\begin{proof}
  Transversality and dimension count imply that the constant
  multiplicity stratum of $\F\big|_N$ is the isolated point $x$.
  Therefore condition (S) for $\F\big|_N$ is vacuously true at $x$.

  Condition (N) for $\F$, combined with 
  Lemma~\ref{lem:precriticality_condition}, yields
  $\Ran\mathcal{H}_{x}\big|_{T_{x}S} = 0$.  We thus obtain
  \begin{equation}
    \label{eq:reduced_ranH}
    \Ran \mathcal{H}_{x}
    = \Ran \mathcal{H}_{x}\big|_{T_{x}N}
    + \Ran \mathcal{H}_{x}\big|_{T_{x}S}
    = \Ran \mathcal{H}_{x}\big|_{T_{x}N}.
  \end{equation}
  In other words, the space $\Ran \mathcal{H}_{x}$ remains unchanged after
  restricting $\F$ to $N$, therefore condition (N) holds for $\F\big|_N$.
\end{proof}

The next step is to separate the Morse data at a critical point $x$
into a smooth part (along $S$) and a  transverse   part (along $N$).
For this purpose we will use the stratified Morse theory of Goresky
and MacPherson \cite{GM88}.  We now show that a point satisfying
non-degenerate criticality conditions (N) and (S) is nondepraved in
the sense of \cite[definition in Sec.~I.2.3]{GM88}.  The setting of
\cite{GM88} calls for a smooth function on a certain manifold which is
then restricted to a stratified subspace of that manifold.  To that
end we consider the graph of the function $\lambda_k$ on $M$, i.e.\
the set $Z_k := \big\{\big(x, \lambda_k(x)\big)\colon x\in M\big\}$ as
a stratified subspace of $\tM := M\times \R$ and the (smooth) function
$\pi: \tM \to \R$ which is the projection to the second component of
$\tM$.  As before, the stratification (both on $M$ and on $Z_k$) is
induced by the multiplicity of the eigenvalue $\lambda_k(x)$.

Recall that a subspace $Q$ of $T_z \tM$ is called a
\term{generalized tangent space} to a stratified subspace
$Z\subset \tM$ at the point $z\in Z$, if there exists a stratum
$\mathcal R$ of $Z$ with $z\in \overline{\mathcal R}$, and a sequence
of points $\{z_i\} \subset \mathcal R$ converging to $z$ such that
\begin{equation}
  \label{gentangent}
  Q = \lim_{i\to\infty} T_{z_i} \mathcal{R}.
\end{equation}

\begin{proposition}
  \label{isoprop} 
  Let the family $\F: M \to \Sym_n$ be  transverse  and the point
  $x\in M$ satisfy conditions of Theorem~\ref{thm:critical}.  Let
  $z := \big(x, \lambda_k(x)\big)$ be the corresponding point on the
  stratified subspace $Z_k \subset \tM$ defined above and
  $\widetilde{S}$ be the stratum of $Z_k$ containing $z$.  Then the
  following two statements hold:
  \begin{enumerate}
  \item \label{item:generalized_subspace} For each generalized tangent
    space $Q$ at $z$ we have $d\pi(z)\big|_Q \neq 0$ except when
    $Q=T_{z}\widetilde S$.
  \item \label{item:isolated_point} $x$ is isolated in the set of
    all points $y$ that are critical for $\lambda_k\big|_{S_y}$,
    where $S_y$ is the constant multiplicity stratum attached to $y$.
  \end{enumerate}
\end{proposition} 

\begin{remark}
  \label{rem:third_condition_nondepraved}
  A point $z := \big(x, \lambda_k(x)\big)$ satisfying conditions
  (\ref{item:generalized_subspace})-(\ref{item:isolated_point}) of
  Proposition~\ref{isoprop} and such that $x$ is non-degenerate as a
  smooth critical point of $\lambda_k\big|_S$ is a nondepraved point
  of the map $\pi|_{Z_k}$ in the sense of Goresky--MacPherson
  \cite[Sec.~I.2.3]{GM88}.\footnote{The definition of a
      \term{nondepraved point} in \cite[Sec.~I.2.3]{GM88} contains
      three conditions.  Conditions (c) and (a) of
      \cite[Sec.~I.2.3]{GM88} correspond to parts
      (\ref{item:generalized_subspace}) and
      (\ref{item:isolated_point}) of Proposition~\ref{isoprop},
      respectively.  The third condition --- condition (b) of
      \cite[Sec.~I.2.3]{GM88} --- holds automatically in our case
      because $x$ is non-degenerate as a smooth critical point of
      $\lambda_k\big|_S$, by condition (S) assumed in
      Theorem~\ref{thm:critical}).  Thus we omit here the general
      description of condition (b), which is rather
      technical.}$^{,}$\footnote{We also mention that \cite{GM88} uses the
    term ``critical'' for the points $y$ that are critical when the
    function in question is restricted to their respective stratum of constant multiplicity.}
\end{remark}

\begin{corollary}
	\label{isopropcor} 
	Let the family $\F: M \to \Sym_n$ be transverse
         and the point $x\in M$ satisfy
        conditions of item (1) of Theorem~\ref{thm:critical}.
        Let $z := \big(x, \lambda_k(x)\big)$ be the corresponding
        point on the stratified subspace $Z_k \subset \tM$ defined
        above and $\pi: M\times \R \to \R$ be the projection to the
        second component of $\tM$. Then $z$ is a nondepraved point of
        the map $\pi|_{Z_k}$ in the sense of Goresky--MacPherson
        \cite[Sec.~I.2.3]{GM88}.
\end{corollary}

\begin{remark}
  Let us discuss informally the idea behind
  part~(\ref{item:generalized_subspace}) of Proposition~\ref{isoprop},
  the proof of which is fairly technical.  When we leave $x$ in the
  direction not tangent to $S$, the multiplicity of eigenvalue
  $\lambda_k$ is reduced as other eigenvalues split off.
  Part~(\ref{item:generalized_subspace}) stipulates that among the
  directions in which the multiplicity splits in a prescribed manner,
  there is at least one direction in which the slope of $\lambda_k$ is
  not equal to zero.  This is again a consequence of transversality:
  the space of directions is too rich to produce only zero slopes.
\end{remark}

\newcommand{\nuR}{\nu_{\mathcal{R}}}

\begin{proof}[Proof of Proposition~\ref{isoprop}, part
  \eqref{item:generalized_subspace}]
  First,  if $p: M\times \R \to M$ denotes the projection to the
  first component of $\tM = M \times \R$, then $dp(\tx):
  T_{\tx}\tM \to T_x M$ is the corresponding projection to
  the first component of $T_{\tx}\tM \cong T_{x} M \times \R$; here
  $\tx = (x,\lambda)$ for some $\lambda \in \R$.

  Since
  $d\pi\big|_{T_{z}\widetilde S} = d\Big(\lambda_k\big|_S\Big) \circ
  dp\big|_{T_{z}\widetilde S}$, where $S = p(\widetilde{S})$ is the
  constant multiplicity stratum of $x$, we conclude from
  Lemma~\ref{lem:precriticality_condition} that $d\pi(z)\big|_Q = 0$
  when $Q=T_{z}\widetilde S$.

  Let now $Q\neq T_{z}\widetilde S$ and assume that
  \begin{equation}
    \label{eq:degen_covec}
    d\pi\left(z\right)\big|_Q = 0.
  \end{equation}
  Let $\nu$ be the multiplicity of the eigenvalue $\lambda_k(x)$ of
  $\F(x)$ and $\esp_k$, $\dim\esp_k = \nu$, be the
  corresponding eigenspace.  Let $\mathcal{R}$ be the stratum used for
  the definition of $Q$ in \eqref{gentangent} and $\nuR<\nu$ be the
  multiplicity of $\lambda_k$ on $p(\mathcal{R})$.

  Let $(z_i) \subset \mathcal{R}$ be the sequence defining $Q$ and let
  $x_i = p(z_i)$.  Let $\esp_k(x_i) \subset \FF^n$ denote the
  $\nuR$-dimensional eigenspace of the eigenvalue $\lambda_k$ of
  $\F(x_i)$ and let $\U_i$ be a choice of linear isometry from $\FF^{\nuR}$ to
  $\esp_k(x_i)$.  Finally, let $W_i \subset T_{x_i} M$ denote the first
  component of the tangent space at $z_i$ to $\mathcal{R}$, namely
  $W_i = dp(z_i)\big(T_{z_i} \mathcal{R}\big)$.

  We would like to use Hellmann--Feynman theorem at $x_i$.  In the
  directions from $W_i$, the eigenvalue $\lambda_k$ retains
  multiplicity $\nu_\mathcal{R}$ \emph{in the linear approximation}.
  In other words, directional derivatives of the eigenvalue group of
  $\lambda_k$ are all equal.  Formally,
  \begin{equation}
    \label{eq:HF_at_xi}
    \U_i^* \big(d\F(x_i) w\big) \U_i = D_w \lambda_k(x_i) \, I_{\nuR},
    \qquad 
    \text{for all } w \in W_i;
  \end{equation}
  here $D_w\lambda_k$ is the directional derivative of $\lambda_k$.
  This expression is invariant with respect to the choice of isometry
  $\U_i$.

  Using compactness of the Grassmannians and, if necessary, passing to
  a subsequence, the spaces $\esp_k(x_i)$ converge to a subspace
  $\esp^{\mathcal{R}}_k$ of the $\nu$-dimensional eigenspace $\esp_k$
  of the matrix $\F(x)$.  The isometries $\U_i$ (adjusted if
  necessary) converge to a linear isometry $\U_{\mathcal{R}}$ from
  $\FF^{\nuR}$ to $\esp^{\mathcal{R}}_k$.  Tangent subspaces $W_i$
  also converge to the subspace $W_0 := dp(z)Q$.  Passing to the
  limit in \eqref{eq:HF_at_xi}, the derivative on the right-hand side
  of \eqref{eq:HF_at_xi} must tend to 0 due to \eqref{eq:degen_covec}.
  Recalling the definition of $\cH$ in \eqref{eq:HF_matrix}, we get
  \begin{equation}
    \label{eq:HF_limiting}
    \U_{\mathcal{R}}^* \big(d\F(x) w\big) \U_{\mathcal{R}}
    = \U_{\mathcal{R}}^* \U \big(\cH_{x} w\big) \U^* \U_{\mathcal{R}} = 0,
        \qquad 
    \text{for all } w \in W_0.
  \end{equation}
  In other words, the matrix $\cH_{x}w$ with $w$ restricted to $W_0$
  maps vectors from $V = \Ran(\U^*\U_{\mathcal{R}}) \subset \FF^{\nu}$ to vectors
  orthogonal to $V$.    We can express this as
  \begin{equation}
    \label{eq:Range_H}
    \Ran \cH_{x}\big|_{W_0}
    \subset \Sym_{\nu} \left(V, V^\perp\right),
  \end{equation}
  where $\Sym_\nu(X, Y)$ denotes the set of all $\nu\times\nu$
  self-adjoint matrices that map $X$ to $Y$.  The space $V$ is
  $\nuR$-dimensional\footnote{From properties of isometries and the
    inclusion $\esp^{\mathcal{R}}_k \subset \esp_k$ it can be seen
    that $(\U^*\U_{\mathcal{R}})^* \U^* \U_{\mathcal{R}} = I_{\nuR}$.}
  and, in a suitable choice of basis, a $\nuR \times \nuR$ subblock of
  $\cH_{x}w$ is identically zero.  Therefore, the dimension of
  $\Sym_{\nu} \left(V, V^\perp\right)$ is
  \begin{equation}
    \label{eq:dim_Sym_VV}
    \dim \Sym_{\nu} \left(V, V^\perp\right)
    = \dim \Sym_{\nu} - \dim \Sym_{\nuR}
    = s(\nu) - s(\nuR).
  \end{equation}

  On the other hand, we have the following equalities,
  \begin{equation}
    \label{eq:KerH_RanH}
    \codim \Ker \cH_{x}
    = \dim \Ran \cH_{x}
    = \dim \Sym_\nu -1
    = s(\nu)
    = \codim T_{x} S.
  \end{equation}
  The first is the rank-nullity theorem, the second is because
  $\Ran \cH_{x}$ has codimension 1 (by condition (N)), the third is
  the definition of $s(\nu)$ and the last is from the properties of
  $S$.  Using $T_{x}S\subset\Ker \cH_{x}$
  (Lemma~\ref{lem:precriticality_condition}) and counting dimensions,
  we conclude
  \begin{equation}
    \label{eq:KerH_TS}
    \Ker \cH_{x} = T_{x}S.
  \end{equation}

  Now we want to show that the stratification on $M$ induced by the
  multiplicity of the eigenvalue $\lambda_k(x)$ satisfies
  \term{Whitney condition A}: any generalized tangent space at
  $z$ contains the tangent space of the stratum containing $z$.
  For this note that the \term{discriminant variety} of
  $\Sym_n$,
  \begin{equation}
    \label{discriminant}
    \mathrm{Discr}_n:=\bigcup_{1\leq k\leq n,\ \nu>1} Q^n_{k,\nu},
  \end{equation}
  is an algebraic variety. Therefore, by classical results of Whitney
  \cite{Whitney65}, $\mathrm{Discr}_n$ admits a stratification
  satisfying Whitney condition A. Consequently, if $\F: M \to \Sym_n$
  is a  transverse  family then the fact that $\mathrm{Discr}_n$
  satisfies Whitney condition A implies that the stratifications on
  $M$ induced by the multiplicity of the eigenvalue $\lambda_k(x)$
  satisfies Whitney condition A as well.
  
  Whitney condition A gives the inclusion $T_{x} S \subset W_0$ and
  therefore $\Ker\cH_{x}\subset W_0$.  Using the
  rank-nullity theorem again, we get
  \begin{multline}
    \label{eq:rank_nullity_H}
    \dim \Ran \cH_{x}\big|_{W_0}
    = \codim_{W_0} \Ker \cH_{x}\big|_{W_0}
    = \codim_{W_0} \Ker \cH_{x}
    = \dim W_0 - \dim \Ker \cH_{x} \\
    = \codim_{T_{x}M} \Ker \cH_{x} - \codim_{T_{x}M} W_0
    = s(\nu) - s(\nuR).
  \end{multline}
  Comparing \eqref{eq:Range_H}, \eqref{eq:dim_Sym_VV} and
  \eqref{eq:rank_nullity_H} we conclude that
  \begin{equation}
    \label{eq:Range_H_exact}
    \Sym_{\nu}\left(V, V^\perp\right) = \Ran \cH_{x}\big|_{W_0}.
  \end{equation}
  Consequently,\footnote{Note that in equation~\eqref{eq:Sym_perp} the
    same notation $\perp$ is used for two different operations: on one
    hand, for the operation of taking orthogonal complement for
    subspaces $\mathbb F^\nu$ and, on the other hand, for the operation
    of taking orthogonal complement for subspaces of $\Sym_{\nu}$.}
  \begin{equation}
    \label{eq:Sym_perp}
    \left( \Ran \cH_{x} \right)^\perp
    \subset \left( \Ran \cH_{x}\big|_{W_0} \right)^\perp
    = \Sym_{\nu}\left(V, V^\perp\right)^\perp
    = \Sym_{\nu}\big(V^\perp, 0\big),
  \end{equation}
  i.e.\ $V^\perp$ is in the kernel of the matrices from
  $\left( \Ran \cH_{x} \right)^\perp$ which contradicts condition (N),
  see \eqref{eq:ndccn}.
\end{proof}

\begin{proof}[Proof of Proposition~\ref{isoprop}, part
  (\ref{item:isolated_point})]
  Assume, by contradiction, that $x$ is an accumulation point of a
  sequence $(x_i)$ of points which are critical on their respective strata.
  Passing to a subsequence if necessary, we can assume that all
  $z_i := \big(x_i, \lambda_k(x_i)\big)$ belong to the same stratum
  $\mathcal R$ and that the sequence of spaces $T_{z_i}\mathcal R$
  converges to a space $Q$.

  Note that $Q$ is a nontrivial generalized tangent space to $Z_k$ at
  $x$.  Since $x_i$ are critical for $\lambda_k$ restricted to the
  stratum $p(\mathcal{R})$, we have
  $d\pi(z_i)\big|_{T_{z_i}\mathcal
    R}=d\lambda_k\big|_{p(\mathcal{R})}(x_i) = 0$ and finally
  $d\pi(z)\big|_Q = 0$, which is a contradiction to part
  (\ref{item:generalized_subspace}) of the Proposition.
\end{proof}

We finish the section with establishing the comforting\footnote{In
  every particular case of $\F$, one still needs to establish
  non-degeneracy of the critical point ``by hand''.  In some
  well-studied cases, such as discrete magnetic Schr\"odinger
  operators \cite{FilKac_am18,AloGor_jst23}, degenerate critical
  points are endemic.} result of
Theorem~\ref{thm:gen_Morse_generic_intro}: the set of generalized
Morse families is open and dense.  We restate
Theorem~\ref{thm:gen_Morse_generic_intro} in an expanded form.

\begin{theorem}[Theorem~\ref{thm:gen_Morse_generic_intro}]
  \label{thm:gen_Morse_genericty}
  The set of families $\F$ having the below properties for every
  $\lambda_k$ is open and dense in the Whitney topology of
  $C^{r} (M,\Sym_n)$, $2\leq r\leq \infty$:
  \begin{enumerate}[\normalfont{(\arabic*)}]
  \item \label{item:trans_generic} at every point $x$, $\F$ is
     transverse  in the sense of Definition~\ref{def:transversality},
  \item \label{item:posdef_generic} at every point $x$, either $\Ran\cH_x$ or
    $\left(\Ran\cH_{x}\right)^\perp$ contains a positive definite
    matrix,
  \item \label{item:smooth_nondeg_generic} in the latter case,
    $\lambda_k$ restricted to the constant multiplicity stratum of $x$
    has a non-degenerate critical point at $x$.
  \end{enumerate}
  In particular, a family $\F$ satisfying the above properties is
  generalized Morse (Definition~\ref{def:generalized_Morse}).  
\end{theorem}

\begin{remark}
  \label{rem:thm:gen_Morse_generic_intro}
  Observe that property \eqref{item:trans_generic} of Theorem
  \ref{thm:gen_Morse_genericty} does not imply property
  \eqref{item:posdef_generic}: a counter-example is provided by
  Example \ref{semidef_example}. Furthermore, when
  \eqref{item:trans_generic} and the second case of
  \eqref{item:posdef_generic} hold --- and thus non-degenerate
  criticality condition (N) is fulfilled --- Lemma
  \ref{lem:precriticality_condition} shows that $\lambda_k \big|_S$
  has a critical point at $x$.  Property
  \eqref{item:smooth_nondeg_generic} posits non-degeneracy of this
  point, strengthening the conclusion to non-degenerate criticality
  condition (S).
\end{remark}

\begin{proof}[Proof of Theorem \ref{thm:gen_Morse_genericty}]
  Lemma~\ref{lem:transversality_reduced_test} showed that
  the transversality in the sense of Definition~\ref{def:transversality}
  is equivalent to the transversality between $\F$ and the
  submanifold $Q^n_{k, \nu}$ at $x$.

  As was mentioned before the discriminant variety $\mathrm{Discr}_n$
  admits a stratification satisfying Whitney condition A.  For such
  stratifications,\footnote{And in fact only for them
    \cite{Trotman78}.} we have the stratified version of the
  weak\footnote{The word ``weak" here is used to distinguish it from
    the jet version of the Thom transversality theorem which is usually
    called strong\cite{AGVbook}.}  Thom transversality theorem (see
  \cite[Proposition 3.6]{Feldman65} or informal discussions in
  \cite[Sec 2.3]{AGVbook}).  Namely, for any $1\leq r\leq \infty$, the
  set of maps in $C^{r}(M,\Sym_n)$ that are  transverse  to
  $ \mathrm{Discr}_n$ is open and dense in the Whitney topology in
  $C^{r}(M,\Sym_n)$.
    
  This establishes that property~\ref{item:trans_generic} holds for
  families $\mathcal F$ from an open and dense set in the Whitney
  topology of $C^{r} (M,\Sym_n)$.

  Properties~\ref{item:posdef_generic}--\ref{item:smooth_nondeg_generic}
  are more challenging because they involve properties of the
  derivatives of $\F$.  Let $J^1(M, \Sym_n)$ denote the space of the
  $1$-jets of smooth families of self-adjoint matrices and let
  $\Gamma^1(\F) \subset J^1(M, \Sym_n)$ denote the graph of the 1-jet
  extension of a smooth family $\F\colon M\to \Sym_n$,
  \begin{equation}
    \label{eq:graph1jetext}
    \Gamma^1(\F) := \left\{\big(x,\F(x), d\F(x)\big) \colon x\in M \right\}.
  \end{equation}
  We will show that our conclusions follows from the transversality
  (in the differential topological sense) of $\Gamma^1\left(\F\right)$
  to certain stratified subspaces of $J^1(M, \Sym_n)$.  Then the
  proposition will follow from a stratified version of the strong (or
  jet) Thom transversality theorem (see \cite[p.\ 38 and p.\
  42]{AGVbook} as well as \cite[Proposition 3.6]{Feldman65}): the set
  of families whose $1$-jet extension graph is  transverse  to a closed
  stratified subspace is open and dense in the Whitney topology of
  $C^{r} (M,\Sym_n)$ with $2\leq r\leq \infty$.  The theorem holds if
  the stratified subspace satisfies Whitney condition A.

  The jet space $J^1(M, \Sym_n)$ is the space of triples $(x,A,L)$
  such that
  \begin{equation}
    \label{eq:jet1_space}
    x \in M, \qquad A \in \Sym_n, \qquad L\in \Hom(T_xM, T_A \Sym_n).
  \end{equation}
  Given an integer $k$, $1\leq k \leq n$, a matrix $A \in Q^n_{k,\nu}$
  and a ``differential'' $L \in \Hom(T_xM, T_A \Sym_n)$, introduce the
  linear subspace
  \begin{equation}
    \label{eq:RanL}
    \Ran L_{k,A} := \Big\{ \U_{k,A}^* (Lv) \U_{k,A}
    \colon v\in T_xM \Big\}  \ \subset \ \Sym_\nu(\FF),
  \end{equation}
  where $\U_{k,A}$ is a linear isometry from $\FF^\nu$ to the $\nu$-dimensional
  eigenspace $\esp_k(A)$ of the eigenvalue $\lambda_k$ of $A$.
  
  We define the following subsets of $J^1(M, \Sym_n)$.
  \begin{align}
    \label{eq:Tc_def}
    T^c
    &:= \bigcup_{1\leq k\leq n,\ \nu\geq1} \Big\{
    (x,A,L) \colon A \in Q^n_{k,\nu},\ (\Ran L_{k,A})^\perp \neq 0
    \Big\}.
    \\
    \label{eq:Td_def}
    T^c_0
    &:= \bigcup_{1\leq k\leq n,\ \nu\geq1} \Big\{
    (x,A,L) \colon A \in Q^n_{k,\nu},\
    \exists B\in(\Ran L_{k,A})^\perp\setminus\{0\},\ \det B=0 \Big\}.
  \end{align}
  We note that that the subspace $\Ran L_{k,A}$ does not depend on the
  base point $x$ but it depends the particular choice of the isometry
  $\U_{k,A}$. Nevertheless, the properties of $\Ran L_{k,A}$ used in
  the definitions of $T^c$ and $T^c_0$ above are independent of the
  choice of the isometry $\U_{k,A}$.

  \begin{lemma}
    \label{lem:TcProperties}
    $T^c$ and $T^c_0$ are stratified spaces satisfying Whitney
    condition A.  Every stratum of $T^c$ has codimension at least $d$
    in $J^1(M, \Sym_n)$, where $d=\dim M$; every stratum of $T^c_0$
    has codimension at least $d+1$.  
  \end{lemma}
  
  \begin{proof} 
    Obviously the sets $T^c$ and $T^c_0$ are closed, with
    stratification induced by $\nu$ and the dimension of
    $(\Ran L_{k,A})^\perp$.  Besides, they are smooth\footnote{For
      bundles whose fibers are stratified spaces, smoothness is
      defined in the usual way --- as smoothness of trivializing maps.
      Smooth maps between stratified submanifolds are maps which are
      restrictions of smooth maps on the corresponding ambient
      manifolds, see \cite[p.~13]{GM88}} fiber bundles over $M$ with
    semialgebraic fibers and therefore satisfy Whitney condition~A.

    Semialgebraicity of the
    fibers of $T^c$ and $T^c_0$ follows from the Tarski--Seidenberg
    theorem stating that semialgebraicity is preserved under
    projections (\cite[Theorem 2.2.1]{BCR1998}, \cite[Theorem
    8.6.6]{Mishra1993}).  Indeed, let
    $\Pi:J^1(M, \Sym_n)\rightarrow M$ be the canonical projection.
    For each $x\in M$, we view
    $\Pi^{-1}(x)\cong \Sym_n\times \mathrm{Hom}(T_xM, \Sym_n)$ as a
    vector space by canonically identifying $T_A\Sym_n$ with $\Sym_n$.
    Focusing on $T^c$, the set
    \begin{equation*}
      \left\{(A,L,\lambda) \colon
        \det(A-\lambda I)=0,
        \ (\Ran L_{k,A})^\perp \neq 0\ \text{for some }k\right\}
    \end{equation*}
    is semialgebraic in the vector space $\Pi^{-1}(x) \times \R$.  Its
    projection on $\Pi^{-1}(x)$ is exactly the fiber
    $T^c \cap \Pi^{-1}(x)$ and it is semialgebraic by the
    Tarski--Seidenberg theorem.  The argument for $T^c_0$ is identical.
    
    Now we prove that every stratum of $T^c$ has codimension at least
    $d$. Let $\Pi_1: J^1(M, \Sym_n)\rightarrow M\times \Sym_n$ be the
    canonical projection. Recall that the codimension of $Q_{k, \nu}$
    in $\Sym_n$ is $s(\nu) := \dim \Sym_\nu - 1$.

    We consider two cases.  If $\nu$ is such that $d\leq s(\nu)$, then
    $\dim \Ran L_{k, A} \leq d < \dim \Sym_\nu$ and therefore
    $(\Ran L_{k,A})^\perp \neq 0$ for every $L$.  We get
    $\Pi_1^{-1}(M\times Q_{k,\nu}) = T^c$ and has
    codimension $s(\nu)\geq d$.

    Assume now that $\nu$ is such that $d> s(\nu)$. Then for an
    $A\in Q_{k, \nu}$ the codimension of the top stratum of
    $\Pi_1^{-1}(x, A)\cap T^c$ in $\Pi_1^{-1}(x, A)$ is equal to the
    codimension of the subset of matrices of the rank
    $\dim\Sym_\nu-1 = s(\nu)$ in the space of all
    $\bigl(\dim\Sym_\nu\bigr)\times d$ matrices, i.e.\ it is equal
    to\footnote{Here we use that the codimension of the set of
      $n_1\times n_2$ matrices of rank $r$ is equal to
      $(n_1-r)(n_2-r)$.}  $d-s(\nu)$.  Hence, the codimension in
    $J^1(M, \Sym_n)$ of the top stratum of $T^c$ is at least
    $d-s(\nu)$ plus $s(\nu)$, the codimension of
    $\Pi_1^{-1}(M\times Q_{k,\nu})$ in $J^1(M, \Sym_n)$.

    To estimate the codimension of the strata of $T^c_0$, we note that
    on the top strata of $T^c$, $(\Ran L_{k, A})^\perp$ must be
    one-dimensional. This implies that the codimension of the
    intersection of $T^c_0$ with such strata is at least $d+1$, while
    the codimension of intersections of $T^c_0$ with the lower strata
    of $T^c$ is automatically not less than $d+1$.
  \end{proof}

  We continue the proof of Theorem~\ref{thm:gen_Morse_genericty}.  Let
  $\F: M\rightarrow \mathrm{Sym_n}$ be a  transverse  family in the
  sense of Definition \ref{def:transversality} so that the graph
  $\Gamma^1\left(\F\right)$ of its $1$-jet extension is  transverse  to
  $T^c$ and $T^c_0$.  As we mentioned above the set of such maps is
  open and dense in the required topology.  From the transversality of
  $T^c_0$ to the $d$-dimensional $\Gamma^1\left(\F\right)$ we
  immediately get
  \begin{equation}
    \label{eq:Td_no_intersection}
    \Gamma^1\left(\F\right) \cap T^c_0 = \emptyset.
  \end{equation}
  
  Choose an arbitrary point $z$ and eigenvalue $\lambda_k$ (of
  multiplicity $\nu$).  If the corresponding $\Ran\cH_z$
  contains a positive definite matrix,
  properties~\ref{item:posdef_generic}--\ref{item:smooth_nondeg_generic}
  hold trivially.  We therefore focus on the opposite case:
  $\Ran\cH_z \cap \Sym_\nu^{++} = \emptyset$.  In the proof of
  Theorem~\ref{thm:regular} in
  Section~\ref{sec:precritical_but_regular} we saw that this means
  $\left(\Ran\cH_z\right)^\perp$ contains a positive
  \emph{semi}definite matrix $B$.  We want to show that $B$ is
  actually positive definite.

  Assume the contrary, namely $\det B = 0$; we will work locally in
  $J^1(M,\Sym_n)$  around the point in the graph $\Gamma^1\left(\F\right)$,
  \begin{equation}
    \label{1jetext}
    Z = (z, A, L) := \big(z,\F(z), d\F(z)\big).
  \end{equation}
  We first observe that $\Ran L_{k,A}$ defined in \eqref{eq:RanL}
  coincides with $\Ran\cH_z$ defined via \eqref{eq:HF_matrix}.  Since
  $B \in \left(\Ran\cH_z\right)^\perp\setminus\{0\}$ and $\det B=0$,
  we conclude that $Z \in \Gamma^1\left(\F\right) \cap T^c_0$,
  contradicting \eqref{eq:Td_no_intersection}.
  Property~\ref{item:posdef_generic} is now verified.

  We now verify property~\ref{item:smooth_nondeg_generic}.  We have
  a positive definite $B \in \left(\Ran\cH_z\right)^\perp$, therefore,
  by Lemma~\ref{lem:precriticality_condition}, $z$ is a smooth
  critical point along its constant multiplicity stratum $S = S_z$.
  Also from the existence of $B$, we have
  \begin{equation}
    \label{eq:XinTc}
    Z \in \Gamma^1\left(\F\right) \cap T^c.
  \end{equation}
  Denote by $T^c_{k,\nu}$ the stratum of $T^c$ containing the point
  $Z$.  By definition of transversality to a stratified space,
  $\Gamma^1\left(\F\right)$ is  transverse  to $T^c_{k,\nu}$ in
  $J^1(M,\Sym_n)$.  By dimension counting and transversality,
  $\left(\Ran L_{k,A}\right)^\perp$ is 1-dimensional along $T^c_{k,\nu}$.

  Define two submanifolds of $J^1(M,\Sym_n)$,
  \begin{align}
    \label{eq:Jknu}
    J_{k,\nu} &:= \left\{ (x,A,L) \in J^1(M,\Sym_n) \colon A \in
                Q^n_{k,\nu} \right\}, \\
    J_S &:= \left\{ (x,A,L) \in J^1(M,\Sym_n) \colon x \in S,\, A \in
          Q^n_{k,\nu},\, L(T_xS) \subset T_A Q^n_{k,\nu}\right\}
          \quad \subset \quad J_{k,\nu}.
  \end{align}
  To see that $J_S$ is a manifold, we note that for each fixed
  $(x, A)\in S \times Q^n_{k,\nu}$, the set of admissible $L$ in $J_S$
  is a vector space smoothly depending on $(x, A)$. In other words,
  $J_S$ is a smooth vector bundle over $S\times Q_{k, \nu}$.
  
  We now use the following simple fact (twice): If $U$, $V$, and $W$
  are submanifolds of $M$ such that $W$ is  transverse  to $U$ in $M$
  and $U \subset V$, then $W \cap V$ is  transverse  to $U$ in $V$.
  Since $T^c_{k,\nu} \subset J_{k,\nu}$, we conclude that
  $\Gamma^1\left(\F\right) \cap J_{k,\nu}$ is  transverse  to
  $T^c_{k,\nu}$ in $J_{k,\nu}$.  And now, since
  \begin{equation}
    \label{eq:extension-restriction}
    \Gamma^1\left(\F\right) \cap J_{k,\nu}
    = \left\{\big(x,\F(x), d\F(x)\big) \colon x\in S \right\}
    \quad \subset \quad
    J_S,
  \end{equation}
  we conclude that $T^c_{k,\nu} \cap J_S$ is  transverse  to
  $\Gamma^1\left(\F\right) \cap J_{k,\nu}$ in $J_S$.

  We have successfully localized our $x$ to $S$.  The space
  \eqref{eq:extension-restriction} looks similar to the graph of the
  1-jet extension of $\F\big|_S$, except that the differential
  $d\F(x)$ is defined on $T_x M$ and not on $T_x S$.  Consider the map
  $\Psi : J_S \to J^1(S,\R)$,
  \begin{equation}
    \label{eq:Psidef}
    \Psi(x, A, L) := \Bigl(x,\, \widehat \lambda_k(A), \,
    d\bigl(\widehat \lambda_k|_{Q_{k, \nu}^n}\bigr)(A)
    \circ L|_{T_xS}\Bigr),
  \end{equation}
  which is well-defined and smooth because $\widehat \lambda_k$
  (defined in \eqref{spectrum}) is smooth when restricted to
  $Q_{k, \nu}^n$,
  $d\bigl(\widehat \lambda_k|_{Q_{k, \nu}^n}\bigr)(A) : T_A Q_{k,
    \nu}^n \to \R$ and $L|_{T_xS} : T_xS \to T_A Q_{k, \nu}^n$ by
  definition of $J_S$.
  
  We want to show that $\Psi$ is a submersion and therefore preserves
  transversality.  To prove submersivity of a map it is enough to
  prove that, for any point $q$ in the domain, any smooth curve in the
  codomain of the map passing through the image of $q$ is the image of
  a smooth curve in the domain passing through $q$.

  Let $(x_0,A_0,L_0)$ be an arbitrary point on $J_S$.  We will work
  in a local chart around $x_0\in M$ in which $S$ is a subspace.  Let
  $P$ denote the projection in $T_xM$ onto $T_xS$, which now does not
  depend on the point $x\in S$.  Consider a smooth curve
  $(x_t,f_t,g_t)$ in $J^1(S,R)$ such that
  $\Psi(x_0,A_0,L_0) = (x_0,f_0,g_0)$.  Then the smooth curve
  \begin{equation*}
    \Big( x_t,\, A_0 + (f_t-f_0)I,\,
    L_0 + I (g_t - g_0) P \Big),
  \end{equation*}
  is in $J_S$ and is mapped to $(x_t,f_t,g_t)$ by $\Psi$.  To see
  this, observe that all sets $Q_{k, \nu}^n$ are invariant under the
  addition of a multiple of the identity matrix
  and also that
  $\widehat \lambda_k(A+\mu I) = \widehat \lambda_k(A) + \mu$ and
  therefore $d\bigl(\widehat \lambda_k|_{Q_{k, \nu}^n}\bigr)(A)I =
  1$.
  
  We now have that
  $\Psi\big(T^c_{k,\nu} \cap J_S\big)$ is  transverse  to
  $\Psi\left(\Gamma^1\left(\F\right) \cap J_{k,\nu}\right)$ in
  $J^1(S,\R)$.  It is immediate that
  \begin{equation}
    \label{eq:PsiGamma1}
    \Psi\left(\Gamma^1\left(\F\right) \cap J_{k,\nu}\right)
    = \Gamma^1\left( \widehat\lambda_k \circ \F |_{S} \right).
  \end{equation}
  We now argue that
  \begin{equation}
    \label{eq:PsiTc}
    \Psi\big(T^c_{k,\nu} \cap J_S\big) = \left\{(x,
      \widehat{\lambda}_k(A), 0) \colon x\in
      S,\, A\in Q_{k, \nu}^n \right\}.
  \end{equation}
  Indeed, at $Z$, the space $\left(\Ran L_{k,A}\right)^\perp$ is
  spanned by a positive definite matrix and this property holds in
  a small neighborhood of $Z$ in $T^c_{k,\nu}$.  By
  Lemma~\ref{lem:precriticality_condition}, $L_{k,A} |_{T_xS} = 0$,
  while by Hellmann-Feynman theorem,
  \begin{equation}
    \label{eq:eig_differential0}
    d\bigl(\widehat \lambda_k|_{Q_{k, \nu}^n}\bigr)(A) \circ L|_{T_xS}
    = \frac1\nu \Tr\left(\U^*_{k,A} L|_{T_xS} \U_{k,A} \right)
    = \frac1\nu \Tr\left( L_{k,A} |_{T_xS} \right) = 0.
  \end{equation}

  Finally, it is well known that the transversality of the graph
  $\Gamma^1\left( \widehat\lambda_k \circ \F |_{S} \right)$ to the
  0-section space \eqref{eq:PsiTc} is equivalent to the
  non-degeneracy of the critical point $z$ of
  $\widehat\lambda_k \circ \F |_{S}$, see \cite[Sec 6.1]{Hirsch94}
  or \cite[Lem 5.23]{BanyagaHurtubise_Morse}.
  Property~\ref{item:smooth_nondeg_generic} is now established.
\end{proof}

\section{Topological change in the sublevel sets:
  part~(\ref{item:Hvia_Rspace}) of Theorem~\ref{thm:homology_geom}}
\label{sec:sublevel_change1}

In this section we describe the change in the sublevel sets of the
eigenvalue $\lambda_k$ when passing through a non-degenerate
topologically critical point $x$.  It will be expressed in terms of
the data introduced in Theorem~\ref{thm:critical}, namely the Morse
index $\mu(x)$ of $\lambda_k$ restricted to the local constant
multiplicity stratum $S$ attached to the point $x$, the
  multiplicity $\nu$, and the relative index $i=i(x)$ 
introduced in equation \eqref{eq:relative_index_def}.  As a
result of the section, we will establish part~\eqref{item:Hvia_Rspace}
of Theorem~\ref{thm:homology_geom}.

First we reduce our considerations to the directions  transverse  to the
constant multiplicity stratum $S$ at a critical point $x$.
  
\begin{lemma}
  \label{lem:GM_factorization}
  Let $N$ be a submanifold of $M$ of dimension $\dim N = \codim_M S =
  s(\nu)$ which intersects $S$ transversely at $x$.
  Then, for small enough $U$ and $\epsilon>0$,
  \begin{equation}
    \label{eq:GM_factorization}
    H_r\Big(U^{+\epsilon}_x(\lambda_k),
     U^{-\epsilon}_x(\lambda_k)\Big)
    \cong H_{r-\mu(x)}\Big(
      U^{+\epsilon}_x\left(\lambda_k\big|_N\right),
      U^{-\epsilon}_x\left(\lambda_k\big|_N\right)\Big).
  \end{equation}
\end{lemma}

\begin{proof}
  By definition, see \cite[Sec I.3.5]{GM88}, the \term{local Morse data} is
  \begin{equation}
    \label{eq:local_Morse_data_def}
    \Big(U^{-\epsilon,+\epsilon}_x\left(\lambda_k\right),
    \partial U^{-\epsilon}_x\left(\lambda_k\right)\Big),
  \end{equation}
  where
  \begin{equation}
    \label{eq:U+-}
    U^{-\epsilon,+\epsilon}_x
    := U^{+\epsilon}_x \setminus U^{-\epsilon}_x.
  \end{equation}
  The \term{normal data} and the \term{tangential data} is simply the
  data of $\lambda_k$ restricted to the submanifolds $N$ and $S$,
  respectively, see \cite[Sec I.3.6]{GM88}.  The normal data is
  \begin{equation}
    \label{eq:normal_data_def}
    (J,K) \cong
    \Big(U_x^{-\epsilon,+\epsilon}\left(\lambda_k\big|_N\right),
    \partial U_x^{-\epsilon}\left(\lambda_k\big|_N\right)\Big),
  \end{equation}
  and, by the local version of the main theorem of the classical Morse
  theory \cite[Theorem 3.2]{Milnor_MorseTheory}, the tangential data
  is
  \begin{equation}
    \label{eq:tang_data}
    (P, Q) \cong \left(\mathbb B^{\mu(x)}, \partial\mathbb B^{\mu(x)}\right), 
  \end{equation}
  where $\mathbb B^{\mu(x)}$ denotes the $\mu(x)$-dimensional ball.

  We already established in Corollary \ref {isopropcor} that
  $\Bigl(x,\lambda_k(x)\Bigr)$ is a nondepraved point of the corresponding map.
  We can now use \cite[Thm I.3.7]{GM88} to decompose the local Morse
  data into a product of tangential and normal data.  More precisely,
  if the tangential data is $(P,Q)$ and the normal data is $(J,K)$,
  the local Morse data is homotopy equivalent to
  $\bigl(P\times J, (P\times K) \cup (Q\times J)\bigr)$.

  We want to compute
  \begin{align}
    \label{eq:homologies_U_Morse}
    H_r\Big(U^{+\epsilon}_x(\lambda_k), 
    U^{-\epsilon}_x(\lambda_k)\Big)
    &\cong H_r\Big(U^{-\epsilon,+\epsilon}_x\left(\lambda_k\right),
      \partial U^{-\epsilon}_x\left(\lambda_k\right)\Big)
    \\ \nonumber
    &\cong H_r\Big(P\times J, (P\times K) \cup (Q\times J) \Big),
  \end{align}
  the first equality being by Excision Theorem and the second by
  \cite[Thm I.3.7]{GM88} (and homotopy invariance).  By the relative
  version of the K\"unneth theorem, see \cite[Proposition
  12.10]{Dold1980}, we have the following short exact sequence
  \begin{align}
    \label{eq:Kunneth_exact_seq}
    0 \rightarrow
    \bigoplus_{j+k=r} H_j(P,Q) \otimes H_k(J,K)
    &\rightarrow
      H_r\big(P\times J, (P\times K) \cup (Q\times J)\big)
    \\ \nonumber
    &\rightarrow
      \bigoplus_{j+k=r-1} \Tor_1\big(H_j(P,Q), H_k(J,K)\big)
      \rightarrow 0.
  \end{align}
  Since
  \begin{equation}
    \label{eq:sphere_homology}
    H_j(P, Q)
    = H_j \left(\mathbb B^{\mu(x)}, \partial\mathbb B^{\mu(x)}\right)
    = \widetilde H_j(\Sph^{\mu(x)})
    =
    \begin{cases}
      0,& j\neq \mu(x),\\
      \mathbb Z, & j=\mu(x),
    \end{cases}
  \end{equation}
  where $\widetilde H_*$ stands for the reduced homology,
  are free, the torsion product terms in \eqref{eq:Kunneth_exact_seq}
  are all 0.  We therefore get
  \begin{align}
    \label{eq:tensors_only}
    H_r\big(P\times J, (P\times K) \cup (Q\times J)\big)
    &\cong
      \bigoplus_{j+k=r} H_j(P,Q) \otimes H_k(J,K)
    \\ \nonumber
    &= \Z \otimes H_{r-\mu(x)}(J,K)
      = H_{r-\mu(x)}(J,K),
  \end{align}
  where we used~\eqref{eq:sphere_homology} again.  Combining
  \eqref{eq:homologies_U_Morse} with \eqref{eq:tensors_only} and
  \eqref{eq:normal_data_def}, we obtain the result.
\end{proof}

The preceding lemma tells us that we can restrict our attention to the
case $M=N$.  In this case the constant multiplicity stratum attached
to $x$ is the isolated point itself and $\dim M =\dim N = s(\nu)$ (see
equation~\eqref{eq:codim_nu_again} for the formula defining $s(\nu)$
and some explanations).  Since the considerations are purely local, we
can assume that $M= \R^{s(\nu)}$, $x=0$, and $\F(x)=0$.

\newcommand{\Deps}{D_{k,\epsilon}}

\begin{lemma}
  \label{homlem1}
  Let $\F:\R^{s(\nu)} \rightarrow \mathrm{Sym}_\nu(\mathbb{F})$, $\F(0)=0$, be
  a smooth family satisfying at $x=0$ non-degenerate criticality
  conditions (N) and (S).
  Then there exists a neighborhood $U$ of $0$ in $\R^{s(\nu)}$, such
  that for sufficiently small $\epsilon>0$ the sublevel set
  $U_0^{+\epsilon}(\lambda_k)$ deformation retracts to
  the set $\Deps \cup U_0^{-\epsilon}(\lambda_k)$, where
  \begin{equation}
    \label{eq:def_retract}
    \Deps :=
    \Big\{x\in U \colon -\epsilon \leq \lambda_1\big(\F(x)\big) =
    \ldots = \lambda_k\big(\F(x)\big) \leq 0 \Big\}.
  \end{equation}
\end{lemma}

\begin{remark}
  \label{rem:boundary_cases}
  It is instructive to consider what happens in the boundary cases
  $k=1$ and $k=\nu$.  We will see that condition (N) implies that $\F$
  is injective and $\F(U)$ does not contain any semidefinite matrices
  except 0 (for a suitably small $U$).  Therefore, when $k=1$,
  \begin{equation*}
    U_0^\epsilon(\lambda_1) = U = D_{1,\epsilon} \cup
    U_0^{-\epsilon}(\lambda_1)  
  \end{equation*}
  and no retraction is needed.

  Similarly,
  when $k=\nu$, we have $U_0^{-\epsilon}(\lambda_\nu) = \emptyset$ and
  \begin{equation*}
    D_{\nu,\epsilon} = \{x \in U \colon \F(x)=0\} = 0,
  \end{equation*}
  and the Lemma reduces to the claim that $U_0^\epsilon(\lambda_\nu)$
  deformation retracts to a point.  Furthermore, the set defined in
  \eqref{eq:U+-} is
  $U_0^{-\epsilon,+\epsilon}(\lambda_\nu)= U_0^\epsilon(\lambda_\nu)$.
  In the proof of Lemma~\ref{homlem1}, we will see that
  $U_0^\epsilon(\lambda_\nu)$ is diffeomorphic to the intersection of
  a smooth $s(\nu)$-dimensional manifold through 0 with a small ball around
  0.  Therefore, we obtain that the normal Morse data in
  \eqref{eq:normal_data_def} is homeomorphic to the pair
  \begin{equation}
    \label{eq:normal_data_top}
    (J,K) = (\mathbb{B}^{s(\nu)}, \emptyset).
  \end{equation}
\end{remark}

\begin{proof}[Proof of Lemma~\ref{homlem1}]
  Since $\F(0)=0 \in \Sym_\nu$, the eigenspace $\esp_k$ in
  \eqref{eq:HF_matrix} is the whole space $\R^\nu$ and
  $\cH_0 = d\F(0)$.  From condition (N) we get that
  $\Ran d\F(0) = \Span\{B\}^\perp$ with $B\in \Sym_\nu^{++}$.  By the
  definition of $s(\nu)$ and dimension counting we conclude that
  $d\F(0) \colon \R^{s(\nu)} \to \Sym_\nu$ is injective (since $\F$ is
  a map between vector spaces, we can consider its differential to be
  a map between the same vector spaces).
  
  Choose a neighborhood $W$ of $0$ such that $d\F(x)$ remains close to
  $d\F(0)$ for all $x\in W$ (and, in particular, injective) and the
  suitably scaled normal to $\Ran d\F(x)$ remains close to $B$ (and, in particular,
  positive definite).  For future reference we note that, under these
  smallness conditions, $\F$ is a diffeomorphism from
  $W$ to $\F(W)$ and the latter set contains no positive or negative
  semidefinite matrices except $0$.

  Denote by $\mathfrak{B}_\delta$ the open ball in $\Sym_\nu$ of
  radius $\delta$ around the origin in the \emph{operator norm}.
  Choose $\delta$ sufficiently small so that
  $\partial \mathfrak{B}_\delta \cap \F(\partial W) = \emptyset$.
  This is possible because $d\F(0)$ is injective and the operator norm
  on $\F(\partial W)$ is bounded from below.  Now we take
  $U = \F^{-1}\big(\mathfrak{B}_\delta \cap \F(W)\big)$.  This set is
  non-empty because it contains $0$; it has the useful property that
  the operator norm (equivalently, spectral radius) of $\F(x)$ is
  equal to $\delta$ for $x\in\partial U$ and is strictly smaller than
  $\delta$ on $U$.

  Given a matrix $F_0 \in \F\left( U_0^{\epsilon}(\lambda_k) \right)$ we
  will describe the retraction trajectory $\Gamma_{F_0}(t)$,
  $t\in[0,1]$, starting at $F_0$.  The trajectory will be piecewise
  smooth, with the pieces described recursively. Define, for
  $m \leq k$,
  \begin{equation}
    \label{eq:Gap_set}
    G^k_m := \left\{F \in \Sym_\nu \colon \ldots < \lambda_m(F) =
      \ldots = \lambda_k(F) \leq \ldots \right\},
  \end{equation}
  which is the set of matrices with a gap below the eigenvalue
  $\lambda_m(F)$ but no gap between $\lambda_m(F)$ and
  $\lambda_k(F)$.  It is easy to see that
  \begin{equation}
    \label{eq:boundaryGap_set}
    \overline{G^k_m} \setminus G^k_m
    = \bigcup_{1\leq m' < m} G^k_{m'},
  \end{equation}
  which we will call the \term{egress set} of $G^k_m$.

  Assume that $\lambda_k(F_0) > -\epsilon$ and that
  $F_0 \in G^k_{m_0}$ for some $m_0>1$; let $t_0=0$.  Define two complementary
  spectral (Riesz) projectors corresponding to $F_0$,
  \begin{equation}
    \label{eq:projectors_def}
    P_- := P\big(\{\lambda < \lambda_k(F_0)\}\big),
    \qquad
    P_+ := P\big(\{\lambda \geq \lambda_k(F_0)\}\big),    
  \end{equation}
  and consider the affine plane in $\Sym_\nu$ defined by
  \begin{equation}
    \label{eq:gamma_submanifold}
    \{F_0 - 2\epsilon s P_+
    + r P_- \colon s,r \in \R \}.
  \end{equation}
  Since the projector $P_+ \in \Sym_\nu^+$ is non-zero and the normal to
    $d\F(x)$ is positive definite locally around $x=0$, this affine plane is
  transverse  to $\F(U)$ in $\Sym_\nu$.
  Their intersection is nonempty because it contains $F_0$ and thus,
  by the Implicit Function Theorem, it is a 1-dimensional embedded
  submanifold of $\F(U)$. Denote by $\gamma_{F_0}^{m_0}$ the connected
  component of the intersection that contains $F_0$.
    
  Furthermore, implicit differentiation of the equation
  $F_0 - 2\epsilon s P_+ + r P_- = \F(x)$ at a point
  $\F(x) \in \gamma_{F_0}^{m_0}$ shows that
  \begin{equation}
    \label{eq:dr}
    \frac{dr}{ds} = 2\epsilon \frac{\langle B_x,
      P_+\rangle}{\langle B_x, P_-\rangle},
    \qquad r(0)=0,
  \end{equation}
  where $B_x$ is either positive or negative definite symmetric matrix
  that spans the orthogonal complement to the differential $d\F(x)$.
  Since $\langle B_x, P_-\rangle>0$ for any $x\in U$, the set of
  points of $\gamma_{F_0}^{m_0}$ where $\gamma_{F_0}^{m_0}$ can be
  locally represented as a function of $s$ is both open and closed in
  the subspace topology of $\gamma_{F_0}^{m_0}$.  We conclude that
  $\gamma_{F_0}^{m_0}$ can be represented by a function of $s$
  \emph{globally}, i.e.\ as long as the closure of
  $\gamma_{F_0}^{m_0}$ in $\mathrm {Sym}_\nu$ does not hit the
  boundary of $\F(U)$.  In a slight abuse of notation, we will refer
  to this function as $\gamma_{F_0}^{m_0}$.
  Figure~\ref{fig:homotopy}(left) shows examples of the curves
  $\gamma_{F_0}^1(s)$ for the family $\F_1$ from
  equation~\eqref{eq:two_cones} and two different initial points
  $F_0$.
  
  \begin{figure}
    \centering
    \includegraphics{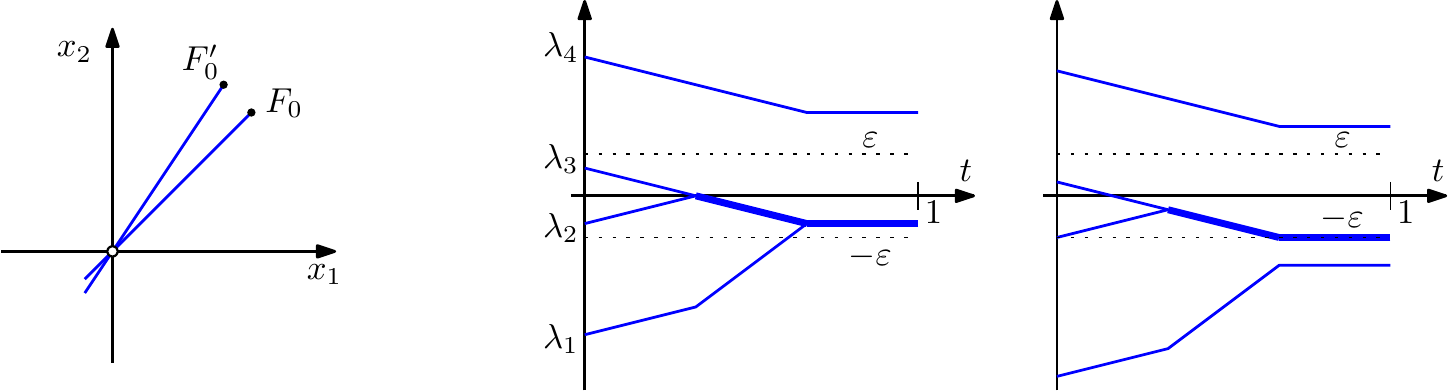}
    \caption{Left: the curves $\gamma_{F_0}^2(t)$ for $k=2$ and $\F_1$
      from equation~\eqref{eq:two_cones}, for a pair of initial
      points $F_0$.  The curves are shown in the 2-dimensional plane
      $\F_1(U)$.  The egress set for $G_2^2$ is the point $(0,0)$.
      Note that the curves intersect on the egress set, which is the
      reason we chose to specify the flow rather than the vector
      field.  Middle and right: evolution of the eigenvalues of
      $\Gamma_{F_0}(t)$ for a pair of $F_0$ with $k=3$ and the family
      $\F(U)=\{F\in\Sym_4: \Tr(F)=0\}$.  Egress points correspond to
      points where $\lambda_k$ increases its multiplicity (the latter
      is shown with thicker lines).}
    \label{fig:homotopy}
  \end{figure}

  The matrices on the curve $\gamma_{F_0}^{m_0}(s)$ have fixed
  eigenspaces but their eigenvalues change with $s$.  For small
  positive $s$ the eigenvalues $\lambda_{m_0} = \lambda_k$ and above
  decrease with the constant speed $2\epsilon$ while the eigenvalues
  below $\lambda_{m_0}$ increase because the derivative in
  \eqref{eq:dr} is positive.  This closes the gap below the eigenvalue
  $\lambda_{m_0}$ and decreases the spectral radius (operator norm) of
  $\gamma_{F_0}^m(s)$.  Therefore, the curve will intersect the egress
  set \eqref{eq:boundaryGap_set} at some time $\hat{s}>0$ \emph{before} it
  reaches the boundary $\F(\partial U)$.
  
  Setting $F_1 := \gamma_{F_0}^{m_0}(\hat{s})$ and $t_1=t_0+\hat{s}$, we determine
  $m_1 < m_0$ such that $F_1\in G^k_{m_1}$ and repeat the
  process starting at $(t_1,F_1)$.  We then join the pieces together,
  \begin{equation}
    \label{eq:Ft_def}
    \Gamma_{F_0}(t) = \gamma_{F_j}^{m_j}(t-t_j), \qquad t_j \leq t \leq t_{j+1}.
  \end{equation}

  There are two ways in which we will terminate this recursive process.
  If an egress point $F_n \in G^k_1$ is reached (which has no
  eigenvalues strictly smaller than $\lambda_k(F_n)$ and equation
  \eqref{eq:dr} becomes undefined due to $P_-=0$), we continue
  $\Gamma_{F_0}$ as a constant, $\Gamma_{F_0}(t) = F_n$ for
  $t\geq t_n$.  An example of this is shown in
  Figure~\ref{fig:homotopy}(middle).  The case $m_0=1$ which we
  previously excluded can now be absorbed into this rule.

  Alternatively, since $\lambda_k$ decreases from an initial value
  below $\epsilon$ at the constant rate $2\epsilon$, we will reach a
  point in $U_0^{-\epsilon}(\lambda_k)$ at some $\tilde{t}\leq1$.  In
  this case we also continue $F(t)$ as a constant for $t>\tilde{t}$,
  see Figure~\ref{fig:homotopy}(right) for an example (with
  $\tilde{t}=2/3$ in this particular case).  The case
  $\lambda_k(F_0) \leq -\epsilon$ can now be absorbed into the above
  description by setting $\tilde{t} = 0$.

  The preceding paragraphs show that the final values
  $\F^{-1}\left(\Gamma_{F_0}(1)\right)$ belong to the set
  $\Deps\cup U_0^{-\epsilon}(\lambda_k)$, see equation \eqref{eq:def_retract}, and that
  $x \mapsto \Gamma_{\F(x)}(t)$ acts as identity on $\Deps\cup U_0^{-\epsilon}(\lambda_k)$
  for all $t$.  This suggest that we have a
  deformation retraction
  \begin{equation}
    \label{eq:retraction_spec}
    (x,t) \mapsto \F^{-1} \left(\Gamma_{\F(x)}(t)\right),
  \end{equation}
  if we establish that the trajectories $\Gamma_{F}(t)$ define a
  continuous mapping $\F(U) \times [0,1] \to \F(U)$.

  We first note that each trajectory is continuous in $t$ by
  construction.  Therefore, we need to show that starting at a point
  $F'$ which is near $F$ will result in $\Gamma_{F'}(t)$ being near
  $\Gamma_{F}(t)$.  A perturbation of arbitrarily small norm may split
  multiple eigenvalues, therefore if $F \in G_{m}^k$ with $m < k$,
  then, in general, $F' \in G^k_{m'}$ with $m\leq m'$ (in fact,
  generically, $m'=k$).  However,
  \begin{align*}
    \left|\lambda_{m}(F') - \lambda_k(F')\right|
    &= \left|\lambda_{m}(F') - \lambda_{m}(F)\right|
    + \left|\lambda_k(F) - \lambda_k(F')\right| \\
    &\leq C |F - F'|,  
  \end{align*}
  with some $F$-independent\footnote{The constant is independent of $F$ but
    may depend on the norm used for $F$; in case of the operator norm,
    Weyl inequality yields $C=2$.} constant $C$,
  and therefore after a time of order $C |F - F'| / 2\epsilon$, the
  $k$-th eigenvalue $\Gamma_{F'}$ will collide with $m$-th eigenvalue.
  To put it more precisely, there is $\tau$,
  $0 < \tau \leq C |F - F'| / 2\epsilon$, such that
  $\Gamma_{F'}(\tau) \in G^k_{m}$.  By choosing $|F - F'|$ to be
  sufficiently small (while $\epsilon$ is small but fixed), we
  ensure that $\Gamma_{F}(\tau)$ is still in $G^k_{m}$.  By noting
  that the trajectories $\Gamma_{F'}(t)$ are continuous in $t$
  uniformly with respect to $F'$, we conclude that $\Gamma_{F'}(\tau)$
  is close to $\Gamma_{F}(\tau)$.

  For two initial points $F$ and $F'$ in the same set $G_{m}^k$, the
  curves $\gamma_{F}^{m}(s)$ and $\gamma_{F'}^{m}(s)$ will remain
  nearby for any bounded time $s<1$.  This can be seen, for example,
  as stability of the  transverse  intersection of the manifold $\F(U)$
  and the manifold \eqref{eq:gamma_submanifold}.  The stability is
  with respect to the parameters $F$, $P_+$ and $P_-$ and the spectral
  projections are continuous in $F$ precisely because $F'$ belongs the
  same set $G_{m}^k$.

  We now chain the two argument in the alternating fashion: short
  time to bring two points to the same set $G^k_{m}$, long time
  along smooth trajectories until one of the trajectories reaches
  an egress point, then short time to bring them to the same set
  $G^k_{m_1}$ and so on.  Since we iterate a bounded number of times,
  the composition is a continuous mapping.
\end{proof}

As before (see the paragraph before the formulation of Theorem
  \ref{thm:homology_geom}), let $\mathcal{C}Y$ and  $\mathcal S Y$ be the cone and the suspension of a topological space $Y$. Also let
$\Sigma Y=\mathcal S Y/ \bigl(\{y_0\}\times I\bigr)$ be the \term{reduced
  suspension} of $Y$, where $y_0\in Y$. Note that if $Y$ is a
CW-complex, then $\Sigma Y$ is homotopy equivalent to $\mathcal S Y$.
In Lemma~\ref{homlem1} we saw that $U_0^{\epsilon}(\lambda_k)$ is
homotopy equivalent to the union of $U_0^{-\epsilon}(\lambda_k)$ and the
space $\Deps$ which we aim to understand further.  We will now
show that $\Deps$ is a cone of the space $\mathcal{R}_\nu^i$ introduced in
equation~\eqref{eq:hatR_def}.

In the present setting (namely, $\F(0)=0 \in \Sym_\nu(\mathbb{F})$),
the relative index $i$ is related to $k$ via $i = \nu - k + 1$, cf.\
\eqref{eq:relative_index_def}.  Notationally, it will be more
convenient to use $k$ instead of $i$, so we introduce a slight change
in the notation, 
\begin{equation}
  \label{eq:base}
  \mathcal{R}_\nu^i = \mathcal R_{k, \nu}
  :=
  \{R\in \Sym^+_\nu \colon \Tr R=1, \dim\Ker R \geq k\}.
\end{equation}

\begin{lemma}
  \label{lem:homlem2}
  Let $\F$, $U$ and $\Deps$ be as in Lemma~\ref{homlem1}.  Then,
  for sufficiently small $\epsilon>0$, the topological space
  $\Deps$ is homeomorphic to $\mathcal{C} \mathcal{R}_{k,\nu}$
  and the topological space
  \begin{equation}
    \label{eq:glue-in-quotioned}
    \left(U_0^{-\epsilon}(\lambda_k)
      \cup
      \Deps \right)
    / U_0^{-\epsilon}(\lambda_k),
    \qquad
    1 \leq k < \nu,
  \end{equation}
  is homeomorphic to $\mathcal{S} \mathcal{R}_{k,\nu}$.
\end{lemma}


\begin{proof}
  The choice of $U$ ensured that $\F$ is a homeomorphism from
  $\Deps \subset U$ to
  \begin{equation}
    \label{eq:F-glue-in}
    \F(\Deps)
    := \big\{ F \in \F(U): -\epsilon \leq \lambda_1(F) = \ldots =
    \lambda_k(F) \leq 0 \big\}.
  \end{equation}
  We will now describe the homeomorphism from
  $\mathcal{C} \mathcal{R}_{k,\nu}$ to $\F(\Deps)$.

  Given a point $R \in \mathcal{R}_{k,\nu}$, consider the intersection
  of $\F(U)$ with the plane
  \begin{equation}
    \label{eq:plane_for_cone}
    \left\{ -\epsilon t I + r R \colon t, r \in \R \right\}.
  \end{equation}
  Mimicking the proof of Lemma~\ref{homlem1}, we conclude that the
  intersection is a 1-dimensional submanifold which has a connected
  component $\phi_R$ containing the matrix $0$.  Moreover, implicit
  differentiation at $\F(x) \in \phi_R$ yields
  \begin{equation}
    \label{eq:tangent_condition_DE}
    \frac{dr}{dt} =
    \epsilon \frac{\langle B_x, I\rangle}{\langle B_x,R\rangle},
  \end{equation}
  therefore the submanifold can be represented by a function of $t$,
  \begin{equation}
    \label{eq:homeo_description}
    \Phi(R,t) = -\epsilon t I + r(t) R,
    \qquad r(0)=0.
  \end{equation}
  When $t\in[0,1]$, we also have $\Phi(R,t) \in \F(\Deps)$
  because equation \eqref{eq:tangent_condition_DE} implies
  $r(t) \geq 0$.  We remark that $\langle B_x,R\rangle$ is bounded
  away from zero uniformly in $x\in U$ and $R \in
  \mathcal{R}_{k,\nu}$, therefore, when $\epsilon$ is sufficiently
  small, $\Phi(R,t)$ will remain in $\F(U)$ until $t>1$.
  Thus the function $\Phi$ is a well-defined\footnote{Namely,
    $\Phi(R,0)=0$ for all $R$.} mapping from
  $\mathcal{C} \mathcal{R}_{k,\nu}$ to $\F(\Deps)$.  It is
  evidently continuous.

  The properties of $\F$ imply that $\F(U)$ contains no multiples of
  identity and no positive semidefinite matrices except for the zero
  matrix.  Therefore, for every $F\in \F(\Deps)$, $F\neq 0$,
  \begin{equation}
    \label{eq:RF}
    R_F := \frac{F - \lambda_1(F)I}{\Tr\big(F - \lambda_1(F)I\big)}
    \in \mathcal{R}_{k,\nu},
  \end{equation}
  is well-defined, and we also have $-\epsilon \leq \lambda_1(F) < 0$.
  Thus
  \begin{equation}
    \label{eq:Phi_inverse}
    \Phi' \colon F \mapsto
    \begin{cases}
      \left( R_F, -\frac{\lambda_1(F)}{\epsilon} \right),
      & \text{if } F \neq 0, \\
      (*, 0), & \text{if } F=0,
    \end{cases}
  \end{equation}
  is a well-defined continuous mapping from $\F(\Deps)$ to
  $\mathcal{C} \mathcal{R}_{k,\nu}$.  It remains to verify that
  $\Phi'$ is the inverse of $\Phi$.  It is immediate that
  $\Phi' \circ \Phi = \mathrm{id}$.  To prove that
  $\Phi \circ \Phi' = \mathrm{id}$ we observe that the intersection
  $\phi_{R_F}$, corresponding to $R_F$ of equation~\eqref{eq:RF},
  contains $F$; we only need to show that $F$ and $0$ belong to the
  same connected component of $\phi_R$.

  The point $F$ on the plane \eqref{eq:plane_for_cone} corresponds to
  $t = -\lambda_1(F)/\epsilon > 0$ and some $r=r'$.  Decreasing $t$
  from this point decreases $r(t)$ and therefore decreases the
  operator norm of $\Phi$.  Thus we will not hit the boundary of
  $\F(\Deps)$ as long as $t\geq 0$.  Therefore, we will arrive at
  the matrix $0$ while staying on the same connected component.

  We have established the first part of the lemma.  To understand the
  quotient in \eqref{eq:glue-in-quotioned}, we note that
  \begin{align*}
    \left(U_0^{-\epsilon}(\lambda_k)
      \cup
      \Deps \right)
    / U_0^{-\epsilon}(\lambda_k)
    &= \Deps / \left(U_0^{-\epsilon}(\lambda_k)
      \cap
      \Deps \right) \\
    &\cong
    \F(\Deps) /
    \big\{ F \in \F(\Deps): -\epsilon = \lambda_1(F) = \ldots =
      \lambda_k(F) \big\} \\
    &= \F(\Deps) / \big\{\Phi(R,1) \colon R \in
    \mathcal{R}_{k,\nu} \} \cong \mathcal{S} \mathcal{R}_{k,\nu},
  \end{align*}
  completing the proof.
\end{proof}

\begin{proof}[Proof of Theorem~\ref{thm:homology_geom},
  part~\ref{item:Hvia_Rspace}]
  We review how the preceding lemmas link together to give the proof
  of the theorem.  Lemma~\ref{lem:GM_factorization} shows that the
  smooth part $\F\big|_S$ gives the classical contribution to the
  sublevel set quotient and we can focus on understanding the
   transverse part $\F\big|_N$.  We remark that by
  Corollary~\ref{cor:transversal_cut_of_CP} the point $x$ remains
  non-degenerate topologically critical when we replace $\F$ with
  $\F\big|_N$.

  Combining Lemmas~\ref{lem:GM_factorization}, \ref{homlem1}, and
  \ref{lem:homlem2}, we compute the $r$-th homology group
  \begin{align}
    \label{eq:homology_bring_together}
    H_r\Big(U_x^{+\epsilon}(\lambda_k),\,
    U_x^{-\epsilon}(\lambda_k)\Big)
    &\cong
      H_{r-\mu(x)}\Big(U_x^{+\epsilon}(\lambda_k|_N),\,
      U_x^{-\epsilon}(\lambda_k|_N)\Big)
    \\ \nonumber
    &\cong
      H_{r-\mu(x)}\Big(U_x^{-\epsilon}(\lambda_k|_N) \cup \Deps,\,
      U_x^{-\epsilon}(\lambda_k|_N)\Big)
    \\ \nonumber
    &\cong \widetilde{H}_{r-\mu(x)}\Big(
    \big(U_x^{-\epsilon}(\lambda_k|_N) \cup \Deps\big) \big/
      U_x^{-\epsilon}(\lambda_k|_N)\Big)
    \\ \nonumber
    &\cong
      \widetilde{H}_{r-\mu(x)}\Big(\mathcal{S}\mathcal{R}_{k,\nu}\Big).
  \end{align}
  Taking into account \eqref{eq:base}, we obtain the claim for
  $k < \nu$.  The answer for $k=\nu$ (equivalently, $i=1$) was already
  established in Remark~\ref{rem:boundary_cases}.
\end{proof}

\section{Topological change in the sublevel sets:
  part~(\ref{item:Hvia_Grass}) of Theorem~\ref{thm:homology_geom}}
\label{sec:sublevel_change2}

We will go from part~(\ref{item:Hvia_Rspace}) to
part~(\ref{item:Hvia_Grass}) of Theorem~\ref{thm:homology_geom} by
relating the space $\mathcal R_{k, \nu}$ to the Thom space of a
particular vector bundle.  Recall \cite{Milnor-Stasheff} that the
\term{Thom space} $\cT(E)$ of a real vector bundle $E$ over a manifold
is the quotient of the unit ball bundle $\B(E)$ of $E$ by the unit
sphere bundle of $E$ with respect to some Euclidean metric on $E$. If
the base manifold of the bundle $E$ is compact, then the Thom space of
$E$ is the Alexandroff (one point) compactification of the total space
of $E$.  As before, we denote by $\Gr_\FF(k,n)$ the Grassmannian of
(non-oriented) $k$-dimensional subspaces in $\FF^n$.



Given a vector bundle $E$ denote by $S^2 E$ the \term{symmetric tensor
  product} of $E$.  Namely, $S^2 E$ is the vector bundle over the same
base as $E$; the fiber of $S^2 E$ over a point is equal to the
symmetric tensor product with itself of the fiber of $E$ over the same
point.  Choosing a Euclidean metric on $E$ we can identify $S^2 E$
with the bundle whose fiber over a point is the space of all
self-adjoint isomorphisms of the fiber of $E$ over the same point.
Then by $S^2_0 E$ we denote the bundle of traceless elements of
$S^2 E$.  Obviously
\begin{equation}
  \label{trivial1}
  S^2 E\cong S^2_0 E \oplus \theta^1,
\end{equation}
where $\theta^1$ is the trivial rank 1 bundle over the base of $E$.
Finally, let $\Taut_\FF(k,n)$ be the tautological bundle over the
Grassmannian $\Gr_\FF(k,n)$: the fiber of this bundle over
$\Lambda \in \Gr_\FF(k,n)$ is the vector space $\Lambda$ itself.

\begin{lemma}
  \label{lem:Agrachev}
  Recall the relative index $i$ of the eigenvalue, equation
  \eqref{eq:relative_index_def}, which in the present situation is
  equal to $i=\nu-k+1$.  Then the space
  \begin{equation*}
    \mathcal R_{k, \nu}
    :=
    \{R\in \Sym^+_\nu \colon \Tr R=1, \dim\Ker R \geq k\}
  \end{equation*}
  with $1\leq k < \nu$ is homotopy equivalent to the Thom space of the
  real vector bundle over the Grassmannian $\Gr_\FF(i-1,\nu-1)$,
  \begin{equation}
    \label{explicitE}
    E_{i,\nu} :=
    S^2_0 \Taut_\FF(i-1,\nu-1) \oplus \Taut_\FF(i-1,\nu-1).
  \end{equation}
  The rank of the bundle is $s(i)-1$, where $s(i)$ is given by
  \eqref{eq:codim_nu}.  The bundle is non-orientable if $\FF=\R$
  and $i$ is even, and orientable otherwise.
\end{lemma}

\begin{remark}
  \label{rem:normal_data_bottom}
  Let us consider the boundary case $k=1$ or, equivalently, $i=\nu$.
  The Grassmannian $\Gr_\FF(i-1,\nu-1)$ is a single point, so the
  vector bundle $E_{\nu,\nu}$ is simply a real vector space of
  dimension $s(\nu)-1$.  Its Thom space is the one-point
  compactification of $\R^{s(\nu)-1}$, namely the sphere
  $\mathbb{S}^{s(\nu)-1}$.  Correspondingly, the cone
  $\mathcal{C}\mathcal{R}_{1, \nu}$ is homotopy equivalent to the ball
  $\mathbb{B}^{s(\nu)}$.  Therefore, we get that the normal data at
  the bottom eigenvalue is homotopy equivalent to the pair
  \begin{equation}
    \label{eq:normal_data_bottom}
    (J,K) = \left(\mathbb{B}^{s(\nu)}, \partial \mathbb{B}^{s(\nu)} \right).
  \end{equation}
\end{remark}

\begin{proof}[Proof of Lemma~\ref{lem:Agrachev}]
  The homotopy equivalence has been established in \cite[Theorem
  1]{A11} and the proof thereof.  For completeness we review the main
  steps here. 

  Fixing an arbitrary unit vector $e \in \FF^\nu$ we
  define
  \begin{equation}
    \label{eq:submanifoldP_def}
    \mathcal{P}_{k,\nu}
    := \left\{\frac1{\nu-k}P \in \Sym_\nu \colon P^2=P,\ 
      \dim\Ker P = k, \ e \in \Ker P \right\}
    \cong \Gr_\FF(\nu-k,\nu-1).    
  \end{equation}
  One can show that $\mathcal{R}_{k,\nu} \setminus
  \mathcal{P}_{k,\nu}$ is contractible: if $P_e = ee^*$ is the
  projection onto $e$, consider
  \begin{equation}
    \label{eq:Agrachev_retraction}
    (A,t) \mapsto \phi_k\big( (1-t)A + tP_e \big),
    \qquad
    A \in \mathcal{R}_{k,\nu} \setminus
    \mathcal{P}_{k,\nu},
    \quad
    t\in [0,1]
  \end{equation}
  where $\phi_k(M)$ acts on the eigenvalues of $M$ as
  \begin{equation}
    \label{eq:phi_eig_action}
    \lambda_j(M) \mapsto \max\big[0, \lambda_j(M)-\lambda_k(M) \big],
  \end{equation}
  followed by a normalization to get unit trace.  Using interlacing
  inequalities\footnote{A particularly convenient form for this task
    can be found in \cite[Thm~4.3]{BerKenKurMug_tams19}.} for the rank
  one perturbation (up to rescaling) of $A$ by $P$, one can show that
  \eqref{eq:Agrachev_retraction} is a well-defined retraction.  In
  particular, \eqref{eq:phi_eig_action} does not produce a zero matrix
  (which cannot be trace-normalized) and the result of
  \eqref{eq:Agrachev_retraction} is not in $\mathcal{P}_{k,\nu}$ for
  any $t$.

  We now obtain that $\mathcal{R}_{k,\nu}$ is homotopy equivalent to
  the Thom space of the normal bundle of $\mathcal{P}_{k,\nu}$ in
  $\mathcal{R}_{k,\nu}$.  Indeed, a tubular neighborhood $T$ of
  $\mathcal{P}_{k,\nu}$ in $\mathcal{R}_{k,\nu}$ is diffeomorphic to
  the normal bundle of $\mathcal{P}_{k,\nu}$, while the above
  retraction allows one to show $\mathcal{R}_{k,\nu}$ is homotopy
  equivalent to $\mathcal{R}_{k,\nu} \big/ \left(\mathcal{R}_{k,\nu}
    \setminus T\right)$.

  The normal bundle of $\mathcal{P}_{k,\nu}$ in
  $\mathcal{R}_{k,\nu}$ is a Whitney sum of the normal
  bundle of $\mathcal{P}_{k,\nu}$ in
  \begin{equation}
    \label{eq:submanifoldPhat_def}
    \widehat{\mathcal{P}}_{k,\nu}
    := \left\{\frac1{\nu-k}P \in \Sym_\nu \colon P^2=P,\ 
      \dim\Ker P = k \right\},
  \end{equation}
  and the normal bundle of $\widehat{\mathcal{P}}_{k,\nu}$ in
  $\mathcal{R}_{k,\nu}$. The fiber in the former bundle is
  $(\Ker P)^\perp$: it consists of the directions in which $e$ can
  rotate out of $\Ker P$. Therefore the former bundle is
  $\Taut_\FF(\nu-k, \nu-1)$.  The fiber in the normal bundle of
  $\widehat{\mathcal{P}}_{k,\nu}$ in $\mathcal{R}_{k,\nu}$ consists of
  all self-adjoint perturbations to the operator $\frac1{\nu-k}P$ that
  preserve its kernel and unit trace.  Identifying these with the space
  of traceless self-adjoint operators on $(\Ker P)^\perp$, we get
  $S_0^2 \Taut_\FF(\nu-k, \nu-1)$.  We get \eqref{explicitE} after
  recalling that $\nu-k=i-1$.

  To calculate the rank we use
  \begin{equation}
    \label{eq:rank_data}
    \rank\big(\Taut_\FF(i-1, \nu-1)\big)
    =
    \begin{cases}
      i-1, & \FF=\R, \\
      2(i-1) & \FF=\C,
    \end{cases}
  \end{equation}
  and
  \begin{equation}
    \label{eq:rank_data2}
    \rank\big(S_0^2\Taut_\FF(i-1, \nu-1)\big)
    =
    \begin{cases}
      \frac12(i-1)i - 1, & \FF=\R, \\
      (i-1)^2-1 & \FF=\C,
    \end{cases}
  \end{equation}
  giving $\frac12(i-1)(i+2)-1$ in total in the real case and
  $i^2-2$ in the complex case.

  Recall that a real vector bundle $E$ is orientable if and only if its
  first Stiefel--Whitney class $w_1(E) \in H^1(B, \Z_2)$ vanishes
  (here $B$ is the base of the bundle).  The first Stiefel--Whitney
  class is additive with respect to the Whitney sum, therefore
  \newcommand{\tmpE}{\mathcal{E}}
  \begin{equation}
    \label{eq:swclass_step1}
    w_1(E_{i,\nu}) = w_1(S^2_0 \tmpE) + w_1(\tmpE),
    \qquad
    \tmpE = \Taut_\R(i-1, \nu-1).
  \end{equation}
  Using additivity on equation \eqref{trivial1} gives
  $w_1(S^2_0 \tmpE)=w_1(S^2 \tmpE)$ because $w_1$ is zero for
  the trivial bundle.  The classical formulas for the
  Stiefel--Whitney classes of symmetric tensor power (see, for
  example, \cite[Sec.~19.5.C, Theorem 3]{FF16}) yield
  $w_1(S^2 \tmpE)=(\rank \tmpE + 1) w_1(\tmpE)$ and, finally,
  \begin{equation}
    \label{w1}
    w_1(E_{i,\nu})
    = (\rank \tmpE+2) w_1(\tmpE),
    \qquad
    \tmpE = \Taut_\R(i-1, \nu-1).
  \end{equation}
  Since the real tautological bundle $\tmpE$ is not orientable and has
  rank $i-1$, $w_1(E_{i,\nu})$ vanishes if and only if $i+1$ is zero
  modulo 2, completing the proof of the lemma.
\end{proof}

Recall that the oriented Grassmannian $\widetilde{\Gr}_\R(k,n)$
consisting of the \emph{oriented} $k$-dimensional subspaces in $\R^n$
is a double cover of $\Gr_\R(k,n)$.  Let $\tau$ denote the
orientation-reversing involution on $\widetilde{\Gr}_\R(k,n)$.  In the
space of $q$-chains of $\widetilde{\Gr}_\R(k,n)$ over the ring $\Z$ we
distinguish the subspace of chains which are skew-symmetric with
respect to $\tau$: $\tau(\alpha)=-\alpha$, where $\alpha$ is a chain.
The subspaces of skew-symmetric $q$-chains are invariant under the
boundary operator and therefore define a complex.  The homology groups
of this complex will be denoted $H_q\big(\Gr_\R(k,n); \widetilde\Z)$.
In the sequel we refer to them as \term{twisted homologies}, as they
are homologies with local coefficients in the module of twisted
integers $\widetilde\Z$, i.e.\ $\Z$ considered as the module
corresponding to the nontrivial action of $\Z_2$ on $\Z$.

\begin{proof}[Proof of Theorem~\ref{thm:homology_geom},
  part~(\ref{item:Hvia_Grass})]
  In the case $1 < i(x) \leq \nu(x)$, we start from
  equation~\eqref{eq:Hvia_Rspace},
  \begin{align}
     H_r\Big(U_x^{+\epsilon}(\lambda_k),\,
    U_x^{-\epsilon}(\lambda_k)\Big)
    &\cong
      \widetilde{H}_{r-\mu(x)}\Big(\mathcal{S}\mathcal{R}_{k,\nu}\Big)
      \cong \widetilde{H}_{r-\mu(x)}\Big(\Sigma\mathcal{R}_{k,\nu}\Big)
    \\ \nonumber
    &\cong \widetilde{H}_{r-\mu(x)}\Big(\Sigma\big(\cT\left(E_{i,\nu}\right)\big) \Big),
  \end{align}
  where $E_{i,\nu}$ is given by \eqref{explicitE}.
  
  Recall \cite[Cor.~16.1.6]{Husemoller} that the reduced
  suspension of a Thom space of a vector bundle is homeomorphic to
  the Thom space of the Whitney sum of this bundle with the trivial
  rank 1 bundle $\theta^1$, i.e.
  \begin{equation}
    \label{Efin}
    \Sigma\big(\cT\left(E_{i,\nu}\right)\big)
    \cong \cT\left(\widehat E_{i,\nu}\right),
    \qquad
    \widehat E_{i, \nu} := E_{i, \nu}\oplus \theta^1
  \end{equation}
  The bundle $\widehat E_{i,\nu}$ is orientable if and only if
  $E_{i,\nu}$ is orientable; its rank is one plus the rank of
  $E_{i,\nu}$.  Lemma~\ref{lem:Agrachev} supplies both pieces of
  information.

  The bundle
  $\widehat E_{i,\nu}$ is orientable if $\FF=\C$ or if $\FF=\R$ and
  $i$ is odd, and we can use the homological version of the Thom
  isomorphism theorem \cite[Lemma 18.2]{Milnor-Stasheff}, which gives
  \begin{equation}
    \label{thomres}
    \widetilde H_{r-\mu(x)}\left(\cT(\widehat E_{i,\nu})\right)
    = H_{r-\mu(x)-s(i)}\big(\Gr_\FF(i-1,\nu-1)\big),
  \end{equation}
  which is the right-hand side of \eqref{eq:relhom_main} in 
  the corresponding cases.
  
  When $\mathbb F=\mathbb R$ and $i$ is even, the bundle
  $\widehat E_{i, \nu}$ is nonorientable \eqref{eq:relhom_main} results
  from the Thom isomorphism for nonorientable bundles
  \cite[Theorem 3.10]{Skl_fpm03}\footnote{An analogous
    result for cohomologies can be found in \cite{Rud80}.},
  \begin{equation}
    \label{thomresno}
    \widetilde H_{r-\mu(x)}\left(\cT(\widehat E_{i,\nu})\right)
    = H_{r-\mu(x)-s(i)}\big(\Gr_\FF(i-1,\nu-1);\widetilde \Z\big).
  \end{equation}

  In the special case $k=\nu$, not covered by
  Lemma~\ref{lem:Agrachev}, we compute directly using
  Lemma~\ref{lem:GM_factorization} and
  Remark~\ref{rem:boundary_cases},
  \begin{align}
    \label{eq:homology_k_nu}
    H_r\Big(U^{\lambda^c+\epsilon}(\lambda_\nu),\,
    U^{\lambda^c-\epsilon}(\lambda_\nu)\Big)
    &\cong
      H_{r-\mu(x)}\Big(U^{\lambda^c+\epsilon}(\lambda_\nu|_N),\,
      U^{\lambda^c-\epsilon}(\lambda_\nu|_N)\Big)
    \\ \nonumber
    &\cong 
      H_{r-\mu(x)}\Big(\mathbb{B}^{s(\nu)}, \emptyset\Big)
      \cong H_{r-\mu(x)}\Big(\{x\}\Big)
    \\ \nonumber
    &\cong H_{r-\mu(x)-s(i)}\Big(\Gr_\FF(i-1,\nu-1)\Big),
  \end{align}
  since $i=1$, $s(i)=0$ and $\Gr_\FF(0,\nu-1)$ is a single point.
\end{proof}

\begin{remark}
  \label{rem:poincare_lefschetz_proof}
  One can also derive \eqref{thomresno}, using Poincar\'e and
  Poincar\'e--Lefschetz dualities in their usual and skew form,
  mimicking the proof of \cite[Lemma 18.2]{Milnor-Stasheff}.  This
  alternative derivation is included as
  Appendix~\ref{sec:thomhom_computation}.
\end{remark}


\section{Proof of part (\ref{item:criticality}) of
    Theorem~\ref{thm:critical}: criticality}
\label{sec:main_theorem_part1}

$\mathbb{F} = \C$. In the setting of
Theorem~\ref{thm:homology_geom}, the Poincar\'e polynomials of the
relative homology groups
$H_*\left(U^{\lambda_k(x)+\epsilon}(\lambda_k),
  U^{\lambda_k(x)-\epsilon}(\lambda_k)\right)$ is equal to
\begin{equation}
  \label{eq:PoincareGrassmannianComplex_0}
  t^{\mu(x)+ s(i)} P_{\Gr_\C(i-1, \nu-1)}(t). 
\end{equation}

Betti numbers for complex Grassmannians were established by Ehresmann,
see \cite[Theorem on p.~409, section II.7]{Ehresmann1934}.  The
$r$-th Betti number is zero if $r$ is odd and is equal to the number
of Young diagrams with $r/2$ cells that fit inside the $k\times (n-k)$
rectangle, if $r$ is even. The Poincar\'e polynomial $P_{\Gr_\C(k,n)}$
is nothing but the generating function for this \term{restricted
  partition problem}.  The latter is well known to be of the form
\begin{equation}
\label{eq:PoincareGrassmannianComplex_1}	
P_{\Gr_\C(k,n)}(t)
= {\binom{n}{k}}_{t^2},
\end{equation}
see \cite[Theorem~3.1, p.~33]{Andrews76}.  By
\eqref{eq:PoincareGrassmannianComplex_0} and
\eqref{eq:PoincareGrassmannianComplex_0}, the Poincar\'e polynomials
of the relative homology groups
$H_*\left(U^{\lambda_k(x)+\epsilon}(\lambda_k),
  U^{\lambda_k(x)-\epsilon}(\lambda_k)\right)$ does not vanish, so $x$
is a critical points.

$\mathbb{F} = \R$.  The calculation of the Poincar\'e polynomials of
integer homologies will be done in the next section, but it is equal
to zero in some cases and so does not lead to the proof of
criticality. Instead, we show that $\mathbb Z_2$-homology groups
$H_*\bigl(U_x^{+\epsilon}(\lambda_k),\,
U_x^{-\epsilon}(\lambda_k);\mathbb Z_2)$ are nontrivial.

Let $P_{Y, \Z_2}(t)$ be the Poincar\'e polynomial of
$\Z_2$-homologies of a topological space $Y$, i.e. the coefficient of
$t^i$ in $P_{Y, \Z_2}(t)$ is equal to the number of copies of $\Z_2$s
in $H_i(Y, \Z_2)$.  The Poincar\'e polynomial of $\Z_2$-homologies of
the Grassmannian $\Gr_\R(k,n)$ is well known (\cite[Theorem~3.1,
p.~33]{Andrews76}, \cite[\S 7]{Milnor-Stasheff}) to be
\begin{equation}
  \label{Poincare_torsion_Z2_Grassmannian}
  P_{\Gr_\R(k,n), \Z_2}(t) = \binom{n}{k}_t.
\end{equation}

Moreover, when the coefficients are $\Z_2$, there is no difference
between symmetric and skew-symmetric chains, therefore if
$P_{Y, \widetilde{\Z_2}}(t)$ is the Poincar\'e polynomial of the
twisted $\Z_2$-homologies of Y, then
$P_{Y, \widetilde{\Z_2}}(t) = P_{Y, \Z_2}(t)$.  Based on this and
\eqref{Poincare_torsion_Z2_Grassmannian}, in the setting of
Theorem~\ref{thm:homology_geom}, the Poincar\'e polynomial of the
relative $\mathbb Z_2$-homology groups
$H_*\left(U^{\lambda_k(x)+\epsilon}(\lambda_k),
  U^{\lambda_k(x)-\epsilon}(\lambda_k);\mathbb Z_2 \right)$ is equal
to
\begin{equation}
  \label{Poincare_torsion_Z2_Grassmannian_1}
  t^{\mu(x)+ s(i)}\binom{\nu-1}{i-1}_t,
\end{equation}
which is not zero. The proof of critically in the case of
$\mathbb F=\mathbb R$ is complete.

\section{Proof of Theorem~\ref{thm:critical}, part
  (\ref{item:Morse_polynomial}): computing the Poincar\'e
  polynomials}
\label{sec:main_theorem_proofs}

Theorem~\ref{thm:critical} will be obtained as a combination of the
next two lemmas.  Lemma~\ref{oldpar2} provides an expression for the
Poincar\'e polynomial of twisted homologies
$H_*\big(\Gr_\R(i-1,\nu-1); \widetilde \Z\big)$ by relating it to the
Poincar\'e polynomial of the oriented Grassmannian
$\widetilde{\Gr}_\R$.  Lemma~\ref{lem:GrPolynomials} below collates
known expressions for the Poincar\'e polynomials of Grassmannians and
oriented Grassmannians.

\begin{lemma}
  \label{oldpar2}
  In the setting of Theorem~\ref{thm:homology_geom},
  the Poincar\'e polynomials of the relative homology groups
  $H_*\left(U^{\lambda_k(x)+\epsilon}(\lambda_k),
    U^{\lambda_k(x)-\epsilon}(\lambda_k)\right)$ is equal to
  \begin{equation}
    \label{eq:Morse_contrib}
    t^{\mu(x)+ s(i)}
    \begin{cases}
      P_{\Gr_\R(i-1, \nu-1)}(t),
      & \text{$\FF=\R$ and $i$ is odd}, \\
      P_{\widetilde{\Gr}_\R(i-1, \nu-1)}(t)-
      P_{\Gr_\R(i-1, \nu-1)}(t),
      & \text{$\FF=\R$ and $i$ is even},\\
      P_{\Gr_\C(i-1, \nu-1)}(t),
      & \text{$\FF=\C$}. \\
    \end{cases}
  \end{equation}
  where $P_Y(t)$ denotes the Poincar\'e polynomial of the manifold
  $Y$.
\end{lemma}
	
\begin{proof}
  Since we already established part~(\ref{item:Hvia_Grass}) of
  Theorem~\ref{thm:homology_geom}, we only need to show that the
  Poincar\'e polynomial of the homology groups
  $H_*(\Gr_\R(i-1,\nu-1); \widetilde \Z)$ is equal to
  $P_{\widetilde{\Gr}_\R(i-1, \nu-1)}(t)- P_{\Gr_\R(i-1, \nu-1)}(t)$.

  We will use homologies with coefficients in $\Q$ (or $\R$).  Indeed,
  since the Betti numbers ignore the torsion part of $H_r(\cdot; \Z)$,
  the Universal Coefficients Theorem (see,
  e.g. \cite[Sec.~3.A]{Hatcher}) implies they can be calculated as the
  rank of $H_r(\cdot; G)$ with any torsion-free abelian group $G$.
  The benefit of using $\Q$ is that now any chain $c$ in
  $\widetilde{\Gr}_\R(i-1,\nu-1)$ can be uniquely represented as a sum
  of a symmetric and a skew-symmetric chains with coefficients in
  $\Q$,
  \begin{equation}
    \label{ASdecomp}
    c=\cfrac{1}{2}\bigl(c+\tau(c)\bigr)+\cfrac{1}{2}\bigl(c-\tau(c)\bigr),
  \end{equation}
  where $\tau$ is the orientation reversing involution of
  $\widetilde {\mathrm{Gr}}(i-1,\nu-1)$ (viewed as a double cover of
  $\mathrm{Gr}(i-1,\nu-1)$).  The analogous statement is of course
  wrong in integer coefficients, as $\cfrac{1}{2}\notin \mathbb Z$.

  Since the boundary operator preserves the parity of a chain, the
  homology $H_r\big(\widetilde\Gr_\R(i-1,\nu-1); \Q\big)$ decomposes
  into the direct sum of homologies of $\tau$-symmetric and
  $\tau$-skew-symmetric chains on $\widetilde\Gr_\R(i-1,\nu-1)$.  The
  former homology coincides with the usual homology of
  $\Gr_\R(i-1,\nu-1)$.  The latter yields, by definition, the twisted
  $\Q$-homology of $\Gr_\R(i-1,\nu-1)$.  To summarize, we obtain
  \begin{equation}
    \label{decompQ}
    H_r(\widetilde\Gr_\R(i-1,\nu-1); \Q)
    = H_r\big(\Gr_\R(i-1,\nu-1); \Q\big)
    \oplus H_r\big( \Gr_\R(i-1,\nu-1); \widetilde{\Q}\big).
  \end{equation}
  The sum in \eqref{decompQ} translates into the
  sum of Poincar\'e polynomials, yielding the middle line in
  \eqref{eq:Morse_contrib}.
\end{proof}

\begin{lemma}
  \label{lem:GrPolynomials}
  Let $k,n \in \mathbb{N}$, $k \leq n$.
  The Poincar\'e polynomials of the Grassmannians $\Gr_\C(k,n)$,
  $\Gr_\R(k,n)$ and $\widetilde{\Gr}_\R(k,n)$ are given by
  \begin{align}
    \label{eq:PoincareGrassmannianComplex}
    P_{\Gr_\C(k,n)}(t)
    &= {\binom{n}{k}}_{t^2},
    \\[10pt]
    \label{eq:PoincareGrassmannianReal}
    P_{\Gr_\R(k,n)}(t)
    &=
      \begin{cases}
        {\dbinom{\lfloor n/2\rfloor}{\lfloor k/2 \lfloor}}_{t^4},
        & \mbox{if } k(n-k) \mbox{ is even}\\[10pt]
        (1+t^{n-1})\dbinom{n/2-1}{(k-1)/2}_{t^4},
        & \mbox{if } k(n-k) \mbox{ is odd},
      \end{cases}
    \\[10pt]
    \label{eq:PoincareGrassmannianOrientedReal}
    P_{\widetilde \Gr_\R(k, n)}(t)
    &=
      \begin{cases}
        (1+t^{n-k})\dbinom{(n-1)/2}{(k-1)/2}_{t^4}
        &  \mbox{if $k$ is odd, $n$ is odd},\\[10pt]
        (1+t^{n-1})\dbinom{n/2-1}{(k-1)/2}_{t^4},
        & \mbox{if $k$ is odd, $n$ is even},\\[10pt]
        \dfrac{(1+t^k)(1+t^{n-k})}{1+t^n} \dbinom{n/2}{k/2}_{t^4},
        & \mbox{if $k$ is even, $n$ is even}.
      \end{cases}
  \end{align}
\end{lemma}

\begin{remark}
  We will not need the Poincar\'e polynomial in the last line of
  \eqref{eq:PoincareGrassmannianOrientedReal} and we include it for
  completeness only.  The case of even $k$ and odd $n$ is covered by
  the first line of \eqref{eq:PoincareGrassmannianOrientedReal} since
  $\widetilde \Gr_\R(k, n) = \widetilde \Gr_\R(n-k, n)$.
\end{remark}

\begin{proof}
  The complex case was already discussed in the beginning of section
  \ref{sec:main_theorem_part1}.

  The $r$-th Betti number of the real Grassmannian has a similar
  combinatorial description \cite[Theorem IV, p. 108]{Iwamoto48}: it
  is equal to the number of Young diagrams of $r$ cells that fit
  inside the $k\times (n-k)$ rectangle and have even length
  differences for each pair of columns and for each pair of rows. From
  this it can be shown that the Poincar\'e polynomial
  $P_{\Gr_\R(k,n)}$ satisfies \eqref{eq:PoincareGrassmannianReal} (see
  also \cite[Theorem 5.1]{CasKod13}).  We remark that for $k$ and $n$
  both even, \eqref{eq:PoincareGrassmannianReal} is a consequence of
  \eqref{eq:PoincareGrassmannianComplex} because the corresponding
  Young diagrams must be made up from $2\times 2$ squares.

  Finally, the oriented Grassmannian is a homogeneous space, namely
  \begin{equation}
    \label{eq:OrientedGrassmannianHomogeneous}
    \widetilde \Gr(k,n)\cong SO(n)/\bigl( SO(k)\times  SO(n-k)\bigr).
  \end{equation}
  The corresponding Poincar\'e polynomial has been computed within the
  general theory of de~Rham cohomologies of homogeneous spaces, see,
  for example, \cite[Chap.~XI]{GWH1976}.  Up to notation, the first
  line of \eqref{eq:PoincareGrassmannianOrientedReal} corresponds to
  \cite[Lines 2-3, col.~3 of Table~II on p.~494]{GWH1976}, the second
  line of \eqref{eq:PoincareGrassmannianOrientedReal} corresponds to
  \cite[Lines 2-3, col.~1 of Table~III on p.~496]{GWH1976} and the
  third line of \eqref{eq:PoincareGrassmannianOrientedReal} corresponds to
  \cite[Lines 2-3, col.~2 of Table~III on p.~495]{GWH1976}.
\end{proof}

\begin{proof}[Proof of Theorem~\ref{thm:critical}, part (\ref{item:Morse_polynomial})]
  We first establish equation~\eqref{eq:Morse_contrib_nonsmooth}.
  Using Lemma~\ref{oldpar2} as well as Lemma~\ref{lem:GrPolynomials}
  with
  \begin{equation*}
    k:=i-1
    \qquad\mbox{and}\qquad
    n:=\nu-1, 
  \end{equation*}
  \begin{enumerate}
  \item 
    The first line of \eqref{eq:Morse_contrib_nonsmooth} is obtained
    directly from the first line of \eqref{eq:Morse_contrib} and the
    first line of \eqref{eq:PoincareGrassmannianReal}.
  \item 
    The second line of \eqref{eq:Morse_contrib_nonsmooth} is obtained from
    the second line of \eqref{eq:Morse_contrib} by combining the second
    lines of \eqref{eq:PoincareGrassmannianReal} and
    \eqref{eq:PoincareGrassmannianOrientedReal}.
  \item 
    The third line of \eqref{eq:Morse_contrib_nonsmooth} is obtained from
    the second line of \eqref{eq:Morse_contrib} by combining the first
    lines of \eqref{eq:PoincareGrassmannianReal} and
    \eqref{eq:PoincareGrassmannianOrientedReal}.
  \item 
    Finally, the last line of \eqref{eq:Morse_contrib_nonsmooth} is obtained
    directly from the last line of \eqref{eq:Morse_contrib} and
    \eqref{eq:PoincareGrassmannianComplex}.
  \end{enumerate}



  Finally, Morse inequalities \eqref{eq:Morse_inequalities} when $M$
  is compact are established by \cite[\S~45] {FF89} using  the tools identical to \cite[\S~5]{Milnor_MorseTheory}.
\end{proof}

\begin{proof}[Proof of Corollary \ref{extremum_criteria}]
  A critical point is a point of local maximum if and only if the local
  Morse data is homotopy equivalent to $(\mathbb B^d, \partial \mathbb
  B^d)$.

  If $x$ is a maximum, its contribution to the Poincar\'e polynomial is
  equal to $t^d$, which occurs only in the cases described by
  Corollary \ref{extremum_criteria}.
  
  To establish sufficiency, we compute the local Morse data at $x$.
  If condition (\ref{item:extremum_branch}) is satisfied, the normal
  data at the point $x$ has been computed in
  Remark~\ref{rem:normal_data_bottom},
  $(J,K) = \left(\mathbb{B}^{s(\nu)}, \partial \mathbb{B}^{s(\nu)}
  \right)$.  From condition (\ref{item:extremum_smooth}) we get the
  tangential data
  \begin{equation}
    \label{eq:tang_data_max}
    (P,Q) = \left(\mathbb{B}^{d-s(\nu)}, \partial
      \mathbb{B}^{d-s(\nu)} \right).
  \end{equation}
  By \cite[Thm I.3.7]{GM88}, the local Morse data is then
  \begin{equation}
    \label{eq:local_Morse_max}
    \bigl(P\times J, (P\times K) \cup (Q\times J)\bigr)
    \cong
    \left(\mathbb{B}^{d}, \partial
      \mathbb{B}^{d} \right),
  \end{equation}
  implying the point is a maximum.

  Similarly, a critical point is a point of local minimum if and only
  if the local Morse data is homotopy equivalent to
  $(\mathbb B^d, \emptyset)$.  If $x$ is a minimum, its contribution
  to the Poincar\'e polynomial is equal to $1$, which occurs only in the
  cases described by Corollary \ref{extremum_criteria}.

  Conversely, condition (\ref{item:extremum_branch}) implies the
  normal data is
  \begin{equation}
    \label{eq:normal_data_top_again}
    (J,K) = \left(\mathbb{B}^{s(\nu)}, \emptyset \right),
  \end{equation}
  see Remark~\ref{rem:boundary_cases}.  From condition (\ref{item:extremum_smooth}), the tangential
  data is
  \begin{equation}
    \label{eq:tang_data_min}
    (P,Q) = \left(\mathbb{B}^{d-s(\nu)}, \emptyset \right).
  \end{equation}
  Combining these using \cite[Thm I.3.7]{GM88} gives the required
  result.
\end{proof}

\appendix

\section{Hellmann--Feynman Theorem}
\label{sec:hellmann_feynman}

In this section we review the mathematical formulation of the formula
that is known in physics as Hellmann--Feynman Theorem or first-order
perturbation theory.  We base our formulation on \cite[Thm II.5.4]{Kato}
(see also \cite{Gru_mn09}).

\begin{theorem}
  \label{thm:hellmann-feynman}
  Let $T: \R \to \Sym_n(\F)$ be differentiable at $x=0$.  Let
  $\lambda$ be an eigenvalue of $T(0)$ of multiplicity $\nu$,
  $\esp\subset \F^n$ be its eigenspace.
  Then, for small
  enough $x$, there are exactly $\nu$ eigenvalues of $T(x)$ close to
  $\lambda$ and they are given by
  \begin{equation}
    \label{eq:eigenvalue_derivative}
    \lambda_j(x) = \lambda + x \mu_j + o(x),
    \qquad
    j=1,\ldots,\nu,
  \end{equation}
  where $\{\mu_j\}$ are the eigenvalues of the $\nu \times \nu$ matrix
  $\big(T'(0)\big)_{\esp}$, see \eqref{eq:restriction_def}.
\end{theorem}

\section{Twisted Thom space homologies
  from Poincar\'e--Lefschetz duality}
\label{sec:thomhom_computation}

The Poincar\'e--Lefschetz duality (see, e.g.\ \cite[Theorem
3.43]{Hatcher}) states that if $Y$ is compact orientable
$n$-dimensional manifold with boundary $\partial Y$, then
\begin{equation}
  \label{Lefshetz1}
  H_r(Y, \partial Y)\cong H^{n-r}(Y),
  \quad
  H^r(Y,\partial Y)\cong H_{n-r}(Y).
\end{equation}

There is also a twisted analogue\footnote{See \cite[Prop
  15.2.10]{Geo_TopMethods} or \cite[chapter 5]{DK01}.  It is also
  sometimes known as Poincar\'e--Verdier duality, see
  \cite[VI.3]{Iversen}.} of Poincar\'e--Lefschetz duality for
non-oriented manifolds: if $Y$ is compact non-orientable
$n$-dimensional manifold with boundary $\partial Y$, then
\begin{equation}
  \label{Lefshetz2}
  H_r(Y, \partial Y) \cong H^{n-r}(Y; \widetilde{\Z}),
  \quad
  H^r(Y, \partial Y)\cong H_{n-r}(Y; \widetilde{\Z}).
\end{equation}
Here, the twisted homology $H_*(Y;\widetilde{\Z})$ was already
introduced in Section~\ref{sec:sublevel_change2}.  To define twisted
cohomology groups, denote by $\widetilde Y$ the orientation cover of
$Y$ and by $\tau$ the corresponding orientation-reversing involution.
$H^*(Y;\widetilde{\Z})$ are the cohomologies of the cochain complex
defined on the spaces of cochains $c$ satisfying
$c\bigl(\tau(\alpha)\bigr)=-c(\alpha)$ for every chain $\alpha$ in
$\widetilde Y$ (see \cite[Se. 3H]{Hatcher} for a more general point of
view).  Such cochains will be called \term{skew-symmetric cochains}.
Note that the space of skew-symmetric cochains can be identified with
the dual space to the space of skew-symmetric chains, as expected.

{\bf (a)} Assume now that $\nu$ is even.  Then the base
$\Gr(i-1,\nu-1)$ of the vector bundle $\widehat E_{i,\nu}$ is
non-orientable and, since the vector bundle is also non-orientable,
the total space $\B(\widehat E_{i,\nu})$ is orientable.  By the
usual Poincar\'e--Lefschetz duality \eqref{Lefshetz1},
\begin{equation}
  \label{case1_1}
  H_r\bigl(\B(\widehat E_{i,\nu}),
  \partial \B(\widehat E_{i,\nu})\bigr)
  \cong H^{\dim \widehat  E_{i, \nu}-r}
  \bigl(\B(\widehat E_{i,\nu})\bigr)
  \cong H^{\dim \widehat E_{i, \nu}-r}
  \bigl(\Gr_\R(i-1,\nu-1)\bigr),
\end{equation}
where $\dim \widehat E_{i, \nu}$ is the dimension of the total space
$\B(\widehat E_{i,\nu})$.  In the last identification we used that
the base $\mathrm{Gr}(i-1,\nu-1)$ is the deformation retract of the
total space of the bundle.

Further, since $\Gr(i-1,\nu-1)$ is non-orientable when $\nu$ is even,
we use the twisted analog of Poincar\'e duality for nonorientable
manifolds (see \cite[Theorem 3H.6]{Hatcher} as well as \eqref{Lefshetz2} with
$\partial Y=\emptyset$) to get
\begin{equation} 
  \label{case1_2}
  H^{\dim \widehat E_{i, \nu}-r}\big(\Gr_\R(i-1,\nu-1)\big)
  \cong H_{r-s(i)}\big(\Gr_\R(i-1,\nu-1);\widetilde\Z),
\end{equation}
where we used 
\begin{equation}
  \label{eq:dimension_differences_1}
  \dim \widehat E_{i, \nu}-\dim\Gr_\R(i-1,\nu-1)
  = \rank \widehat E_{i, \nu}
  = s(i).
\end{equation}

{\bf (b)} Consider the case of odd $\nu$. Then the base
$\Gr(i-1,\nu-1)$ is orientable, the bundle is non-orientable and
therefore the total space $\B(\widehat E_{i,\nu})$ is
non-orientable.  By the twisted
Poincar\'e-Lefschetz duality \eqref{Lefshetz2},
\begin{equation}
  \label{case2_1}
  H_r\bigl(\B(\widehat E_{i,\nu}),
  \partial \B(\widehat E_{i,\nu})\bigr)
  \cong H^{\dim \widehat E_{i, \nu}-r}
  \bigl(\B(\widehat E_{i,\nu}), \widetilde{\Z}\bigr)
\end{equation}

The orientation double cover $\doublehat E_{i,\nu}$ of
$\widehat E_{i,\nu}$ can be constructed from the tautological bundle
of the oriented Grassmannian $\widetilde{\Gr}_\R(i-1,\nu-1)$ in the
same way as $\widehat E_{i,\nu}$ was constructed from the tautological
bundle of the Grassmannian $\Gr_\R(i-1,\nu-1)$ by relations
\eqref{explicitE} and \eqref{Efin}.  In particular,
$\doublehat E_{i,\nu}$ is a bundle of rank $s(i)$ over the oriented
Grassmannian $\widetilde{\Gr}_\R(i-1,\nu-1)$.  Therefore, retracting
the unit ball bundle $\B(\doublehat E_{i,\nu})$ of
$\doublehat E_{i,\nu}$ to its base, we get that the integer cohomology
groups of $\B(\doublehat E_{i,\nu})$ are isomorphic to the integer
cohomology groups of the oriented Grassmannian
$\widetilde{\Gr}_\R(i-1,\nu-1)$, i.e.
\begin{equation*}
  H^{\dim \widehat E_{i, \nu}-r}\bigl(\B(\doublehat E_{i,\nu})\bigr)
  \cong H^{\dim \widehat E_{i, \nu}-r}(\widetilde{\Gr}_\R(i-1,\nu-1)).
\end{equation*}
Moreover, the retraction can be made to preserve the spaces of
skew-symmetric chains, which implies that
\begin{equation}
  \label{eq:cease2_2}
  H^{\dim \widehat E_{i, \nu}-r}\bigl(\B(\widehat E_{i,\nu}); \widetilde{\Z}\bigr)
  \cong H^{\dim \widehat E_{i, \nu}-r}(\Gr_\R(i-1,\nu-1); \widetilde \Z ).
\end{equation}

When $\nu$ is odd, $\Gr_\R(i-1,\nu-1)$ is orientable and so is
$\widetilde{\Gr}_\R(i-1,\nu-1)$.  Moreover, the map from the usual
Poincar\'e duality (see \cite [Thm.~3.30]{Hatcher} as well as
\eqref{Lefshetz1} with $\partial T = \emptyset$) applied to
$\widetilde{\Gr}_\R(i-1,\nu-1)$ sends the equivalence classes of
skew-symmetric cochains to the corresponding skew-symmetric chains.
Thus, we arrive to
\begin{equation}
  \label{case2_3}
  H^{\dim \widehat E_{i, nu}-r}(\Gr_\R(i-1,\nu-1); \widetilde \Z)
  \cong H_{r-s(i)}(\Gr_\R(i-1,\nu-1); \widetilde \Z).
\end{equation}

To summarize, we get the corresponding line in (\ref{eq:relhom_main})
whether $\nu$ is even or odd.

\bibliographystyle{myalpha}
\bibliography{eigmorse}

\end{document}